\documentclass[12pt,twoside,reqno]{amsart}

\usepackage[OT1]{fontenc}
\usepackage{type1cm}

\usepackage{comment}
\usepackage{enumerate}

\usepackage{amsthm}
\RequirePackage{amsmath,amsfonts,amssymb,amsthm}

\RequirePackage[dvips]{graphicx}

\usepackage{psfrag}
\usepackage[usenames]{color}
  \numberwithin{equation}{section}

\usepackage[active]{srcltx}  

\usepackage{booktabs}

\usepackage{xcolor}
\usepackage[pagebackref]{hyperref}
\hypersetup{
    colorlinks,
    linkcolor={red!60!black},
    citecolor={blue!60!black},
   urlcolor={blue!90!black}
}

  \newcommand{\N}{\mathbb{N}}         
  \newcommand{\R}{\mathbb{R}}         

  \newcommand{\EE}{\mathbb{E}}        
  \newcommand{\PP}{\mathbb{P}}        
  \newcommand{\leb}{\mathcal{L}}      
  \newcommand{\supp}{\text{supp}}        
  \newcommand{\BB}{\mathcal{B}}         
  \newcommand{\diam}{\text{diam}}       
  
  \newcommand{\dimloc}{\text{dim}}

  \newcommand{\e}{\varepsilon}

  \renewcommand{\AA}{\mathbb{A}}          
  \newcommand{\Q}{\mathcal{Q}}          
  \newcommand{\lam}{\lambda}

  \newcommand{\simi}{\text{SIM}}

\newcommand{\cP}{\mathcal{P}}
\newcommand{\zero}{\mathcal{Z}}
\newcommand{\wt}{\widetilde}

  \newcommand{\I}{\text{I}}
  \newcommand{\II}{\text{II}}
  \newcommand{\III}{\text{III}}

  \DeclareMathOperator{\dist}{dist}

  \newtheorem{thm}{Theorem}[section]
  \newtheorem{lemma}[thm]{Lemma}
  \newtheorem{prop}[thm]{Proposition}
  \newtheorem{cor}[thm]{Corollary}

  \theoremstyle{remark}
  \newtheorem{rem}[thm]{Remark}
  \newtheorem{rems}[thm]{Remarks}
  
  \newtheorem{defn}[thm]{Definition}

  \DeclareMathOperator{\spt}{spt}

\begin{document}

\title{Patterns in random fractals}

\author{Pablo Shmerkin}
\address{Department of Mathematics and Statistics, Torcuato Di Tella University, and CONICET, Buenos Aires, Argentina}
\email{pshmerkin@utdt.edu}
\thanks{P.S. was partially supported by Projects PICT 2013-1393 and PICT PICT 2014-1480 (ANPCyT). Part of this research was completed while P.S. was visiting the University of Oulu.}
\urladdr{http://www.utdt.edu/profesores/pshmerkin}

\author{Ville Suomala}
\address{Department of Mathematical Sciences, University of Oulu, Finland}
\email{ville.suomala@oulu.fi}
\thanks{V.S. was partially supported by the Academy of Finland via the Centre of Excellence in Analysis and Dynamics Research.}
\urladdr{http://cc.oulu.fi/~vsuomala/}



\begin{abstract}
We characterize the existence of certain geometric configurations in the fractal percolation limit set $A$ in terms of the almost sure dimension of $A$. Some examples of the configurations we study are: homothetic copies of finite sets, angles, distances, and volumes of simplices. In the spirit of relative Szemer\'{e}di theorems for random discrete sets, we also consider the corresponding problem for sets of positive $\nu$-measure, where $\nu$ is the natural measure on $A$. In both cases we identify the dimension threshold for each class of configurations. These results are obtained by investigating the intersections of the products of $m$ independent realizations of $A$ with transversal planes and, more generally, algebraic varieties, and extend some well known features of independent percolation on trees to a setting with long-range dependencies.
\end{abstract}

\maketitle

\tableofcontents

\section{Introduction and summary of main results}

\subsection{Introduction}
A classical general problem in combinatorics is to understand what conditions (especially, conditions of \emph{structure} and \emph{size}) on a set $A$ imply that $A$ contains certain configurations, like $3$-term arithmetic progressions.  Indeed, the classical theorem of Roth \cite{Roth53} implies that if $A\subset\N$ has positive upper density, then it must contain $3$-term arithmetic progressions. The famous theorem of Szem\'{e}rdi \cite{Szemeredi75} generalizes this to arbitrarily long arithmetic progressions. The celebrated theorem of Green and Tao \cite{Greentao08} generalizes the former statement to subsets of the primes, and has stimulated a large amount of further research over the past decade.

There has been much interest also in this kind of problems when $A$ is a subset of Euclidean space. One might heuristically conjecture that if $A\subset\R$ is `large', then it should contain progressions. If $A$ is large in the sense of measure, then this is indeed the case. A well known corollary of the Lebesgue density theorem asserts that any set $A\subset\R$ with positive Lebesgue measure contains arbitrary long arithmetic progressions and, more generally, in any dimension, homothetic copies of all finite sets. A wide open conjecture of Erd\H{o}s states that for all infinite sets $S\subset\R$, there is a set $A\subset\R$ with positive measure which does \emph{not} contain a similar copy of $S$.

In the zero-measure case, a natural candidate for size is Hausdorff dimension. Already in 1959, Davies, Marstrand, and Taylor \cite{Daviesetal59} showed that there are compact sets of dimension $0$ which contain a homothetic copy of all finite sets. This was recently extended to polynomial patterns in \cite{MolterYavicoli16}. Hence, a small dimension itself does not rule out containing rich sets of configurations. On the other hand, there are compact sets $A\subset \R$ of Hausdorff dimension $1$ without any arithmetic progressions (see \cite{Keleti88}) and compact sets in arbitrary dimension that do not contain parallelograms \cite{Maga10}. Thus, Hausdorff dimension of $A\subset\R$ alone gives no information whatsoever about the existence of arithmetic progressions (and certain other configurations) in $A$. We note, however, that the situation is different for other patterns. For example, Iosevich and Liu \cite{IosevichLiu16} recently proved that for $d\ge 4$ there exists $\e_d>0$ such that any Borel subset $A$ of $\R^d$ of dimension $>d-\e_d$ contains the vertices of an equilateral triangle. This was known to be false in dimension $d=2$ (\cite{Falconer92, Maga10}) and the problem is open in dimension $d=3$.

It turns out that in many cases a lower bound on the Hausdorff dimension does imply the presence of a \emph{positive measure} of configurations in some class (even if it often does not guarantee the existence of any one given configuration). Perhaps the most classical example is the distance set problem: Falconer \cite{Falconer85} conjectured that if $A\subset\R^d$, $d\ge 2$ is a Borel set with $\dim_H(A)>\tfrac{d}{2}$, then its distance set $D(A)=\{|x-y|:x,y\in A\}$ has positive Lebesgue measure. The best current results towards this conjecture are due to Wolff \cite{Wolff99} and Erdogan \cite{Erdogan05}: $\dim_H(A)>\tfrac{d}{2}+\tfrac{1}{3}$ suffices. Many other problems of a similar kind have been investigated, see e.g. \cite{GILP15, GGIP15} and references there.

A fruitful parallel line of work has focused on finding \emph{pseudo-randomness} conditions (in addition to size conditions) on subsets of $\mathbb{R}^d$ that ensure the presence of configurations (such as arithmetic progressions). Typically, these pseudo-randomness conditions take the form of a suitable Fourier decay of a measure supported on the set in question: see \cite{LabaPramanik09, Carnovale15, Chanetal16, Henriotetal16}. Recall that the Fourier transform of a finite measure $\mu$ on $\R^d$ is given by $\hat{\mu}(\xi)=\int\exp(-2\pi i \xi \cdot x)\,d\mu(x)$. To a give a flavour for this type of results, we state a very special case of \cite[Theorem 1.3]{Henriotetal16}: given $D>0, \beta\in (0,1)$, there is $\e(D,\beta)>0$ such that the following holds: if $\mu$ is a measure on $\R$ such that
\begin{align*}
&\mu([x-r,x+r])\le D r^\alpha\text{ for all }r>0\,,x\in\R\,,\\
&\hat{\mu}(k)\le D (1-\alpha)^{-B}|k|^{-\beta/2}\text{ for all }k\in\N\,,
\end{align*}
then, provided that $\alpha>1-\e(D,\beta)$, the topological support of $\mu$ contains a $3$-term arithmetic progression. In fact, the results from \cite{Henriotetal16} apply to many linear, and some polynomial, patterns in $\R^d$. An interesting feature of \cite{Carnovale15} is that the progressions can be found in any set of positive $\mu$-measure, not just the topological support. Although these are deep results, we note that the hypotheses are difficult to verify for concrete examples, in part because the bound on the mass decay exponent $\alpha$ depends on the constant $D$ (it was shown in \cite{Shmerkin16} that such dependence cannot be removed). The examples showing that measures satisfying conditions such as the above exist are random, see \cite{LabaPramanik09}. Moreover, the value of $\e(D,\beta)$ is not explicit, and certainly far from sharp.

Going back to the discrete setting, the last few years saw an explosion of \emph{relative} Szemer\'{e}di Theorems. That is, given some discrete set $A$, one is interested in knowing whether sets of positive density \emph{relative to $A$} contain large arithmetic progressions. The Green-Tao Theorem mentioned above is of this type, with $A$ equal to the prime numbers. A general approach to relative Szemer\'{e}di theorems (which in particular yields a simpler proof of the Green-Tao Theorem) was developed in \cite{ConlonEtAl15}. Closer to our work, sharp relative Szemer\'{e}di theorems have been obtained for random discrete sets by Conlon and Gowers \cite{ConlonGowers16} and, independently, by Schacht \cite{Schacht16}: for $\delta>0$, $k\in\N_{\ge 3}$, let us say that a set $A\subset\{1,\ldots,n\}$ is $(\delta,k)$-Szemer\'{e}di if every subset $A'\subset A$ with $|A'|\ge \delta|A|$ contains an arithmetic progression of length $k$. If $[n]_p$ denotes the canonical random set obtained by keeping each number in $\{1,\ldots,n\}$ independently with probability $p$ then, provided $p\ge C n^{-1/(k-1)}$ (with $C$ a suitably large absolute constant), the probability that $[n]_p$ is $(\delta,k)$-Szemer\'{e}di tends to $1$ as $n\to\infty$. Moreover, this threshold is sharp (up to the value of $C$). We note that the threshold for the existence of $k$-progressions in $A$ itself is $p\sim n^{-2/k}$, and this is a far more elementary fact.

\subsection{Summary of results}
\label{subsec:sum-res}

This circle of results show that, despite the substantial progress achieved, the connection between size, pseudo-randomness and the presence of progressions and other patterns, is far from being elucidated, especially in the continuous setting. The goal of this work is to present a systematic study of the existence of patterns in random fractals. That is, rather than dealing with pseudorandomness (such as fast Fourier decay), we will consider `honest' random sets and measures. This will also give us the chance to explore `relative Szemer\'{e}di' type of results in our setting.

Unlike the discrete case, there is no canonical random set or measure of fractional dimension. In \cite{ShmerkinSuomala14}, we proposed a large class of random fractal measures on Euclidean space which aims to capture the main properties of the canonical discrete random set. For concreteness, in this article we focus on what is perhaps the best known and studied model of stochastically self-similar set: fractal percolation. Nevertheless, the method should work for many other random fractals, including far more general subdivision constructions and Poissonian cutouts. In fact, our main abstract result in Section \ref{sec:continuity} holds for a wide variety of random measures satisfying suitable  martingale and weak dependency conditions.

In order to state some of our results more precisely, let us recall the definition of fractal percolation. For convenience of notation, we will consider fractal percolation on the dyadic grid only. Fix a parameter $p\in (0,1)$ and $d\in\N$. We subdivide the unit cube in $\R^d$ into $2^d$ equal sub-cubes. We retain each of them with probability $p$ and discard it with probability $1-p$, with all the choices independent. For each of the retained cubes, we continue inductively in the same fashion,  further subdividing them into $2^d$ equal sub-cubes, retaining them with probability $p$ and discarding them otherwise, with all the choices independent of each other and the previous steps. The fractal percolation limit set $A=A^{\text{perc}(d,p)}$ is the set of points which are kept at each stage of the construction.
It is well known that if $p\le 2^{-d}$, then $A$ is almost surely empty, and otherwise a.s.
\begin{equation}\label{eq:s(d,p)}
\dim_H A=\dim_B A=s(d,p):=d+\log_2 p
\end{equation}
conditioned on non-extinction (i.e. $A\neq\varnothing$). Here, and throughout the paper, $\dim_H,\dim_B$ denote Hausdorff and box-counting (Minkowski) dimensions, respectively. Fractal percolation can be seen as a Euclidean realization of a Galton-Watson branching process. See \cite{LyonsPeres17} for extensive background on branching processes, fractal percolation and dimension.

Our first class of results identify the dimension threshold $s(d,p)$ for the presence of a wide variety of geometric configurations in $A$:

\begin{thm} \label{thm:patterns}
The following hold for $A=A^{\text{perc}(d,p)}$, provided the required conditions on $d$ and $s=s(d,p)$ hold:
\begin{enumerate}
\item If $m\ge 2$ and $s>d-(d+1)/m$, then $A$ contains a homothetic copy of all $m$-point sets.
\item If $m\ge 2$ and $s>d-d/m$, then for any subset $\{x_1,\ldots,x_m\}\subset ]0,1]^d$, $A$ contains a translation of $\{ x'_1,\ldots,x'_m\}$ whenever $x'_i$ are close enough to $x_i$.
\item If $d\ge 2$ and $s>1/2$, then there is $\e>0$ such that $(0,\e)\subset D(A)$.
\item If $s>1/(d+1)$, then there is $\e>0$ such that $A$ contains the vertices of a simplex of all volumes in $(0,\e)$.
\item If $d=2$ and $s>1$, then for $\{ x_1,x_2,x_3\}\subset ]0,1[^2$, $A$ contains an isometric copy of $\{ x'_1,x'_2,x'_3\}$ whenever $x'_i$ are close enough to $x_i$.
\item If $d\ge 2$ and $s>1/3$, then triples of points in $A$ determine all angles in $]0,\pi[$.
\item If $d\ge 2$ and $s>2/3$, then $A$ contains the vertices of all non-degenerate triangles, up to similarities.
\item If $m\ge 3$, $d=2$ and $s>2-4/m$, then up to similarities $A$ contains the vertices of all non-degenerate $m$-gons.
\end{enumerate}
To be more precise, the claims $(2)$ and $(5)$ hold with positive probability and the others hold a.s. on $A\neq\varnothing$.
Moreover, in all these cases, the range of $s$ is sharp, in the sense that if $s$ is smaller or equal than the given threshold, then any one given configuration has probability zero of occurring in $A$. For example, for any $m$-point set $S\subset\R^d$, if $s\le d-(d+1)/m$, then a.s. $A$ contains no similar copy of $S$.

Furthermore, the thresholds are sharp for packing dimension (up to the endpoint), even for deterministic sets. For example, if $A\subset\R^d$ contains a homothetic copy of all $m$-point sets, then $\dim_P(A)\ge d-(d+1)/m$.
\end{thm}

This theorem will be proved in the course of Sections \ref{sec:affine-intersections}--\ref{sec:nonlinear}. For now, we make some general remarks:
\begin{rems}
\begin{enumerate}[(i)]
\item Proving the existence of a single configuration is already more challenging than in the random discrete setting, although in general it can be done via the second moment method (see e.g. Lemma \ref{lem:Yt-survives}). However, the main challenge is proving the existence of open sets/all configurations simultaneously, which is an issue that obviously does not arise in the discrete world.
\item All the configurations arising in Theorem \ref{thm:dim_patterns} can be realized as the zero set of a suitable polynomial, and the dimension thresholds are derived from a general statement about intersections (of the Cartesian powers of $A$) with algebraic varieties, see Corollary \ref{cor:nonlinear}.
\item The statement about the distance set of $A$ was proved, in a slightly weaker form, by Rams and Simon \cite{RamsSimon14}. Although we use some of their ideas (as we did already in our paper \cite{ShmerkinSuomala14}), there are substantial differences that allow us to get stronger results, including the `relative Szemer\'{e}di' version discussed below.
\end{enumerate}
\end{rems}

One can visualize Theorem \ref{thm:patterns} by considering the following joint construction of all fractal percolation processes. Let $(U_Q)$ be a sequence of IID random variables, uniform in $[0,1]$, where $Q$ ranges over all dyadic cubes of all levels, starting with the unit cube. Given any $p$, we can construct a set $A_p$ by retaining cubes $Q$ for which $U_Q \le p$, and discarding those with $U_Q>p$. In this way we get an increasing ensemble $(A_p)_{p\in [0,1]}$, where $A_p$ has the distribution of $A^{\text{perc}(d,p)}$. Theorem \ref{thm:patterns} then shows that almost surely the sets $A_p$ undergo a phase transition for the presence of geometric configurations at the corresponding critical value of $p$. For example, given a fixed $m$-element set $S$ in $\R^d$,  $A_p$ contains no homothetic copy of $S$ as long as $\log_2 (1/p)\ge\tfrac{d+1}{m}$, but as soon as $\log_2 (1/p)<\tfrac{d+1}{m}$, the set $A_p$ transitions to containing a homothetic copy not just of $S$ but of all $m$-point configurations.

We are able to sharpen Theorem \ref{thm:patterns} as follows: for each class of configurations, if $s$ is above the given threshold, not only we get that $A$ contains all/an open set of configurations, but we can precisely measure how often each configuration arises. We give only one example here, deferring  further discussion to Section \ref{sec:szemeredi}.
\begin{thm} \label{thm:dim-patterns-homothetic}
Let $A=A^{\text{perc}(d,p)}$. If $s=s(d,p)>d-\tfrac{d+1}{m}$, then almost surely on $A\neq\varnothing$, for each $m$-point $S\subset\R^d$,
\[
\dim_H( \{ (a,b)\in (0,\infty)\times \R^d: aS+b\subset A\}) = m(s - d) +d+1.
\]
\end{thm}

There is a natural random measure supported on the fractal percolation limit set (this is sometimes called `branching measure' in the context of Galton-Watson trees). This is obtained as the weak-* limit of the measures $\nu_n:=p^{-n} \leb^d|_{A_n}$, where $A_n$ is the union of the surviving cubes of side-length $2^{-n}$, and $\leb^d$ is $d$-dimensional Lebesgue measure (See Section \ref{sec:perco} for more details). Then a.s. $\nu_n$ converges weakly to a limit $\nu=\nu^{\text{perc}(d,p)}$. Moreover, if $p>2^{-d}$, then $\nu\neq 0$ a.s. on $A\neq\varnothing$,  and in this case the Hausdorff dimension of $\nu$ equals $s(d,p)$ (that is, if $A'$ is a Borel set with $\nu(A')>0$, then $\dim_H(A')\ge s(d,p)$).

Positive $\nu$-measure is then a natural analogue of `positive relative density' in the discrete random case, and this gives us a way to investigate relative Szemer\'{e}di phenomena for fractal percolation:
\begin{thm}\label{thm:relative_S}
Let $\nu=\nu^{\text{perc}(d,p)}$. Almost surely, the following holds for each Borel set $A'$ such that $\nu(A')>0$ under the given conditions on $d$ and $s=s(d,p)$:
\begin{enumerate}
\item\label{eq:rs_patterns} If $m\ge 2$ and $s>d-\tfrac{1}{m-1}$, then $A'$ contains a homothetic copy of all $m$ point sets.
\item If $s>1$, then the distance set of $A'$
has non-empty interior.
\item If $s>\tfrac{1}{d}$, then the set of volumes of simplices with vertices in $A'$ has non-empty interior.
\item If $d=2$ and $s>\tfrac{3}{2}$, then there is an open set of
triples $\{a_1,a_2,a_3\}$ such that $A'$ contains an isometric image of $\{a_1,a_2,a_3\}$.
\item If $d\ge 2$ and $s>\tfrac12$, then $A'$ contains all angles in $]0,\pi[$.
\item If $d\ge 2$ and $s>1$, then $A'$ contains a similar copy of all non-degenerate triangles.
\item If $m\ge 3$, $d=2$ and $s>2-\tfrac{2}{m-1}$, then $A'$ contains a similar copy of all non-degenerate $m$-gons.

\end{enumerate}
Moreover, these thresholds are sharp, in the sense that for any countable set of configurations in each class, if $s$ is smaller or equal than the given threshold, then a.s. there is a Borel set $A'$ of full $\nu$-measure which does not contain any configuration in the countable set. For example, if $s\le 1$, then there is a full $\nu$-measure set $A'$ which does not contain any rational distances.
\end{thm}

It is interesting to compare the different thresholds with what is known or conjectured for deterministic sets. For example, we pointed out earlier that sets of full Hausdorff dimension in the line may fail to contain three-term progressions, and sets of full dimension in the plane may fail to contain equilateral triangles. Thus, (1) and (6) are very far from holding for general sets of the given dimensions. On the other hand, the distance set conjecture in the plane (but not in higher dimensions) almost gives (2) for \emph{any} set $A'$ of dimension $>1$.

Despite the formal analogy, the proof of this theorem does not use any of the methods of \cite{ConlonGowers16, Schacht16}. In our setting, the stochastic self-similarity of $\nu$ and the density point theorem will allow us weaken the statement from `full probability' to `positive probability' and from `positive measure' to `full measure'.  On the other hand, as is already the case for Theorem \ref{thm:patterns}, we have to deal with uncountable families of configurations.

\subsection{General strategy}
We will obtain all the aforementioned results by investigating the intersections of the Cartesian products $A^m\subset \R^{md}$ with families of affine subspaces, and more general algebraic varieties. For instance, to show that $A\subset\R$ contains similar copies of all triples $\{t_1,t_2,t_3\}\subset\R$, we have to show that $A\times A\times A$ intersects the $2$-dimensional plane
\[
V_{t}=\{(x,x,x)+\lambda(t_1,t_2,t_3)\,:\,x,\lambda\in\R\}\subset\R^3,
\]
outside of the diagonal $\{x=y=z\}$, for all choices of $t=(t_1,t_2,t_3)$. This will be verified by considering the intersections or `slices' of $\nu\times\nu\times\nu$ with the planes $V_{t}$, and showing that the total mass of these intersections is bounded away from zero. This will be achieved by showing that the total mass of a slice behaves in a continuous way (as a function of $t$). This continuity, in turn, will be derived as a consequence of a general intersection result for weakly dependent martingales (to be defined in Section \ref{sec:continuity} below). The main abstract result (Theorem \ref{thm:Holder-continuity}) yields H\"older continuity for the map $t\mapsto Y^t$, where $Y^t$ is the total mass of the intersection of $\mu$ and $\eta_t$, where $\{\eta_t\}_{t\in\Gamma}$ is a suitable deterministic family of measures on $\R^d$ (or rather $(\R^d)^m$ in our applications), parametrized by a metric space $\Gamma$, and $\mu$ is a random limit measure of absolutely continuous measures $\mu_n$ satisfying certain size and `weak spatial dependence' assumptions. This setup generalizes the spatially independent martingales from \cite{ShmerkinSuomala14}; in particular, Theorem \ref{thm:Holder-continuity} extends the main result of \cite{ShmerkinSuomala14}. We hope this more general framework will find further applications beyond those explored in this article. In a future work, we hope to relax also the martingale condition, and derive further applications to self-convolutions of $\nu$ and certain maximal operators.

To give an idea of the method, we discuss the proof of  the existence of $3$-patterns for fractal percolation sets $A\subset[0,1]$. See also the survey \cite{ShmerkinSuomala16} for a complete proof of this particular case. In this case, the natural measure on $A^3$ is the weak limit of $\mu_n=\nu_n\times\nu_n\times\nu_n$, recall that $\nu_n=p^{-n}\mathcal{L}|_{A_n}$. Given $t=(t_1,t_2,t_3)\in\R^3$, we consider the total mass of the intersection of $\mu_n$ and $\mathcal{H}^2|_{V_{t}}$ defined as
\[Y_n^{t}=\int_{V_{t}}\mu_n(x)\,d\mathcal{H}^2(x)\,.\]
The increments $Y_{n+1}^{t}-Y_n^{t}$ may be expressed as sums of the random variables
$X_Q=\int_{Q\cap V_{t}}\mu_{n+1}(x)-\mu_{n}(x)\,d\mathcal{H}^2(x)$ where $Q$ runs over all dyadic cubes of side-length $2^{-n}$. If these random variables $\{X_Q\}$ were independent, and if $\mu_n$ satisfied the martingale condition $\EE(\mu_{n+1}(x)\,|\,A_n)=\mu_n(x)$ for all $x\in\R$, $n\in\N$, we could apply a large deviation estimate to show that $Y_n^{t}$ converges very rapidly, and derive a continuity modulus for the limit with respect to $t$ (this is the strategy of \cite{ShmerkinSuomala14}, which in turn was inspired in \cite{PeresRams16}). These assumptions would be satisfied, for instance, if $\mu_n$ was the fractal percolation measure on $[0,1]^3$, instead of the $3$-fold self-product fractal percolation on $[0,1]$. In fact, such a $\mu_n$ would be a model example of the SI-martingales considered in \cite{ShmerkinSuomala14} and would allow us to conclude that the limits $Y^{t}=\lim_{n\to\infty}Y^{t}_n$ are a.s. H\"older continuous in $t$, provided that the dimension of the limit measure $\mu$ is larger than $1$. In the present situation, however, both the martingale condition and the spatial independence condition fail. For instance, if $Q,Q'$ are two dyadic cubes with the same $x$-coordinate, then $X_Q$ and $X_{Q'}$ are clearly dependent. A priori, there can be many such dependencies, since the planes $V_{t}$ intersect the hyperplanes $\{x=c\}$ in a line (and there could be many surviving cubes along this line). The martingale condition, on the other hand, breaks down at the dyadic cubes meeting one of the diagonals $\{x=y\}$, $\{x=z\}$ or $\{y=z\}$. It turns out that the amount of dependencies can be inductively bounded by looking at the slices of the lower dimensional product $\nu_n\times\nu_n$ with `transversal' lines. These bounds make the dependencies sparse enough that a large deviation estimate for $Y_{n+1}^{t}-Y_n^{t}$ can still be derived, so that the continuity in $t$ can then be established along the lines of \cite{ShmerkinSuomala14}, provided $\dim\mu>1$ (which is equivalent to $\dim\nu=\dim A>1/3$).  Meanwhile, since in order to find non-degenerate patterns we want to avoid the diagonals, we will be able to work with the product $\wt{\mu}_n=\nu_n^{(1)}\times \nu_n^{(2)}\times \nu_n^{(3)}$ of three independent realizations, instead of $\mu_n$. Now $\wt{\mu}_n$ is easily checked to be a martingale, although it has the same dependency issues as before. This strategy is formalized in the general Theorem \ref{thm:Holder-continuity} below.

\section{Notation}
\label{sec:notation}

We will use Landau's $O(\cdot)$ and related notation. If $n>0$ is a variable, by $g(n) = O(f(n))$ we mean that there exists $C>0$ such that $0\le g(n)\le C f(n)$ for all $n$. By $g(n)=\Omega(f(n))$ we mean $f(n)=O(g(n))$.
Occasionally we will want to emphasize the dependence of the constants implicit in the $O(\cdot)$ notation on other, previously defined, constants; the latter will be then added as subscripts. For example, $g(n) = O_\delta(f(n))$ means that $0\le g(n)\le C_\delta f(n)$ for some constant $C_\delta$ which is allowed to depend on $\delta$.

The notation $B(x,r)$ stands for the closed ball with centre $x$ and radius $r$. Open balls will be denoted by $B^\circ(x,r)$. 
We will write $E(\e)$ for the open $\e$-neighbourhood $\{ x\in\R^M: \dist(x,E)<\e\}$. Moreover, $E^\circ$ and $\overline{E}$ denote the interior and closure of $E$, respectively. Given $L\in\N$, we let $[L]=\{1,\ldots,L\}$. We will denote by $|\cdot|$ both the absolute value $|x|$ of an element of $\R^M$, as well as the cardinality $|I|$ of a (finite) set $I$.

By a measure we always mean a locally finite Borel-regular outer measure on a metric space. 
Given a measure $\mu$ on $\R^M$,
we denote $\|\mu\|=\mu(\R^M)$. We will denote by $\PP$ the law of the fractal percolation as well as various other probability measures (it should be always clear from the context what probability measure we are referring to). In general, we will denote $\PP$-measurable events by $\mathfrak{C}$, $\mathfrak{F}$, etc.

We denote by $\mathcal{Q}_n$ (or $\mathcal{Q}^{M}_n$) the family of dyadic cubes of $\R^M$ with side length $2^{-n}$, and by $\mathcal{Q}$ the union $\cup_{n\in\N}\mathcal{Q}_n$. It will be convenient that these are pairwise disjoint, so we consider a suitable half-open dyadic filtration.

As noted earlier, $\dim_H$ denotes Hausdorff dimension. We denote upper box-counting (or Minkowski) dimension by $\overline{\dim}_B$, and box-counting dimension (when it exists) by $\dim_B$, while packing dimension is denoted by $\dim_P$. A good introduction to fractal dimensions can be found in \cite[Chapters 2 and 3]{Falconer03}.

The Grassmann manifold of $k$-dimensional linear subspaces of $\R^M$ will be denoted $\mathbb{G}_{M,k}$. It is a compact manifold of dimension $k(M-k)$, and its metric is
\[
d(V,W)=\|\pi_V-\pi_W\|\,,
\]
where $\pi_{(\cdot)}$ denotes orthogonal projection. The manifold of $k$-dimensional affine subspaces of $\R^M$ will be denoted $\AA_{M,k}$. It is diffeomorphic to $\mathbb{G}_{M,k} \times \R^{M-k}$, and this identification defines a natural metric. The metrics on all these different spaces will be denoted by $d$; the ambient space will always be clear from context (also note that the ambient dimension is sometimes denoted by the same symbol $d$).

\begin{table}
\begin{tabular}[h]{|l|l|}
\hline
$\leb$ & the Lebesgue measure\\
$\dim_H$, $\dim_P$, $\dim_B$ & Hausdorff, packing, and box-counting dimensions\\
$(\mu_n)$,\,\, $\mu$ & a sequence of random measures and its (weak-*) limit\\
$\|\mu\|$ & the total mass of $\mu$\\
$A^{\text{perc}}$, $\nu^{\text{perc}}$  & fractal percolation set and the natural measure\\
$A_n$, $\nu_n$ & level $n$ approximations of $A=A^{\text{perc}}$ and $\nu=\nu^\text{perc}$\\
$N_n$ & the total number of cubes forming $A_n$\\
$\PP$ & the law of $A_n$, $\nu_n$ (or of some other random sequence $\mu_n$)\\
$\mathfrak{C}$, $\mathfrak{F}$ & $\PP$-measurable events\\
$\mathcal{DI}_n$ & dependency degree\\
$\{\eta_t\,:\,t\in\Gamma\}$ & parametrized family of (deterministic) measures\\
$\mu_{n}^t$ & the `intersection' of $\mu_n$ and $\eta_t$\\
$Y_{n}^t$,\,\, $Y^t$ & the total mass of $\mu_{n}^t$, and its limit\\
$\mathcal{Q}$, $\mathcal{Q}^M$,\,\,  & the family of half-open dyadic cubes of $\R^M$\\
$\mathcal{Q}_n$, $\mathcal{Q}^M_n$ & and the ones with side-length $\ell(Q)=2^{-n}$\\
$\simi_M$ & the family of non-singular similarities on $\R^M$\\
$\mathbb{G}_{M,k}$ & the manifold of $k$-dimensional linear subspaces of $\R^M$\\
$\AA_{M,k}$		& the manifold of $k$-dimensional affine subspaces of $\R^M$\\
$V$, $W$ & elements of $\AA_{M,k}$\\
$\cP_{r,q,M}=\cP_{r}$ & the polynomials $P\colon\R^M\to\R^q$ of degree $\le r$\\
$\cP_{r}^{\text{reg}}$ & the regular polynomials in $\cP_{r}$\\
$P,P_1,P_2$ & elements of $\cP_{r}$\\
$\zero(P)$, $\zero_P$ & the set $P^{-1}(0)\cap[0,1]^M$\\
$[L]$ & the integers $1,\ldots,L$\\
$(x_1,\ldots,x_m)$ & notation for the elements of $(\R^d)^m$\\
$(x_i^1,\ldots,x_i^d)$ & the (real) coordinates of $x_i$ in the above notation\\
$\Delta$ & the diagonals $\{x_i=x_j\}$, $i\neq j$ in the above notation\\
$\pi_W$ & orthogonal projection onto $W$\\
$\pi_i$ & orthogonal projection onto the $i$
:th coordinate\\
$\pi^i$ & orthogonal projection onto the $[m]\setminus\{i\}$ coordinates\\
$E(\e)$ & open $\e$-neighbourhood of a set $E\subset\R^d$\\
$E^\circ$ & interior of $E$\\
$\overline{E}$ & closure of $E$\\
\hline
\end{tabular}
\vspace{2mm}
\caption{Summary of notation\label{notation}}
\end{table}

Starting from Section \ref{sec:affine-intersections}, we will be working on the space $(\R^{d})^m$ for some integers $m,d$, which we sometimes shorten to $\R^{md}$. We will denote the elements of $(\R^d)^m$ by $(x_1,\ldots,x_m)$, where $x_j=(x_j^1,\ldots,x_{j}^d)\in\R^d$ for each $j\in[m]$, so that the (real) coordinates of $x_j$ are denoted $x_{j}^i$, $i\in[d]$. Given $1\le j\le m$, we will denote by $\pi_j$ the orthogonal projection onto the subspace
\[H_j:=\{x\in(\mathbb{R}^{d})^m\,:\, x_i=0\text{ for }i\neq j\}\in\mathbb{G}_{md,d}\,.\]
and by $\pi^j$ the projection onto the orthogonal complement of $H_j$ (which is an element of $\mathbb{G}_{md,(m-1)d}$). Furthermore, we will identify each $H_j$ with $\R^d$, and $H_j^{\perp}$ with $(\R^d)^{m-1}$. That is, for $x=(x_1,\ldots,x_m)\in(\R^d)^m$, $\pi_j(x)=x_j$, $\pi^j(x)=(x_1,\ldots,x_{j-1},x_{j+1},\ldots,x_m)$. We will denote by $\Delta\subset(\R^{d})^m$ the union of all the diagonals $\{x_i=x_j\}$, $i\neq j$. Furthermore, given an index set $I\subset[m]$, we denote
\[H^{I}=\{(x_1,\ldots,x_m)\in(\R^{d})^m\,:\,x_j=0\text{ for all }j\in I\}\,.\]
Moreover, if $j\in[m]\setminus I$ and $k\in[d]$, we let
\[H^{I,j,k}=\{x\in H^{I}\,:\,x_j^k=0\}\,.\]

In Sections \ref{sec:nonlinear} and \ref{sec:szemeredi}, we will replace the linear subspaces $V\in\AA_{M,k}$ ($M=md$) by algebraic varieties $\zero_P=\zero(P):=P^{-1}(0)\cap [0,1]^M$, where $P$ is a polynomial $P\colon\R^M\to\R^q$. Let $\cP_{r,q,M}$ denote the family of polynomials $\R^M\to\R^q$ of degree $\le r$
and write $\cP_{r,q,M}^{\text{reg}}$ for the polynomials in $\cP_{r,q,M}$ for which $0$ is a regular value on $[0,1]^M$.
We identify elements $P=(P_1,\ldots,P_q)$ of $\cP_{r,q,M}$ with the coefficients of $P_i$, $i\in[q]$ and in this way see $\cP_{r,q,M}$ as a subset of some Euclidean space. The Euclidean distance between the coefficients of $P_1,P_2\in\cP_{r,q,M}$, induces a metric on $\cP_{r,q,M}$ and this will be denoted by $|P_1-P_2|$.

Throughout the paper, $C, C', C_1$, etc., denote positive deterministic constants whose precise value is of no importance (and their value may change from line to line), while $K, K', K_1$ etc. will always denote random positive real numbers.

Our notation is summarized in Table \ref{notation}.  These will be specified later whenever needed.

\section{Preliminaries on fractal percolation}
\label{sec:perco}

In this section we review some standard facts about fractal percolation. It will be convenient for us to work with fractal percolation conditioned on survival, so we begin by describing this variant.

Given $d\in\N$ and $2^{-d}<p<1$, we consider fractal percolation in $[0,1]^d$ with parameter $p$ (recall the definition from the beginning of \S\ref{subsec:sum-res}). Denote by $\widetilde{A}_n$ the union of the retained cubes in $\mathcal{Q}_n$ (the cubes that have not been removed in the first $n$ generations of the fractal percolation process). Let
\[A=A^{\text{perc}}=A^{\text{perc}(d,p)}=\overline{\bigcap_{n\in\N}\widetilde{A}_n}\]
denote the fractal percolation limit set. We take the closure to ensure the compactness of $A$; recall that the elements of $\mathcal{Q}$ are half-open.

Since $p>2^{-d}$, it is well known that the limit set $A$ is non-empty with positive probability. Nevertheless, for any $p<1$, the probability of extinction (i.e. $A=\varnothing$) is  positive. We consider the \textbf{surviving fractal percolation} defined via the following procedure. Let $p>2^{-d}$, and given $k\in[2^d]$, denote by $p_k>0$ the probability
\begin{align*}
p_k=\PP\left(\text{ There are exactly }k\text{ surviving cubes }Q\in\mathcal{Q}_1\,|\,A\neq\varnothing\right)
\end{align*}
where $Q\in\mathcal{Q}_n$ is called \textbf{surviving} if for each $m\ge n$, there is $Q'\in\mathcal{Q}_m$ such that $Q'\subset \widetilde{A}_m\cap Q'$ (expressing $A$ via the associated dyadic tree, this means that the sub-tree rooted at the vertex corresponding to $Q$ is infinite).
Although the precise formula is not important, we note that
\[p_k=\binom{2^d}{k}p^k(1-q)^{k-1}(1-p(1-q))^{2^d-k}\,,\]
where $q\in]0,1[$ is the probability that $A=\varnothing$. We recall that $q$ is the smallest root of $f\colon[0,1]\to\R$, where
\[f(t)=\sum_{k=0}^{2^d}\binom{2^d}{k}p^k(1-p)^{2^d-k}t^k\,.\]
is the probability generating function corresponding to the Galton-Watson process associated to fractal percolation with extinction (See \cite[\S 5.1]{LyonsPeres17}).

For $n\ge 0$, denote by $A_n$ the union of the surviving cubes in $\mathcal{Q}_n$. Then,
\[A=\bigcap_{n=0}^\infty \overline{A_n}\,.\]
and notice that with this notation, $A\neq\varnothing$ and $A_0=[0,1]^d$ are the same event, so that we can condition on each of them indistinctly.
We observe that the law of $A_n$ (and whence $A$) on $A\neq\varnothing$ is given by a Galton-Watson process with the offspring probabilities $(p_k)_{k\in[2^d]}$ (see \cite[Proposition 5.28]{LyonsPeres17} for details). In particular, $A_0=[0,1]^d$ and for each $Q\in\mathcal{Q}_n$, conditional on $Q\subset A_n$, the probability that $Q'\subset A_{n+1}\cap Q$ for exactly $k$ cubes $Q\in\mathcal{Q}_{n+1}$ equals $p_k$. Furthermore, denoting by $\textbf{1}[A_n]$ the indicator function of $A_n$ and letting
\[\nu_n=p^{-n}\textbf{1}[A_n]\,,\]
the distribution of $\nu_{n+1}|_Q$, $Q\in\mathcal{Q}_{n}$ are independent conditional on $\mathcal{B}_n$, where $\mathcal{B}_n$ is the sigma-algebra generated by the random sets $A_n$. One easily checks that $\PP(Q\subset A_1)=p$ for $Q\in\mathcal{Q}_1$, and this together with the stochastic self-similarity implies the martingale property
\[
\EE(\nu_{n+1}(x)\,|\,\mathcal{B}_n)=\nu_n(x)\text{ for all }x\in[0,1]^d,\,n\in\N\,.
\]
Note that since each $\mathcal{Q}_n$ consists of pairwise disjoint cubes, this holds also on the boundaries of the dyadic cubes.  We may interpret each $\nu_n$ as a measure (assigning mass $\nu_n(B)=p^{-n}\mathcal{L}(B\cap A_n)$ to each Borel set $B\subset\R^d$). It is easy to see that this sequence of measures is almost surely convergent in the weak-* sense, we denote the limit measure by $\nu$. The above discussion shows also that if $\wt{\nu}$ denotes the original fractal percolation measure (defined via the retained cubes instead of the surviving ones), then conditioned on $\wt{A}\neq\varnothing$, the measures $\nu$ and $\wt{\nu}$ are multiples of each other.

It is known (see \cite[Theorem 4.1]{Liu01}) that
\[
\underline{\dimloc}(\nu,x)=\liminf_{r\downarrow 0}\frac{\log\nu(B(x,r))}{\log r}=s\text{ for all }x\in A.
\]
This property implies, via the mass distribution principle, that $\dim_H(A')\ge s$ for any set $A'$ of positive $\nu$-measure; in particular this is true for $A$. On the other hand, since $2^{-sn} |\{Q\in\mathcal{Q}_n\,:\,Q\subset A_n\}|$ is a positive martingale, we get that $\overline{\dim}_B(A)\le s$ and therefore $\dim_H(A)=\dim_B(A)=s$.

Throughout the rest of the paper, we will always work with the surviving fractal percolation as just defined, and denote the associated probability measure by $\PP$. If needed, the original definition via the sets $\widetilde{A}_n$ will be referred to as \textbf{fractal percolation with extinction} and its law is denoted $\widetilde{\PP}$. To conclude, note if an event $\mathfrak{F}$ holds $\PP$-almost surely, then $\wt{\PP}(\mathfrak{F}\,|\,A\neq \varnothing)=1$. Hence, it is enough to prove all the theorems stated in \S\ref{subsec:sum-res} for surviving fractal percolation.

We now present a zero-one law for surviving fractal percolation that will be very useful in our study of patterns. From now on, let
\[N_n=|\{Q\in\mathcal{Q}_n\,:\,Q\subset A_n\}|\]
be the number of generation $n$ cubes for (surviving) fractal percolation.
\begin{lemma} \label{lem:0-1-strong}
Let $\mathfrak{C}$ be a collection of subsets of $[0,1]^d$ such that $\PP(\mathfrak{C})>0$ (for simplicity of notation, we denote $\PP(\mathfrak{C})=\PP(A\in\mathfrak{C})$ and, in particular, we assume that $A\in\mathfrak{C}$ is a measurable event).  Then almost surely there exists $n_0$ such that for all $n\ge n_0$ there is a cube $Q\in\mathcal{Q}_n$ such that $h_Q(A\cap \overline{Q})\in\mathfrak{C}$, where $h_Q$ is the homothety renormalizing $\overline{Q}$ back to $[0,1]^d$.
\end{lemma}
\begin{proof}
We claim that there is a constant $\delta=\delta(d,p)>0$ such that
\begin{equation} \label{eq:many-surviving-squares}
\PP(N_n \le \delta n) \le (1-\delta)^n.
\end{equation}
Let $f(t)=\sum_{k=1}^{2^d} p_k t^k$ be the probability generating function for the associated Galton-Watson process. Note that
\[
f(t) \le t(p_1+(1-p_1)t) \le t(p_1+(1-p_1)/2) =: \gamma^2 t
\]
for $t\le 1/2$, so that $f^n(1/2)\le \gamma^{2n}$ (note that $\gamma<1$). Now by Markov's inequality, and \cite[Proposition 5.2]{LyonsPeres17},
\[
\PP(2^{-N_n}>\gamma^n) \le \frac{\EE(2^{-N_n})}{\gamma^n} \le \gamma^n.
\]
This shows that \eqref{eq:many-surviving-squares} holds.

For each $n$, let $\mathfrak{F}_n$ be the event that $h_Q(A\cap Q)\notin\mathfrak{C}$ for all $Q\in\mathcal{Q}_n$ making up $A_n$. Then \eqref{eq:many-surviving-squares}  gives
\[
\PP(\mathfrak{F}_n) \le \PP(N_n \le \delta n) + \PP(\mathfrak{F}_n|N_n\ge \delta n) \le (1-\delta)^n + (1-\PP(\mathfrak{C}))^{\delta n}.
\]
The Borel-Cantelli lemma now yields the result.
\end{proof}

As a corollary, we obtain a small variant of the standard zero-one law for Galton-Watson processes (see e.g. \cite[Proposition 5.6]{LyonsPeres17}).
\begin{cor}\label{cor:0-1}
Let $\mathfrak{C}$ be a collection of subsets of $\R^d$ such that: (i) if $E\subset E'$ and $E\in\mathfrak{C}$ then $E'\in\mathfrak{C}$, (ii) any homothetic copy of $E\in\mathfrak{C}$ is again in $\mathfrak{C}$. Furthermore, assume that the event $A\in\mathfrak{C}$ is measurable.

Then $\PP(A\in\mathfrak{C})\in \{0,1\}$.
\end{cor}
\begin{proof}
Suppose $\PP(A\in\mathfrak{C})>0$. Then Lemma \ref{lem:0-1-strong} ensures that almost surely $h_Q(A\cap Q)\in\mathfrak{C}$ for some cube $Q$, which in light of the assumptions on $\mathfrak{C}$ gives the claim.
\end{proof}

In our applications of these zero-one laws, $\mathfrak{C}$ will consist of sets containing certain patterns, such as all angles in a given open set. Another very useful basic property is Harris' inequality (a special case of the FKG inequality). We state it in a form suited to fractal percolation. Recall that $q=q(d,p)$ denotes the extinction probability for the fractal percolation with extinction.
\begin{lemma} \label{lem:Harris}
Let $\mathfrak{C}_1,\mathfrak{C}_2$ be collections of subsets of $[0,1]^d$ which are closed under taking supersets, and such that $A\in\mathfrak{C}_i$ is measurable  Then,
\[
\PP(\mathfrak{C}_1\cap\mathfrak{C}_2) \ge (1-q)\PP(\mathfrak{C}_1)\PP(\mathfrak{C}_2).
\]
\end{lemma}

\begin{proof}
If $\PP(\mathfrak{C}_1)=1$ or $\PP(\mathfrak{C}_2)=1$ the claim is trivially true. We may thus assume that $\varnothing\notin \mathfrak{C}_1\cup\mathfrak{C}_2$.
Let $\widetilde{\PP}$ denote the law of fractal percolation with extinction. Since we are assuming that $\varnothing\notin\mathfrak{C}_1\cup\mathfrak{C}_2$, it follows that
\begin{equation}\label{eq:ne_does_not_matter}
(1-q)\PP(\mathfrak{C})=\widetilde{\PP}(\mathfrak{C}),\text{ for }\mathfrak{C}=\mathfrak{C}_1,\, \mathfrak{C}_2,\, \mathfrak{C}_1\cap\mathfrak{C}_2\,.
\end{equation}
Recalling that fractal percolation with extinction corresponds to Bernoulli percolation on a $2^d$-adic tree and $\mathfrak{C}_1$, $\mathfrak{C}_2$ correspond to increasing events, Harris inequality (see \cite[\S 5.8]{LyonsPeres17}) yields
\[\widetilde{\PP}(\mathfrak{C}_1\cap\mathfrak{C}_2) \ge \widetilde{\PP}(\mathfrak{C}_1)\widetilde{\PP}(\mathfrak{C}_2)\,.\]
Combining with \eqref{eq:ne_does_not_matter} gives the claim.
\end{proof}

\begin{rem}
Suitable versions of Lemma \ref{lem:0-1-strong} and Lemma \ref{lem:Harris} hold also for the finite level approximations $A_n=A^{\text{perc}}_n$ (with the same proofs).
\end{rem}

A classical result of Lyons asserts that for an arbitrary tree, the critical survival percolation parameter equals the branching number (essentially, the Hausdorff dimension of the boundary). Representing sets via their associated dyadic  trees, this yields the the following Euclidean version; see \cite[Theorem 15.11]{LyonsPeres17} for the proof (of a sharper and more general result).
\begin{thm} \label{thm:hausdorff-dim-percolation}
Let $B\subset [0,1]^d$ be a closed set, and let $A=A^{\text{perc}(d,p)}$.  If $\PP(A\cap B\neq\varnothing)>0$, then $\dim_H(B)\ge d-s$.
\end{thm}
This is very useful when $B$ is random, because it allows to estimate the Hausdorff dimension of a random set by testing survival of a smaller random set, which is a priori an easier problem.

\section{A class of random measures and their intersections with parametrized families of deterministic measures}

In this section we state and prove our main result on continuity of intersections. This result is presented and proved in an abstract framework. In the later sections we will apply this result mostly to Cartesian products of fractal percolation to deduce our geometric applications. As mentioned in the introduction, we believe that Theorem \ref{thm:Holder-continuity} should have similar applications to a wide variety of random measures including various subdivision and cut-out type random fractals.  We start by defining the necessary concepts.

\subsection{Random measures}
Our goal in this section is to study intersections of random measures $\mu$ with a deterministic family of measures $\{ \eta_t\}_{t\in\Gamma}$.

We consider a sequence of Borel functions $\mu_n\colon\R^M\to[0,+\infty)$, corresponding to the densities of absolutely continuous measures (also denoted $\mu_n$). We note that these are actual functions (defined for every $x$) and not equivalence classes, since we will be integrating them against arbitrary measures. We assume that the following standing assumptions hold:

\begin{enumerate}
\renewcommand{\labelenumi}{(RM\arabic{enumi})}
\renewcommand{\theenumi}{RM\arabic{enumi}}
\item \label{RM:bounded} $\mu_0$ is a deterministic bounded function with bounded support.
\item \label{RM:measurable} There exists an increasing filtration of $\sigma$-algebras $\mathcal{B}_n$ (on some space $\Omega$) such that $\mu_n$ is $\mathcal{B}_n$-measurable.
\item \label{RM:quotients-bounded}
There is $C<\infty$ such that $\mu_{n+1}(x)\le C\mu_n(x)$ for all $n\in\N$, $x\in\R^M$.
\end{enumerate}

The last condition is of technical nature and could certainly be weakened. If we replace $C$ by a deterministic sequence $C_n$ growing at most subexponentially, then the proof of our main abstract theorem, Theorem \ref{thm:Holder-continuity}, goes through with very minor changes.  The papers \cite{CKLS14, FalconerJin16} consider geometric properties of random measures which satisfy \eqref{RM:bounded}, \eqref{RM:measurable}, and a variant of \eqref{RM:quotients-bounded} in which $C$ is random and/or grows quite fast with $n$. This suggests that there is scope for weakening the last condition considerably. Since for Cartesian products of fractal percolation, which is the focus of this article, \eqref{RM:quotients-bounded} holds as stated, we do not consider these variants here.

We now introduce the parametrized families $\{ \eta_t\}_{t\in\Gamma}$ of (deterministic) measures. We always assume the parameter space is a totally bounded metric space $(\Gamma,d)$.

Our main objects of interest will be the `intersections' of the random measures $\mu_n$ with  $\eta_t$ as $n\rightarrow\infty$, and their behaviour as $t$ varies.
Formally, we define:
\begin{align*}
\mu_n^t(A) &= \int_A \mu_n(x) d\eta_t(x),\\
\end{align*}
for each Borel set $A\subset\R^M$, $n\in\N$ and $t\in\Gamma$.
We are mainly interested in the asymptotic behaviour of the total mass, and denote
\begin{align*}
Y_n^t &= \|\mu_n^t\| = \int \mu_n(x) d\eta_t(x),\\
Y^t &=\lim_{n\to\infty} Y_n^t \quad\text{(if the limit exists)}.
\end{align*}

\subsection{Martingale condition}
\label{subsec:martingale}
Conditions \eqref{RM:bounded}--\eqref{RM:quotients-bounded} by themselves are far too weak to guarantee the convergence of $\mu_n$ or the regularity of the intersections $\mu_n^t$ (or $Y_{n}^t$). Thus, we need to impose further conditions, to at least ensure the a.s. existence of a limit measure $\mu$.

\begin{defn} \label{def:martingale}
A random sequence $( \mu_n)$ satisfying \eqref{RM:bounded}--\eqref{RM:quotients-bounded} will be called a \textbf{martingale measure}, if for all
 $x\in\R^M$ and $n\in\N$,
\begin{equation}\label{eq:martingalecondition}
\EE(\mu_{n+1}(x)|\BB_n) = \mu_n(x).
\end{equation}
\end{defn}

In other words, a martingale measure is a $T$-martingale in the sense of Kahane \cite{Kahane87}  with the extra growth condition \eqref{RM:quotients-bounded}, and it is well known and easy to see that, in this case, the sequence $\mu_n$ converges a.s. in the weak*-sense to a random limit measure $\mu$. Furthermore, for each fixed $t\in\Gamma$, also $\mu_n^t$ (and $Y_n^t$) converges a.s. to a random limit $\mu^t$ (resp. $Y^t=||\mu^t||$).

\subsection{Spatial independence}
\label{subsec:spatial-independence}

Martingale measures may exhibit long range spatial dependencies; in order to obtain any results about intersections, we need to impose conditions that guarantee a sufficient degree of independence in the process of defining $\mu_n$. The best that we could hope for is that if $\{ Q_j\}$ are dyadic cubes of side length $2^{-n}$ that are pairwise disjoint then, conditioned on the $n$th step of the construction, the masses $\{\mu_{n+1}(Q_j)\}$ are independent random variables. This is the content of the next definitions originating from \cite{ShmerkinSuomala14}.

\begin{defn}\label{def:usi}
A sequence $(\mu_n)_{n\in\N}$ satisfying \eqref{RM:bounded}--\eqref{RM:quotients-bounded} is \textbf{uniformly spatially independent} (USI) if there exists a constant $C>0$ such that for any $(C 2^{-n})$-separated family $\mathcal{Q}$ of dyadic cubes of side-length $2^{-(n+1)}$, the restrictions $\{\mu_{n+1}|_Q | \mathcal{B}_n\}$ are independent.
\end{defn}
\begin{defn}\label{def:si}
A sequence $(\mu_n)_{n\in\N}$ satisfying \eqref{RM:bounded}--\eqref{RM:quotients-bounded}
is called \textbf{spatially independent} (with respect to the family $\{ \eta_t:t\in\Gamma\}$) if there exists a constant $C>0$ such that for any $t\in\Gamma$, any $n\in\mathbb{N}$, and for any $C 2^{-n}$-separated family $\mathcal{Q}$ of dyadic cubes of side-length $2^{-(n+1)}$, the random variables $\{ \mu_{n+1}^t(Q)|\mathcal{B}_n\}_{Q\in\mathcal{Q}}$ are independent.
\end{defn}

The paper \cite{ShmerkinSuomala14} deals with martingale measures which are spatially independent (these will be termed \textbf{SI-martingales} for short). In order to handle cartesian products of independent fractal percolations, we will need to allow some long-range dependencies between the masses $\mu_{n+1}^t(Q)$ as long as they are ``sparse'' with large probability. In order to define this notion formally, we recall the concept of dependency graph: given an index set $I$, a graph with vertex set $I$ is a \textbf{dependency graph} for a family of random variables $\{ X_i: i\in I\}$ if for any $i\in I$ and any subset $J\subset I$ such that there is no edge from $i$ to any element of $J$, the random variable $X_i$ is independent from $\{ X_j:j\in J\}$.

\begin{defn} \label{def:DI-n}
Let $(\mu_n)_{n\in\N}$ be a martingale measure, and let $\{\eta_t\}_{t\in\Gamma}$ be a family of measures. The \textbf{dependency degree} at step $n$, denoted $\mathcal{DI}_n$, is the smallest constant $\Psi$ such that, \emph{for all} $t\in\Gamma$, there is a dependency graph for $\{ \mu_{n+1}^t(Q)|\mathcal{B}_n\}_{Q\in\mathcal{Q}_{n+1}}$ of degree at most $\Psi$.
\end{defn}

Clearly, if $(\mu_n)$ is spatially independent, then $\mathcal{DI}_n$ is bounded over all $n$.
In many cases $(\mu_n)$ will only be weakly spatially dependent, in the sense that $\mathcal{DI}_n$ will grow at a sufficiently slow rate:

\begin{defn}
A sequence $(\mu_n)_{n\in\N}$ satisfying \eqref{RM:bounded}--\eqref{RM:quotients-bounded} is \textbf{weakly spatially dependent (WSD)} with parameter $\delta\ge 0$ 
(with respect to a family $\{ \eta_t:t\in\Gamma\}$) if there is a random sequence $\Psi(n)$, such that
the following holds.
\begin{enumerate}
\item $\mathcal{DI}_n\le \Psi(n)$ for all $n$.
\item For each $\varepsilon>0$, there exist a (deterministic) $C=C_\varepsilon>0$ and $\varepsilon_n>0$ with $\sum_n \varepsilon_n<\varepsilon$, and $\mathcal{B}_n$-measurable events
\[\mathfrak{C}_{n}\subset\{\Psi(n)\le C 2^{\delta n}\}\]
such that for $n\ge 1$,
\[
\PP_{\mathfrak{C}_{n-1}}(\mathfrak{C}_{n}|\mathcal{B}_{n-1})\ge 1-\varepsilon_{n}\,,
\]
where we denote $\mathfrak{C}_0=\Omega$ (the entire probability space).
\end{enumerate}
\end{defn}
Here, and in the sequel, given a positive probability event $\mathfrak{C}$, we denote by $\PP_\mathfrak{C}$ the induced conditional probability distribution, i.e. $\PP_{\mathfrak{C}}(\mathfrak{F})=\PP(\mathfrak{F}|\mathfrak{C})$. We also remark that $\PP(\cdot|\mathcal{B}_n)$ is a random variable, and hence the WSD condition requires that whatever the realization of $\mu_n$ (so long as $\mathfrak{C}_n$ is satisfied), there is a very large probability that $\mathfrak{C}_{n+1}$ again is satisfied.

The reason we introduce the random sequence $\Psi(n)$, rather than dealing with $\mathcal{DI}_n$ directly, is that in practice we will have information about certain natural dependency graphs, while the minimum in the definition of $\mathcal{DI}_n$ is an awkward quantity to work with.

This definition can be motivated by the example $\mu_n = \nu_n^{(1)}\times\nu_n^{(2)}\times \nu_n^{(3)}$, with $\nu_n^{(i)}$ independent realizations of the fractal percolation measure on $[0,1]$. In this case, the masses $\mu_{n+1}(I_{i,1}\times I_{i,2}\times I_{i,3})$ are independent, conditional on $\mathcal{B}_n$, provided that $I_{i,j}$ are all contained in different intervals in $\mathcal{Q}_n$. However, given $I_i,J\in\mathcal{Q}_{n+1}$, then (for example) $\mu_{n+1}(J\times I_1\times I_2)$ and $\mu_{n+1}(I_3\times I_4\times J)$ are certainly not independent given $\mathcal{B}_n$. We then see that one can bound the dependency degree in terms of the maximum of the cardinalities of intersections of $A_n^{(1)}\times A_n^{(2)}\times A_n^{(3)}$ with planes in principal directions, and this will allow us to show weak spatial dependence with a suitable small $\delta$.

\subsection{H\"{o}lder continuity of intersections}

\label{sec:continuity}

The role of the spatial independence (or weak dependence) condition is to ensure that, with overwhelming probability, the convergence of $Y_n^t$ is very fast. More precisely, we decompose $Y_{n+1}^t - Y_n^t = \sum_{Q\in\mathcal{Q}_n} X_Q^t$, where
\[
X_Q^t = \int_Q (\mu_{n+1}-\mu_n)\,d\eta_t.
\]
Weak dependence ensures that there is enough independence among the $X_Q^t$ that its sum is tightly concentrated around the mean, implying that $Y_{n+1}^t$ is very close to $Y_n^t$ with very large probability. These ideas are made precise in Lemma \ref{lem:large-deviation}, which is a small adaptation of \cite[Lemma 3.4]{ShmerkinSuomala14}. Before stating this lemma, we recall the following definition.
\begin{defn}
We say that a measure $\eta$ has \textbf{Frostman exponent} $\kappa\ge 0$, if there exists a constant $C>0$ such that
\begin{equation}\label{eq:def_dim_deterministic_family}
\eta(B(x,r)) \le C r^\kappa\quad\text{for all }x\in\R^M,\,0<r<1\,.
\end{equation}
The family $\{ \eta_t\}_{t\in\Gamma}$ has Frostman exponent
$\kappa$ if each $\eta_t$ satisfies \eqref{eq:def_dim_deterministic_family} with a uniform constant $C$.
\end{defn}

\begin{lemma} \label{lem:large-deviation}
Let $(\mu_n)$ be a martingale measure, and let $\eta$ be a measure with Frostman exponent $\kappa\ge 0$. Fix $n$, positive constants $\Psi,\Upsilon$, and write $\mathfrak{F}$ for the event $\mathcal{DI}_n\le \Psi$, $\sup_{x\in\R^M}\mu_n(x)\le\Upsilon$, and suppose that \eqref{eq:martingalecondition} holds for $x\in\supp(\eta)$.

Then, for any $\varrho>0$ with
\begin{equation}\label{kappaeq}
\varrho^2 2^{\kappa n}\Upsilon^{-1}\Psi^{-1}\ge c_0>0\,,
\end{equation}
it holds that
\begin{equation*}
\begin{split}
\PP_{\mathfrak{F}}\left(\left|\int(\mu_{n+1}-\mu_n)\,d\eta\right|\ge \varrho\sqrt{\int\mu_n\,d\eta}\,|\,\mathcal{B}_n\right)\le O\left(\exp\left(-\Omega\left(\varrho^2 2^{\kappa n}\Upsilon^{-1}\Psi^{-1}\right)\right)\right),
\end{split}
\end{equation*}
where the implicit constants depend only on $c_0$, the ambient dimension $M$, and the constant $C$ in the definition of Frostman exponent of $\eta$.
\end{lemma}

We underline that the lemma provides a uniform bound over all realizations of $\mu_n$ in the event $\mathfrak{F}$. Of course, this implies that same bound holds conditioning only on $\mathfrak{F}$, but knowledge of this is not enough for us. In the proof we will use a generalization of Hoeffding's inequality due to Janson \cite[Theorem 2.1]{Janson04}, which allows for dependencies among the random variables:
\begin{lemma} \label{lem:HoeffdingJanson}
Let $\{ X_i: i\in I\}$ be zero mean random variables uniformly bounded by $R>0$, and suppose there is a dependency graph with degree $\Psi$. Then
\begin{equation} \label{eq:Hoeffding-Janson}
\mathbb{P}\left(\left|\sum_{i\in I} X_i\right|> \varrho\right) \le 2\exp\left(\frac{- 2 \varrho^2}{(\Psi+1)|I| R^2}\right).
\end{equation}
\end{lemma}

\begin{proof}[Proof of Lemma \ref{lem:large-deviation}]
We condition on a realization of $\BB_n$ which is contained in $\mathfrak{F}$. We will obtain the desired bound irrespective of the specific realization, hence establishing the claim. By assumption, there is $C_1>0$ such that $\eta(Q)\le C_1 2^{-\kappa n}$ for all $Q\in\mathcal{Q}_{n}$. Define $d\eta_n=\mu_n d\eta$.

We decompose $\mathcal{Q}_{n+1}$ into the families
\[
\mathcal{Q}_{n+1}^\ell = \{ Q\in \mathcal{Q}_{n+1}: C_1 \Upsilon 2^{-\kappa\ell}< \eta_n(\widehat{Q}) \le C_1 \Upsilon 2^{\kappa(1-\ell)} \},
\]
where $\widehat{Q}\in\mathcal{Q}_n$ is the dyadic cube containing $Q$. Since
\[
\eta_n(Q) \le \int_Q \Upsilon\, d\eta \le C_1 \Upsilon 2^{-\kappa n}
\]
for $Q\in\mathcal{Q}_n$, we see that $\mathcal{Q}_{n+1}^\ell$ is empty for all $\ell\le n$.

For each $Q\in\Q_{n+1}$, let $X_Q=\eta_{n+1}(Q)-\eta_n(Q)$. Then $\EE(X_Q)=0$ for all $Q\in\mathcal{Q}_{n+1}$; recall that we are conditioning on $\mathcal{B}_n$. Also, by \eqref{RM:quotients-bounded},
\[
|X_Q| \le O(\eta_n(Q)) \le O(1) 2^{-\kappa\ell} \Upsilon \,\,\text{ for all } Q\in\Q_{n+1}^\ell.
\]
Moreover, since $\|\eta_n\| = \sum_{Q\in\Q_n} \eta_n(Q)$, we have
\[
|\Q_{n+1}^\ell| \le O(1) 2^{\kappa\ell}\Upsilon^{-1}\|\eta_n\|.
\]
Furthermore, by definition of the dependency index, there is a dependency graph for $\{ X_Q:\mathcal{Q}_{n+1}^\ell\}$ of degree at most $\Psi$. Therefore, by the Hoeffding-Janson inequality \eqref{eq:Hoeffding-Janson},
\[
\PP\left(\left|\sum_{Q\in\Q_{n+1}^\ell} X_Q\right|>\frac{\varrho\sqrt{\|\eta_n\|}}{2(\ell-n)^2} \right) =O\left( \exp\left(-\Omega\left((\ell-n)^{-4}\varrho^2 2^{\kappa\ell} \Upsilon^{-1}\Psi^{-1}\right)\right)\right).
\]
for any $\varrho>0$. It follows that
\begin{align*}
\PP\left(\left|\int (\mu_{n+1}-\mu_n)\,d\eta\right|>\varrho\sqrt{\|\eta_n\|}\right) &\le \sum_{\ell> n}\PP\left(\left|\sum_{Q\in\Q_{n+1}^\ell} X_Q\right|>\frac{\varrho\sqrt{\|\eta_n\|}}{2(\ell-n)^2} \right) \\
&=O\left(\exp\left(-\Omega\left(\varrho^2 2^{\kappa n}\Upsilon^{-1}\Psi^{-1}\right)\right)\right)\,,
\end{align*}
for any $\varrho>0$, where \eqref{kappaeq} is used for the last estimate.
\end{proof}

The following is the main abstract result of the paper.
\begin{thm} \label{thm:Holder-continuity}
Let $(\mu_n)_{n\in\N}$ be a WSDM with parameter $\delta$, and let $\{\eta_t\}_{t\in\Gamma}$ be a family of measures indexed by the metric space $(\Gamma,d)$.
We assume that there are positive constants $\alpha,\kappa,\theta,\gamma_0$ such that the following holds:
\begin{enumerate}
\renewcommand{\labelenumi}{\textup{(H\arabic{enumi})}}
\renewcommand{\theenumi}{\textup{H\arabic{enumi}}}
\item \label{H:size-parameter-space} $\Gamma$ has finite upper box-counting dimension (i.e. it can be covered by $O(1) r^{-O(1)}$ balls of radius $r$ for all $r\in (0,1]$).
\item \label{H:dim-deterministic-measures} The family $\{\eta_t\}_{t\in\Gamma}$ has Frostman exponent $\kappa$.
\item \label{H:codim-random-measure} Almost surely, $\mu_n(x)\le 2^{\alpha n}$ for all $n\in \N$ and $x\in\R^M$.
\item \label{H:Holder-a-priori} Almost surely, there is a (random) integer $K$, such that
    \begin{equation} \label{eq:Holder-a-priori}
        \sup_{t,u\in\Gamma,t\neq u;n\ge K} \frac{\left|Y_n^t-Y_n^u\right|}{2^{\theta n}\,d(t,u)^{\gamma_0}} <\infty.
    \end{equation}
\end{enumerate}

Further, suppose that the various parameters satisfy
\begin{equation} \label{eq:relations-parameter}
\kappa-\alpha-\delta >0\,.
\end{equation}

Then there is a deterministic number $\gamma>0$ (depending on all parameters) such that almost surely $Y_n^t$ converges uniformly in $t$, exponentially fast, to a limit $Y^t$, and the function $t\mapsto Y^t$ is H\"{o}lder continuous with exponent $\gamma$.
\end{thm}

The following lemma captures the core probabilistic argument needed in the proof of Theorem \ref{thm:Holder-continuity}. It is only here that the weak dependence assumption gets used, via an application of Lemma \ref{lem:large-deviation}.

\begin{prop} \label{prop:probabilistic-technical}
Under the assumptions of Theorem \ref{thm:Holder-continuity}, the following holds. Let $\lambda, B>0$ be such that
\begin{align}
\label{eq:bound-lambda} \lambda&< \frac12(\kappa-\alpha-\delta).
\end{align}
For each $n$, let $\Gamma_n\subset\Gamma$ be a subset with $O(\exp(O(nB)))$ elements.

We define
\begin{align*}
Z_n &= \max_{v\in\Gamma_n} |Y_{n+1}^v - Y_n^v|,\\
\overline{Y}_n &= \max_{t\in\Gamma} Y_n^t.
\end{align*}
Then almost surely there exists an integer $K_1$ such that
\begin{equation} \label{eq:main-claim}
Z_n \le 2^{-\lambda n}\max(\overline{Y}_n,1) \quad\text{for all }n\ge K_1.
\end{equation}

\end{prop}
\begin{proof}
Let $\Psi(n)$ be the random variable in the definition of weak martingale. Fix $\varepsilon>0$, and let $C_\e>0$, $\varepsilon_n>0$ and $\mathfrak{C}_{n}\subset\{\Psi(n)\le C_\e 2^{n\delta}\}$ be as in the definition of WSDM. Denote
\begin{align*}
\mathfrak{F}_{n} &= \left\{ Z_n \le 2^{-\lambda n} \max(\overline{Y}_n,1)\right\}.
\end{align*}
Hence, we want to show that $\PP(\liminf \mathfrak{F}_n)=1$.

To begin, we claim that there is a deterministic $n_0(\e)\in\N$ such that
\begin{equation} \label{eq:inductive-bound}
\PP_{\mathfrak{C}_n}((\mathfrak{C}_{n+1}\cap \mathfrak{F}_{n})^c\,|\,\BB_n)\le  \exp(2^{-c n})+\varepsilon_{n+1}
\end{equation}
for all $n\ge n_0(\e)$, and some constant $c>0$ independent of $n, \e$.

For a given $v\in\Gamma_n$, we know from Lemma \ref{lem:large-deviation} and our assumptions that
\begin{equation} \label{eq:application-large-dev}
\PP_{\mathfrak{C}_n}(|Y_{n+1}^v - Y_n^v|> \tfrac12 2^{-\lambda n} \sqrt{Y_n^v} \,|\,\BB_n) \le O\left(\exp\left(-\Omega_{\e}\left(2^{(\kappa-\alpha-\delta-2\lambda)n}\right)\right)\right)\,.
\end{equation}
Note that because the bound in Lemma \ref{lem:large-deviation} is uniform, it continues to hold after conditioning on the $\mathcal{B}_n$-measurable event $\mathfrak{C}_n$. Observe that \eqref{kappaeq} holds with $\Omega=\Omega_{\e}$ by \eqref{eq:bound-lambda}. Hence, recalling that $|\Gamma_n| = O(\exp(O(nB)))$, and using \eqref{eq:bound-lambda},
\begin{align*}
\PP_{\mathfrak{C}_n}\left(Z_n> \tfrac12 2^{-\lambda n}\overline{Y}_n^{1/2}\,|\,\BB_n\right)
&\le O(\exp(O(nB)))\exp\left(-\Omega_{\e}(1)2^{(\kappa-\alpha-\delta-2\lambda)n}\right)\\
&\le \exp(2^{-c n})
\end{align*}
for $c=(\kappa-\alpha-\delta-2\lambda)/2$, provided $n$ is large enough in terms of $\e$ only. Since $x^{1/2}\le \max(x,1)$, this yields $\PP_{\mathfrak{C}_n}(\mathfrak{F}_{n}^c\,|\,\BB_n)\le \exp(2^{-c n})$ for all $n\ge n_0(\e)$. As $\PP_{\mathfrak{C}_n}(\mathfrak{C}_{n+1}^c\,|\,\BB_n)\le \varepsilon_{n+1}$ by assumption, the estimate \eqref{eq:inductive-bound} follows.

Note that $\mathfrak{F}_k$ is $\mathcal{B}_{k+1}$-measurable (but not $\mathcal{B}_k$-measurable). Hence it follows from \eqref{eq:inductive-bound} that
\[
\PP(\mathfrak{C}_{k+1}\cap \mathfrak{F}_{k}|\mathfrak{C}_{k}\cap \mathfrak{F}_{k-1}) \ge  1-\exp(2^{-c k})-\varepsilon_{k+1}
\]
for all $k\ge n_0(\e)$, and likewise if we condition only on $\mathfrak{C}_k$. Using this and  the definition of WSDM, for any $n\ge n_0(\e)$ we estimate
\begin{align*}
\PP&\left(\bigcap_{k=n}^\infty \mathfrak{F}_{k}\right)\ge \PP\left(\bigcap_{k=n}^\infty \mathfrak{C}_{k}\cap \mathfrak{F}_{k}\right)\\
&\ge\PP(\mathfrak{C}_{n})\PP(\mathfrak{C}_{n+1}\cap \mathfrak{F}_{n}\,|\,\mathfrak{C}_{n})\prod_{k=n+1}^\infty \PP(\mathfrak{C}_{k+1}\cap \mathfrak{F}_{k}|\mathfrak{C}_{k}\cap \mathfrak{F}_{k-1})\\
&\ge\PP(\mathfrak{C}_{1})\prod_{\ell=1}^{n-1}\PP\left(\mathfrak{C}_{\ell+1}\,|\,\mathfrak{C}_{\ell}\right) \PP\left(\mathfrak{C}_{n+1}\cap \mathfrak{F}_{n}\,|\,\mathfrak{C}_{n}\right) \prod_{k=n+1}^\infty \PP\left(\mathfrak{C}_{k+1}\cap \mathfrak{F}_{k}|\mathfrak{C}_{k}\cap \mathfrak{F}_{k-1}\right)\\
&\ge\prod_{\ell=1}^\infty\left(1-\varepsilon_\ell\right)\prod_{k=n}^\infty\left(1- \exp(2^{-c k})-\e_{k+1}\right)\,,
\end{align*}
Since $\sum_k \varepsilon_k<\varepsilon$, we conclude that $\PP\left(\cap_{k=n}^\infty \mathfrak{F}_k\right)>1-O(\e)$ if $n$ is sufficiently large (depending on $\e$). This is what we wanted to show.
\end{proof}

\begin{proof}[Proof of Theorem \ref{thm:Holder-continuity}]
Having established Proposition \ref{prop:probabilistic-technical}, the proof is a small variant of that of \cite[Theorem 4.1] {ShmerkinSuomala14}. We give full details for the reader's convenience.

Pick constants $\lambda, B, B_0$ such that $0<\lambda<\frac12(\kappa-\alpha-\delta)$, $0<B_0<B$, and
\begin{equation} \label{eq:bound-A}
 \gamma_0-\frac{\theta}{B_0}=\frac{\lambda}{B_0}\,.
\end{equation}
Also, let
\begin{equation}\label{eq:choice_gamma}
0<\gamma<\frac{\lambda}{B_0}\,.
\end{equation}
Further, for each $n$, let $\Gamma_n$ be a $(2^{-nB})$-dense family with $O(\exp(O(nB)))$ elements, whose existence is guaranteed by \eqref{H:size-parameter-space}.

By Proposition \ref{prop:probabilistic-technical} and \eqref{H:Holder-a-priori}, almost surely there is $K\in\N$ such that \eqref{eq:Holder-a-priori} holds and
\begin{equation}\label{eq:bound_Z_n}
Z_n\le 2^{-\lambda n}\max(\overline{Y}_n,1) \quad\text{for all }n\ge K.
\end{equation}
For the rest of the proof we condition on such $K$. Our goal is to  estimate $X_{n+1}$ in terms of $X_n$, where $X_k=\sup_{t\neq u} X_k(t,u)$, and
\[
X_k(t,u) = \frac{|Y_k^t-Y_k^u|}{d(t,u)^\gamma}.
\]
If $d(t,u)\le 2^{-B_0 n}$, we simply use \eqref{eq:Holder-a-priori} to get a deterministic bound. Otherwise, we find $t_0,u_0$ in $\Gamma_n$ such that $d(t,t_0), d(u,u_0)<2^{-B n}$ and estimate
\[
|Y_{n+1}^t-Y_{n+1}^u| \le \I+\II+\III,
\]
where
\begin{align*}
\I &= |Y_n^t-Y_n^u|,\\
\II &= |Y_{n+1}^t-Y_{n+1}^{t_0}|+|Y_n^t-Y_n^{t_0}|+|Y_{n+1}^u-Y_{n+1}^{u_0}|+|Y_n^u-Y_n^{u_0}|,\\
\III &= |Y_{n+1}^{t_0}-Y_n^{t_0}|+|Y_{n+1}^{u_0}-Y_n^{u_0}|.
\end{align*}
The term $\I$ will be estimated inductively, for $\II$ we will use the a priori estimate \eqref{eq:Holder-a-priori} and to deal with $\III$ we appeal to Proposition \ref{prop:probabilistic-technical}.

We proceed to the details. If $n\ge K$ and $d(t,u)\le 2^{-B_0 n}$ then, by \eqref{eq:Holder-a-priori},
\begin{equation} \label{eq:estimate-close}
|Y_{n+1}^t-Y_{n+1}^u|\le 2^{(n+1)\theta} d(t,u)^{\gamma_0} \le  O(1) d(t,u)^{\gamma_0-\theta/B_0}.
\end{equation}
Hence in this regime, we get a H\"{o}lder exponent $\gamma_0-\theta/B_0=\lambda/B_0>0$,  thanks to \eqref{eq:bound-A}.

From now on we consider the case $d(t,u)>2^{-B_0 n}$. By definition,
\begin{equation} \label{eq:bound-I}
\I \le X_n d(t,u)^\gamma.
\end{equation}

Let $t_0,u_0\in\Gamma_n$ be $(2^{-Bn})$-close to $t,u$. Using \eqref{eq:Holder-a-priori}, if $n\ge K$ then $|Y_{k}^t-Y_{k}^{t_0}|\le 2^{k\theta} 2^{-\gamma_0 B n}$ for $k=n,n+1$, and likewise for $u,u_0$, whence
\begin{equation} \label{eq:bound-II}
\II \le O(1) 2^{-(\gamma_0 B- \theta-\gamma B_0)n}\, d(t,u)^\gamma.
\end{equation}
Note that due to \eqref{eq:bound-A}  and \eqref{eq:choice_gamma}, the exponent $\gamma_0 B- \theta-\gamma B_0$ is positive.

We are left to estimating $\III$. We first claim that
\begin{equation}\label{eq:bound_C_N}
\sup_{n\ge K} \overline{Y}_n\le O(2^{\alpha K})<\infty\,,
\end{equation}
recall that $\overline{Y}_n=\sup_{t\in\Gamma}Y^V_n$.
The point here is that $O(2^{\alpha K})$, although random, is independent of $n$. Let $n\ge K$. Using \eqref{H:Holder-a-priori} again
to estimate $Y_{n+1}^t$ via $Y_{n+1}^{t_0}$, with $t_0\in\Gamma_n$, $d(t,t_0)\le 2^{-B n}$, we have
\begin{align*}
\overline{Y}_{n+1} &\le \left(\max_{v\in\Gamma_n} Y_{n+1}^v\right)+ O(1)2^{(\theta-B\gamma_0)n} \\
&\le \overline{Y}_n + \max(1,\overline{Y}_n) 2^{-\lambda n}+ O(1)2^{(\theta-B\gamma_0)n}.
\end{align*}
Recall that we are conditioning on \eqref{eq:bound_Z_n}. Since $\lambda>0$, $\theta-B\gamma_0<0$, and $\overline{Y}_K=O(2^{\alpha K})$, this implies \eqref{eq:bound_C_N}.

Now it follows that $Z_n\le O(2^{\alpha K}) 2^{-\lambda n}$ and, in particular,
\begin{equation} \label{eq:bound-III}
\III \le O(2^{\alpha K}) \, 2^{-(\lambda-\gamma B_0) n} \,d(t,u)^\gamma\,,
\end{equation}
where $\lambda-\gamma B_0>0$ by \eqref{eq:choice_gamma}.

Putting together \eqref{eq:bound-I}, \eqref{eq:bound-II} and \eqref{eq:bound-III}, we have shown that there exist $\e>0$ and an almost surely finite random variable $K'>0$,  such that
\[
|Y_{n+1}^t-Y_{n+1}^u| \le (X_n+K' 2^{-\e n})d(t,u)^\gamma\quad\text{for all } n\ge K, t, u\in\Gamma,
\]
which immediately yields $\overline{X}:=\sup_n X_n<\infty$.

We are left to show that almost surely $Y_n^t$ converges uniformly, at exponential speed, since then we will have
\[
|Y^t-Y^u| = \lim_{n\to\infty} |Y_n^t-Y_n^u| \le \overline{X} d(t,u)^\gamma.
\]
To see this, observe that \eqref{eq:bound_C_N}, \eqref{eq:main-claim} yield that $Z_n$ decreases exponentially. Estimating $Y_m^t- Y_n^t$ via $Y_m^{t_0}-Y_n^{t_0}$ and using the estimates \eqref{eq:Holder-a-priori}, \eqref{eq:bound_Z_n}, we conclude that for all $t$,
\begin{align*}
|Y_{n+1}^t-Y_n^t| &\le |Y_{n+1}^t - Y_{n+1}^{t_0}| + |Y_{n+1}^{t_0}-Y_n^{t_0}| + |Y_{n}^{t_0}-Y_n^{t}| \\
&\le O(Y_K+1) 2^{-\lambda n}  + O(1) 2^{(\theta-B\gamma_0)n}.
\end{align*}
This shows that the sequence $\{ Y_n^t\}$ is uniformly Cauchy with exponentially decreasing differences, finishing the proof.
\end{proof}

\begin{rems}
\begin{enumerate}[(i)]
\item An inspection of the proof shows that one can take any
\[
 \gamma < \frac{\gamma_0(\kappa-\alpha-\delta)}{2(\theta+\lambda_0)}\,.
\]
We do not expect this to be optimal (see \cite{ShmerkinSuomala16} for a special case in which the sharp H\"{o}lder exponent can be determined.)
\item The proof works with minor changes if, instead of assuming \eqref{H:size-parameter-space}, we assume the weaker condition that $\Gamma$ can be covered by $\exp(r^{-\xi})$ balls of radius $r$ for small $r$, where
\[
 \xi\theta < \gamma_0(\kappa-\alpha-\delta)\,.
\]
Although we do not treat them here, we note that there are natural families that satisfy a size bound of this kind (but have infinite box counting dimension), such as Hausdorff measures on convex curves, see \cite[Proposition 6.1]{ShmerkinSuomala15}.
\item In all our applications, we will be able to take the random variable $K$ of \eqref{H:Holder-a-priori} to be $1$. Roughly speaking, this is because the a priori H\"{o}lder continuity will follow from the transversality of the hyperplanes in the dyadic grid with certain algebraic varieties, which is a purely deterministic geometric phenomenon. Allowing $K$ to be random is useful when the martingale $\mu_n$ is not tied to a fixed geometric frame, as is the case, for example, for Poissonian cutouts - see \cite{ShmerkinSuomala14} for further discussion.
\item In practice, the most important assumption on the parameters is $\kappa>\alpha$. Once this holds, it is often possible to find the required  $0<\delta<\kappa-\alpha$. In many situations, and certainly for Cartesian powers of fractal percolation, $d-\alpha$ equals the dimension of the random measure $\mu$, and thus $\kappa>\alpha$ simply means that $\dim \mu+\dim\eta_t>M$. When $\dim\mu+\dim\eta_t<M$, we can no longer expect $Y^{t}_n$ to converge to a continuous (or even finite) limit. However, we will still need to bound the size of the intersections in this case, in the sense of having good control on the growth rate of $Y_{n}^t$. For this, we use the following variant of Theorem \ref{thm:Holder-continuity} and \cite[Theorem 4.4]{ShmerkinSuomala14}.
\end{enumerate}
\end{rems}

\begin{thm}\label{thm:small_dimension_projections}
Suppose that $(\mu_n)$ is a WSDM with parameter $\delta$ satisfying
\eqref{H:size-parameter-space}--\eqref{H:codim-random-measure} from Theorem \ref{thm:Holder-continuity}, together with the following condition:
\begin{enumerate}
\renewcommand{\labelenumi}{\textup{(H\arabic{enumi})}}
\renewcommand{\theenumi}{\textup{H\arabic{enumi}}}
\setcounter{enumi}{5}
\item \label{H:finite_approx_family}
There are $\gamma>0$ and deterministic families $\Gamma_n\subset\Gamma$ with at most $\exp(O(n))$ elements, such that a.s. there is a random integer $K$ such that for each $n\ge K$,
\[
\sup_{t\in\Gamma}Y_n^t\le\sup_{t\in\Gamma_n}Y_n^t+O(2^{\gamma n})\,.
\]
\end{enumerate}
Suppose further that
\begin{equation}
0<\alpha+\delta-\kappa<\gamma\,.\label{eq:dim_small-1}\\
\end{equation}
Then, almost surely,
\[
\sup_{n\in\N,t\in\Gamma}2^{-\gamma n}Y_n^t<\infty\,.
\]
\end{thm}

\begin{proof}
The proof is similar to that of Proposition \ref{prop:probabilistic-technical}, and to Theorem 4.4 in \cite{ShmerkinSuomala14}, so we skip some details. Let $\Gamma_n\subset\Gamma$ be as in \eqref{H:finite_approx_family},  $Z_n=\max_{v\in\Gamma_n} |Y_{n+1}^v - Y_n^v|$ and
$\overline{Y}_n=\max_{t} Y_n^t$.

Firstly, we claim that almost surely there is $K_1\in\N$ such that, for some $\gamma_0<\gamma$,
\begin{equation}\label{eq:Zn_bound}
Z_n \le\sqrt{2^{\gamma_0 n}\overline{Y}_n}\,,
\end{equation}
for all $n\ge K_1$. Indeed, pick $\gamma_0<\gamma$ so that \eqref{eq:dim_small-1} continues to hold with $\gamma_0$ in place of $\gamma$. We can apply Lemma \ref{lem:large-deviation} to get, for each fixed $v\in \Gamma_n$,
\[
\PP_{\mathfrak{C}_n}(|Y_{n+1}^v - Y_n^v|> \tfrac12 2^{\gamma_0 n/2} \sqrt{Y_n^v} \,|\,\BB_n) \le \exp\left(-\Omega_{\e}\left(2^{(\kappa+\gamma_0-\alpha-\delta)n}\right)\right)\,,
\]
where $\varepsilon$  is as in the definition of WSDM.  Recalling \eqref{eq:dim_small-1}, we deduce that, provided $n\gg_\e 1$,
\[
\PP_{\mathfrak{C}_n}\left(Z_n>2^{\gamma_0 n/2} \sqrt{Y_n^v}\,|\,\mathcal{B}_n\right)< \exp(-2^{cn})
\]
for $c=\tfrac12(\kappa+\gamma_0-\alpha-\delta)>0$. From here the proof of \eqref{eq:Zn_bound} is concluded exactly as in the proof of Proposition \ref{prop:probabilistic-technical}.

In combination with \eqref{H:finite_approx_family}, the bound \eqref{eq:Zn_bound} implies that a.s. there is a random integer $K$ such that for all $n\ge K$,
\[
\overline{Y}_{n+1}\le \overline{Y}_n+\sqrt{2^{\gamma_0 n} \overline{Y}_n}
\]
Writing $K'= 2^{-\gamma K}Y_K$, we conclude by induction in $n\ge K$ that $\overline{Y}_{n}\le O(K')2^{\gamma n}$ for all $n\ge K$, as claimed.
\end{proof}

\section{Affine intersections and linear patterns}
\label{sec:affine-intersections}

\subsection{Intersections with affine planes}

In this section we start applying Theorem \ref{thm:Holder-continuity} to study the geometry of fractal percolation. Recall from Section \ref{sec:perco}, that $A_n$ denotes the union of the surviving cubes $Q\in\mathcal{Q}_n$. Moreover, $\nu_n=p^{-n}\mathcal{L}|_{A_n}$, $\nu=\nu^{\text{perc}(d,p)}=\lim_n \nu_n$ denotes the natural measure on the (surviving) fractal percolation set $A=A^{\text{perc}(d,p)}=\spt\nu=\cap_{n}\overline{A_n}$.
Furthermore, if $p>2^{-d}$, then a.s. $\dim_H A=\dim_B A=s=s(d,p)=d+\log_2 p$ and
\[
\underline{\dimloc}(\nu,x) = \liminf_{r\downarrow 0} \frac{\log\mu(B(x,r))}{\log r} = s
\]
for all $x\in A$ (and the limit exists for $\nu$-almost all $x$). From now on, $s$ will always refer to this number and,  even when not explicitly mentioned, we assume that the parameters $d,p$ have been fixed accordingly. Recall also that $\nu_n$ is uniformly spatially independent, and that $\nu_n(x)\le 2^{\alpha n}$ for $\alpha=d-s=-\log_2 p$ and all $x\in\R^d$. Indeed, $\nu_n(x) \in \{ 0, 2^{\alpha n}\}$ for all $x,n$.

We are interested in the Cartesian powers $(\nu_n^\text{perc})^m$ for $m\ge 2$ and, as pointed out in the introduction, these are neither spatially independent, nor martingales. In all our applications below, we will be able to work instead with the product $\mu_n$ of $m$ independent realizations of the fractal percolation measure, which is a martingale measure. Briefly, the reason is that given distinct $Q_1,\ldots,Q_m\in\mathcal{Q}_n^d$, then conditional on $Q=Q_1\times\cdots\times Q_m$ surviving, the restriction of $\nu^m$ to $Q$ has (up to rescaling and normalizing) the distribution of the product of $m$ independent fractal percolations. Regarding independence, we will show that $\mu_n$ is a WSDM with parameter $\delta$, depending on $m,d,p$ and the parametrized family $\{\eta_t\}_{t\in\Gamma}$.

From now on, we let $\nu_n^{(1)},\ldots, \nu_n^{(m)}$ be independent realizations of $\nu_n^{\text{perc}(d,p)}$, and write $\mu_n=\nu_n^{(1)}\times\cdots\times\nu_n^{(m)}$, and likewise for the limits $\nu^{(i)}$ and $\mu$. Our first result concerns the intersections of $\mu$ with affine planes. To formulate the result, we need a notion of angle between two affine subspaces. If $V\in\mathbb{G}_{M,k}, W\in\mathbb{G}_{M,\ell}$, we define the angle $0\le \angle(V,W)\le\pi/2$ by setting $\angle(V,W)=0$ if $\dim(V\cap W)>\max\{0,\dim V+\dim W-M\}$, and otherwise defining $\angle(V,W)$ as
\[
 \angle(V,W) = \inf\{ \angle(v,w): v\in V_0, w\in W_0 \}
\]
where $V_0,W_0$ are the orthogonal complements of $V\cap W$ inside $V,W$ respectively. Equivalently, $\angle(V,W)$ is the $j$-th smallest principal angle between $V$ and $W$, where $j=\dim(V\cap W)+1$.  This notion of angle measures how `transversal' the subspaces are. In particular, if $c>0$, then the requirement $\angle(V,W)\ge c$ can be seen as a uniform transversality condition. Such conditions will arise repeatedly in the sequel.

\begin{rems} \label{rem:angles}
\begin{enumerate}[(i)]
\item If $\angle(V,W)\neq 0$ then by elementary geometry we also have
\[
\sin\angle(V,W)=\inf_{x\in W\setminus V}\frac{\dist(x,V)}{\dist(x,V\cap W)}\,.
\]
\item The map $(V,W)\mapsto \angle(V,W)$ is continuous. In particular, if $\dim(V\cap W)=\max \{0,\dim V+\dim W-M\}$, then there exist $c>0$ and neighbourhoods $\mathcal{V},\mathcal{W}$ of $V,W$ such that $\angle(V',W')\ge c$ for all $V'\in\mathcal{V}, W'\in\mathcal{W}$. This observation will be used repeatedly.
\end{enumerate}
\end{rems}

If $V\in\mathbb{A}_{M,k}$ and $W\in\mathbb{A}_{M,\ell}$, we define the angle between $V$ and $W$ as the angle between the linear subspaces $V'\in\mathbb{G}_{M,k}$, $W'\in\mathbb{G}_{M,\ell}$ parallel to $V$ and $W$, respectively. For practical purposes, we also define $\angle(V,\{0\})=1$ for all $V\in\AA_{M,k}$. Recall that for an index set $I\subsetneq[m]$ and $j\in[m]\setminus I$, $i\in[d]$, we denote
\begin{align}
&H^{I}=\{(x_1,\ldots,x_m)\in(\R^{d})^m\,:\,x_j=0\text{ for all }j\in I\}\in\mathbb{G}_{md,(m-|I|)d}\,, \label{eq:trans_coordinate}\\
&H^{I,j,i}=\{x\in H^{I}\,:\,x_j^i=0\}\in\mathbb{G}_{md,(m-|I|)d-1}\text{ where }j\in[m]\setminus I, i\in[d]\,.\label{eq:trans_hyperplanes}
\end{align}

\begin{thm} \label{thm:independent-product-linear-int}
Let $\nu_n^{(i)}$, $i=1,\ldots,m$ be independent realizations of $\nu_n^{\text{perc}(d,p)}$,  and let $\mu_n=\nu_n^{(1)}\times\cdots\times \nu_n^{(m)}$. Let $\Gamma\subset\mathbb{A}_{md,k}$ such that for some $c>0$ and all $I\subsetneq[m]$, $j\in[m]\setminus I$, $i\in[d]$,
each $V\in\Gamma$ makes an angle $>c$ with the planes $H^{I}$, $H^{I,j,i}$.

If $s=s(d,p)>d-k/m$, then there is a deterministic $\gamma>0$ depending on $s,k,d,m$ such that a.s. there is $K<\infty$ for which
\begin{enumerate}
\item The sequence $Y_n^V:=\int_{V}\mu_n(x)\,d\mathcal{H}^k$ converges uniformly over all $V\in\Gamma$; denote the limit by $Y^V$.
\item $|Y_n^V-Y_n^W|\le K\, d(V,W)^\gamma$ for all $n$ and $V,W\in\Gamma$, and in particular the same holds for $Y^V$.
\end{enumerate}

If $s=s(d,p)\le d-k/m$, then almost surely
\begin{enumerate}
\setcounter{enumi}{2}
\item
$\sup_{n\in\N, V\in\Gamma}2^{-\gamma n}Y^{V}_{n}<\infty$,
\end{enumerate}
for any $\gamma>m(d-s)-k$.
\end{thm}

\begin{rems}
\begin{enumerate}[(i)]
\item As long as $\Gamma$ is compact, the transversality conditions in Theorem \ref{thm:independent-product-linear-int} are equivalent to assuming that
\begin{align}
\dim V\cap H^I&=\max\{0,k-m|I|\}\,,\label{eq:trans_dim_coordinate}\\
\dim V\cap H^{I,j,i}&=\max\{0,k-m|I|-1\}\,,\label{eq:trans_dim_hyper}
\end{align}
for all $I\subsetneq[m]$, $j\in[m]\setminus I$, $i\in[d]$. (Recall Remark \ref{rem:angles} (ii).) In particular, if $V\in\AA_{md,k}$ satisfies \eqref{eq:trans_dim_coordinate}, \eqref{eq:trans_dim_hyper}, then the transversality assumptions 
of Theorem \ref{thm:independent-product-linear-int} are valid when $\Gamma$ is a small neighbourhood of $V$.
\item The transversality with respect to the coordinate hyperplanes  \eqref{eq:trans_hyperplanes} is needed to establish the a priory H\"older continuity condition \eqref{H:Holder-a-priori}, while the transversality with respect to the planes \eqref{eq:trans_coordinate} is used to bound the dependency degree of $\mu_n$.  Depending on the value of $k$, $m$, $d$ one of these conditions may (or may not) imply the other.
\end{enumerate}
\end{rems}

The theorem is a rather direct application of Theorems \ref{thm:Holder-continuity} and \ref{thm:small_dimension_projections}. All the hypotheses in these theorems are fairly easy to check, except for the fact that $\mu_n$ is a WSDM with a suitably small parameter $\delta$. This will be verified by a joint probabilistic induction in $n$ and $m$. See also the survey \cite{ShmerkinSuomala16} for the proof of a special case highlighting the main ideas.

Many of our arguments will feature an induction on $m$; in order to set it up we need some further notation. Given $V\in\mathbb{A}_{md,k}$, we let
\[
V'_{i,t} = \pi^i(V\cap\{x\in(\R^d)^m\,:\,x_i=t\}), \quad (i\in[m],\, t\in [0,1]^d)\,.
\]
(Recall the notation from Section \ref{sec:notation}.) For $\Gamma\subset\mathbb{A}_{md,k}$, we define
\begin{equation} \label{eq:def-R}
\mathcal{R}(\Gamma) =  \{ V'_{i,t}: V\in\Gamma, i\in[m], t\in\R^d \}.
\end{equation}
For $2\le p\le m-1$, we inductively define $\mathcal{R}^p(\Gamma)=\mathcal{R}(\mathcal{R}^{p-1}(\Gamma))$.

\begin{rem}\label{rem:transversality}
  In terms of the families $\mathcal{R}(\Gamma)$, the transversality assumptions in Theorem \ref{thm:independent-product-linear-int} are equivalent to the claim that for each $p$ such that $\mathcal{R}^p(\Gamma)$ is nontrivial, all
$V\in\mathcal{R}^p(\Gamma)$ make an angle $\ge c$ with the planes of the type $\{x_j=0\}$, $\{x_{j}^i=0\}$ for all $j\in[m-p]$, $i\in[d]$.
\end{rem}

In the course of the proof of Theorem \ref{thm:independent-product-linear-int}, we will require the following tail bound for $\overline{Y}_n$, whose proof may be gleaned from the proofs of Theorems \ref{thm:Holder-continuity} and \ref{thm:small_dimension_projections}.
\begin{lemma}\label{lemma:AWSD_product}
Under the assumptions of Theorem \ref{thm:independent-product-linear-int}, let
\[
\delta>0,\quad \lambda>\max\{0,m(d-s)-k+\delta\}\,,
\]
and denote $\overline{Y}_n=\sup_{W\in\Gamma} Y_n^{W}$. Suppose $\Psi(n)$ is a bound for the dependency degree
as in the definition of WSDM. Let $\mathfrak{F}\in\mathcal{B}_n$ such that $\Psi(n)\le C 2^{n\delta}$ and $\overline{Y}_n\le C^4 2^{\lambda n}$ on $\mathfrak{F}$, where $C<\infty$ is a sufficiently large constant depending only on $\lambda$.

Then
\[
\mathbb{P}_\mathfrak{F}\left(\overline{Y}_{n+1}\ge C^4 2^{\lambda(n+1)}\mid \mathcal{B}_n\right)=O(\exp(-C 2^{\Omega(n)})),
\]
with the $O(\cdot),\Omega(\cdot)$ are independent of $C$ and $n$.
\end{lemma}

\begin{proof}
In the course of the proof, the $O(\cdot),\Omega(\cdot)$ constants independent of $n$ or $C$ (but allowed to depend on any other parameters).

To begin with, we verify that for each dyadic cube $Q\in\mathcal{Q}^{md}_{n}$, the map $V\mapsto\mathcal{H}^k(V\cap Q)$, $\Gamma\to\R$ is Lipschitz, where the Lipschitz constant is independent of $n$ and $Q$ (it only depends on the transversality constant $c$). To see this, fix $V,W\in\Gamma$, denote $\varepsilon=d(V,W)$ and note that on $\AA_{md,k}$, our metric $d$ is equivalent to the Hausdorff metric on $[0,1]^d$. It follows that
\begin{equation}\label{eq:P_Wint}
\pi_W\left(V\cap Q\setminus\partial Q(O(\varepsilon))\right)\subset W\cap Q\,,
\end{equation}
where $\pi_W$ denotes the orthogonal projection onto $W$. On the other hand, since $V$ forms an angle $\ge c$ with the faces of $Q$,
\begin{equation}\label{eq:bndry}
\mathcal{H}^k\left(V\cap(\partial Q(O(\varepsilon)))\right)=O(\varepsilon)\,.
\end{equation}
Combining \eqref{eq:P_Wint}, \eqref{eq:bndry}, and using that orthogonal projection does not increase $\mathcal{H}^k$-measure, we get
\[\mathcal{H}^k(W\cap Q)\ge\mathcal{H}^k(V\cap Q)-O(\varepsilon)\ge \mathcal{H}^k(V\cap Q)-O(d(V,W))\,.\]
By symmetry, we end up with the estimate
\begin{equation}\label{eq:Lipschitz-a-priori}
|\mathcal{H}^k(W\cap Q)-\mathcal{H}^k(V\cap Q)|=O(d(V,W))\,,
\end{equation}
as required. Since each $V\in\Gamma$ intersects at most $O(2^{nk})$ cubes in $\mathcal{Q}_n$, and $\mu_n(x) \le 2^{m(d-s)n}$ for all $x$, from \eqref{eq:Lipschitz-a-priori} we get the (crude but sufficient) estimate
\begin{equation}\label{Nn-Holder-product}
\sup_{V\neq W\in\Gamma} |Y_{n}^V-Y_{n}^W|=O\left(d(V,W)2^{(m(d-s)+k)n}\right)\,,
\end{equation}
for all $n\in\N$.

Let $\Gamma_{n+1}\subset\Gamma$ be $\Omega(2^{n(\lambda-m(d-s)-k)})$ dense. By restricting $\Gamma$ to those planes that hit a fixed neighbourhood of the unit cube, we may assume that $\Gamma$ is bounded. It is then easily seen that $\Gamma$ has finite upper box dimension in the metric on $\mathbb{A}_{md,k}$ defined in Section \ref{sec:notation}. Thus, we may assume that $\Gamma_{n+1}$ has  $O(\exp(O(n))$ elements. Now \eqref{Nn-Holder-product} implies
\begin{equation}\label{Y_Vbound}
\sup_{V\in\Gamma}Y^{V}_{n+1}\le \sup_{V\in\Gamma_{n+1}}Y^{V}_{n+1}+2^{\lambda n}\,.
\end{equation}
For each $V\in\Gamma_{n+1}$, we can use Lemma \ref{lem:large-deviation} in a similar way to \eqref{eq:application-large-dev} to estimate
\begin{align*}
  \PP_{\mathcal{F}}\left(|Y_{n+1}^V - Y_n^V|> C^3 2^{\lambda n} \mid \mathcal{B}_n\right)&\le\exp\left(-\Omega( C)2^{n(\lambda+k-m(d-s)-\delta}\right)\\
 &\le \exp\left(-C 2^{\Omega(n)}\right)\,.
\end{align*}
Thus, if $C$ is so large that $C^4 2^\lambda > C^4 + C^3 + 1$, then \eqref{Y_Vbound} and the fact that $\Gamma_{n+1}$ has $O(\exp(O(n))$ elements, imply the claim.
\end{proof}

\begin{proof}[Proof of Theorem \ref{thm:independent-product-linear-int}]
It is clear that $\mu_n$ is a martingale measure in $\R^{md}$, and that \eqref{H:dim-deterministic-measures}, \eqref{H:codim-random-measure} hold with $m(d-s)$ in place of $\alpha$, and with Frostman exponent $k$. Further,
\eqref{H:Holder-a-priori} holds for $\theta=m(d-s)+k$ and $\gamma_0=1$ as explained in \eqref{Nn-Holder-product}. Also \eqref{H:size-parameter-space} holds, and \eqref{H:finite_approx_family} is valid for any $\gamma>0$, since $\Gamma$ has finite upper box dimension as explained in the proof of Lemma \ref{lemma:AWSD_product}. Thus, all the claims follow from Theorems \ref{thm:Holder-continuity} and \ref{thm:small_dimension_projections} if we can show that
\begin{equation}\label{AWSDclaim}
\mu_n\text{ is a WSDM with parameter }\delta
\end{equation}
for some $\delta<k-m(d-s)$, if $k-m(d-s)>0$, and for any $\delta>0$, if $k-m(d-s)\leq 0$.

Without loss of generality, we may assume that $\Gamma$ is translation invariant, so that $V+z\in\Gamma$ for all $V\in\Gamma$ and all $z\in\R^d$ for which $(V+z)\cap[0,1]^{md}\neq\emptyset$.

Let us start by defining dependency graphs for $\mu^{V}_{n+1}(Q)$, $Q\in\mathcal{Q}^{md}_n$. We observe that
if $Q\in\mathcal{Q}_{n}^{md}$ and $\mathcal{Q}'\subset\mathcal{Q}^{md}_n$ are such that $\pi_j(Q)\neq \pi_{j'}(Q')$ for all $j,j'\in[m]$ and all $Q'\in\mathcal{Q}'$ then, conditional on $\mathcal{B}_n$, $\mu_{n+1}^V(Q)$ is independent of $\{\mu_{n+1}^V(Q')\}_{Q'\in\mathcal{Q}'}$. Hence the graph defined (for each $V\in\Gamma$) by drawing an edge between $Q,Q'\in\mathcal{Q}_n^{md}$ if and only if $Q\cap V\neq\varnothing\neq Q\cap V'$ and $\pi_j(Q)=\pi_{j'}(Q')$ for some $j,j'\in [m]$, is a dependency graph for $\{\mu_{n+1}^V(Q)\}_{Q\in\mathcal{Q}^{md}_n}$.

If $k\le d$, the transversality with respect to the hyperplanes \eqref{eq:trans_hyperplanes} implies that for all $j\in[m]$, $Q\in\mathcal{Q}^{d}_n$, and $V\in\Gamma$, we have
\[\diam(V\cap\pi_j^{-1}(Q))=O_c(1)\,.\]
Hence, given $Q\in\mathcal{Q}^{md}_n$, there can be at most $O_{c,m}(1)$ cubes $Q'\in\mathcal{Q}^{md}_n$ such that $\pi_j(Q)=\pi_{j'}(Q')$ for some $j,j'\in[m]$. Consequently, the dependency degree of $\{\mu_{n+1}^V(Q)\}_{Q\in\mathcal{Q}^{md}_n}$ is uniformly bounded, so in this case \eqref{AWSDclaim} holds for any $\delta>0$.

For the rest of the proof, we assume that $k>d$ and we prove \eqref{AWSDclaim} by induction on $m$ (with the data $\Psi(n), \mathfrak{C}_n,\e_n$ to be specified in the course of the proof). If $m=1$, the claim is clearly true since then $(\mu_n)$ is (uniformly) spatially independent, so that $\Psi(n)=O(1)$, recall Definition \ref{def:usi}. Suppose, then, the claim holds for $m-1\ge 1$.

We apply the induction hypothesis to the intersection of the measures
\[
\wt{\mu}_n = \nu_n^{(1)}\times\cdots\times \nu_n^{(m-1)}
\]
with the elements of $\mathcal{R}(\Gamma)$. For $V'\in\mathcal{R}(\Gamma)$, consider
\begin{equation*}
\widetilde{Y}_{n}^{V'}=\int_{V'} \wt{\mu}_n \,d\mathcal{H}^{k-d}\,,
\end{equation*}
and let $Z_n=\sup_{V'\in\mathcal{R}(\Gamma)}\widetilde{Y}_{n}^{V'}$.
We first claim that there is a constant $C>0$ (depending only on $m$) such that
\begin{equation}\label{Delta_n_bound}
\mathcal{DI}_n \le \Psi(n):=CZ_n2^{n(k-d-(m-1)(d-s))} \,.
\end{equation}
where $\mathcal{DI}_n$ is the dependency degree of $\mu_n$ with respect to $\Gamma$.

Suppose that $V\in\Gamma$, $Q\in\mathcal{Q}_n^{md}$ and $\mathcal{Q}\subset\mathcal{Q}^{md}_n$ are such that $Q\subset A_n^{(1)}\times\cdots\times A_n^{(m)}$, $V\cap Q\neq\emptyset$ and for all $Q'\in\mathcal{Q}$ it holds that $Q'\subset (A_n)^m$, $V\cap Q'\neq\emptyset$, and also  $\pi_j(Q)=\pi_{j'}(Q')$ for some $j,j'\in[m]$. If we can show that
\begin{equation}\label{eq:pigeon1}
|\mathcal{Q}|=O_m(Z_n2^{n(k-d-(m-1)(d-s))})\,,
\end{equation}
then \eqref{Delta_n_bound} follows by virtue of the dependency graphs defined above.

We note that it suffices to show \eqref{eq:pigeon1} in the case when $j$ is fixed and $j'=j$ for all $Q'\in\mathcal{Q}$. Indeed, since there are only $m^2$ possible pairs $(j,j')$, there is a subset $\mathcal{Q}'\subset\mathcal{Q}$ with $|\mathcal{Q}'|\ge |\mathcal{Q}|/m^2$ and $j,j'\in\{1,\ldots, m\}$ such that the above conditions are satisfied with a fixed $j,j'$ for all $Q'\in\mathcal{Q}'$. Furthermore, replacing $Q$ by any $Q'\in\mathcal{Q}'$ (and $\mathcal{Q}'$ by $\mathcal{Q}'\setminus\{Q'\}$), we may assume that $j=j'$.
Thus, in the following we assume that $Q,\mathcal{Q}$ are as above and $j=j'$ is fixed.

Without loss of generality, we may assume that $V$ contains the origin so that $V\in \mathbb{G}_{md,k}$ and furthermore, that $0\in\pi_j(Q)$ (this is just to simplify notation). Let $V_0=V\cap\{x_j=0\}$, let $V^\perp\in \mathbb{G}_{md,md-k}$ be the orthogonal complement of $V$, and let $\widetilde{V}=V\cap V_0^\perp\in \mathbb{G}_{md,d}$ the orthogonal complement of $V_0$ relative to $V$.
Then, the map $\pi_j|_{\widetilde{V}}\colon \widetilde{V}\to\pi_j(\widetilde{V})$ is $O(1/c)$-Bi-Lipschitz, where $c>0$ is the constant appearing in  condition \eqref{eq:trans_coordinate}. Indeed, if $x,y\in \widetilde{V}$ and $x\neq y$ then, using Remarks \ref{rem:angles}(i),
\begin{equation}\label{eq:bilip}
\sin c\le\frac{\dist(x-y,\{x_j=0\})}{\dist(x-y,V\cap\{x_j=0\})}=\frac{|\pi_j(x)-\pi_j(y)|}{|x-y|}\le 1
\end{equation}
since $V$ forms an angle $\ge c$ with the plane $\{x_j=0\}\in \mathbb{G}_{md,(m-1)d}$.

Let us denote $U=\cup_{Q\in\mathcal{Q}}Q$,
\begin{align*}
B_1&=B(0,\sqrt{md-k}\, 2^{-n})\subset V^\perp,\\
B_2&= B(0,\sqrt{d}O(1/c)2^{-n})\subset\widetilde{V}\,,
\end{align*}
and
\begin{align*}
A'_n &= A_n^{(1)}\times\cdots\times A_n^{(m)}\,,\\
\wt{A}'_n &= A_n^{(1)}\times\cdots\times A_n^{(m-1)}\,.
\end{align*}
Using Fubini's theorem, we arrive at the estimate
\begin{equation}\label{eq:volume_argument}
\begin{split}
|\mathcal{Q}| 2^{-nmd}&=\mathcal{L}^{md}\left(U\right)\\
&=\int\limits_{y\in B_1}\int\limits_{z\in B_2}\mathcal{H}^{k-d}\left(U\cap (V+y)\cap\{x_j=\pi_j(z)\}\right)\,d\mathcal{L}^{d}(z)\,d\mathcal{L}^{md-k}(y)\\
&\le O(2^{-n((m+1)d-k)})\sup\limits_{y\in\R^{md},z\in\R^d}
\mathcal{H}^{k-d}\left(A'_n\cap(V+y)\cap\{x_j=\pi_j(z)\}\right)\,,
\end{split}
\end{equation}
with the $O$ constant depending on $c,d,m,k$. Noting that \begin{align*}
\mathcal{H}^{k-d}\left(A'_n \cap(V+y)\cap\{x_j=\pi_j(z)\}\right)=\mathcal{H}^{k-d}\left(\wt{A}'_n \cap V'\right)=2^{n(m-1)(s-d)}\widetilde{Y}_n^{V'}\,,
\end{align*}
for $V'=\pi^j((V+y)\cap\{x_j=\pi_j(z)\})\in\mathcal{R}(\Gamma)$, \eqref{eq:volume_argument} yields \eqref{eq:pigeon1}.

To complete the proof of \eqref{AWSDclaim} we still need to verify that $\Psi(n)$ as defined in \eqref{Delta_n_bound}
fulfils the conditions in the definition of WSDM.

Write $\xi=k-d-(m-1)(d-s)$ for simplicity. If $\xi>0$, we know by the induction assumption (recall Remark \ref{rem:transversality}) that $\wt{\mu}_n$ is a WSDM (with respect to $\mathcal{R}(\Gamma)$) with parameter $\widehat{\delta}$, for some $\widehat{\delta}<\xi$, while if $\xi\le 0$, then this holds for any $\widehat{\delta}>0$. Choosing a suitable
\begin{equation}\label{eq:lambdadef}
\lambda>\max\{0,\widehat{\delta}-\xi\}\,,
\end{equation}
and considering different cases depending on the signs of $k-m(d-s)=\xi+s$ and $\xi$, we can make sure that for  \begin{equation}\label{eq:deltadef}
0<\delta:=\lambda+\xi\,,
\end{equation}
we have $\delta<k-m(d-s)$ if $s>d-k/m$, while $\delta$ can be arbitrarily small if $s\le d-k/m$.

Having such parameters $\widehat{\delta},\lambda,\delta$ fixed, given $\varepsilon>0$, there are $C<\infty$ and events $\widehat{\mathfrak{C}}_{n}\subset\{\widehat{\Psi}(n) \le C 2^{\widehat{\delta} n}\}$ such that $\PP(\widehat{\mathfrak{C}}_{1})>1-\varepsilon_1$ and
$\mathbb{P}\left(\widehat{\mathfrak{C}}_{n+1}\mid\mathcal{B}_n\right)\ge 1-\varepsilon_{n+1}$ on $\widehat{\mathfrak{C}}_{n}$, where $\sum_n \varepsilon_n<\varepsilon$.
Here $\widehat{\Psi}(n)$ is the bound for the dependency degree for $\wt{\mu}_n$ and $V'\in\mathcal{R}(\Gamma)$ as in the definition of WSDM.

Let us define $\mathfrak{E}_{n}=\widehat{\mathfrak{E}}_{n}\cap\{Z_n\le C^4 2^{\lambda n}\}$. Thanks to \eqref{eq:lambdadef}, and making $C$ larger if needed, we may apply Lemma \ref{lemma:AWSD_product} to $\wt{\mu}_n$ and $\mathcal{R}(\Gamma)$. We deduce that
\begin{equation}\label{Delta_n_j_bound}
\mathbb{P}_{\mathfrak{C}_n}\left(\mathfrak{C}_{n+1}\mid \mathcal{B}_n\right)\ge 1-\varepsilon_{n+1}-O\left(\exp(-C 2^{\Omega(n)})\right)\,.
\end{equation}
Since $\sum_{n} \varepsilon_{n}+O(\exp(-C 2^{\Omega(n)}))$ can be made as small as we wish by making $\varepsilon$ small enough, and $C$ large enough, and as
\[
\Psi(n)=C Z_n2^{-n\xi}\le C^5 2^{n(\lambda-\xi)}=C^5 2^{\delta n}
\]
on $\mathfrak{C}_{n}$, we conclude that $(\mu_n)$ is a WSDM with a parameter $\delta$ obeying the desired bounds.
\end{proof}

\begin{rem}
All of our applications of Theorem \ref{thm:independent-product-linear-int} deal with the case $s>d-k/m$. However, in the induction step, the case $s\le d-k/m$ is also needed.
\end{rem}

Theorem \ref{thm:independent-product-linear-int} gives non-trivial information only if there is a positive probability that $Y^V>0$ for some $V\in\Gamma$. We next show that, for a fixed $V$ for which it is a priori possible to have $\PP(Y^V>0)>0$, this is indeed true.  In the spatially independent case, this was proved in \cite[Lemma 3.6]{ShmerkinSuomala14} via a tail estimate. Here we use the second moment method.
\begin{lemma} \label{lem:Yt-survives}
In the setting of Theorem \ref{thm:independent-product-linear-int}, suppose that $s>d-k/m$. Then, for each $V\in\Gamma$ with $V\cap]0,1[^{md}\neq\varnothing$,
\[\PP(Y^V>0)>0\,.\]
\end{lemma}
\begin{proof}
It is enough to show that
\[
Y_n^V = \int_{V} \mu_n \, d\mathcal{H}^k
\]
is an $L^2$ bounded martingale, as the claim then follows from the martingale convergence theorem. Since $Y_n^V$ is clearly a martingale, it remains to establish $L^2$ boundedness.

Let $\mathcal{Q}_n^{md}(V)$ denote the cubes in $\mathcal{Q}_n^{md}$ that hit $V$.  To begin, we estimate
\begin{align}
\EE(Y_n^2) &= 2^{2n(m(d-s))} \EE\left(\left(\mathcal{H}^k(V\cap A_n)\right)^2\right)  \nonumber \\
&=2^{2n(m(d-s))}\sum_{Q,Q'\in\mathcal{Q}_n^{md}(V)} \PP(Q\cup Q'\subset A_n) \mathcal{H}^k(V\cap Q)\mathcal{H}^k(V\cap Q')\nonumber  \\
&\le O(1) 2^{2n(m(d-s)-k)}\sum_{Q,Q'\in\mathcal{Q}_n^{md}(V)} \PP(Q\in A_n)\PP(Q'\in A_n|Q\in A_n) \nonumber \\
&=O(1)2^{2n(m(d-s)-k)} 2^{n(m(s-d))} \sum_{Q\in\mathcal{Q}_n^{md}(V)} \sum_{Q'\in\mathcal{Q}_n^{md}(V)} \PP(Q'\in A_n|Q\in A_n) \nonumber \\
&=O(1)2^{n(m(d-s)-k)}  \max_{Q\in\mathcal{Q}_n^{md}(V)}   \sum_{Q'\in\mathcal{Q}_n^{md}(V)} \PP(Q'\in A_n|Q\in A_n)\,. \label{eq:L2-estimate}
\end{align}
Hence, we fix $Q\in\mathcal{Q}_n^{md}(V)$ and set out to estimate
\[
\zeta(Q) :=  \sum_{Q'\in\mathcal{Q}_n^{md}(V)} \PP(Q'\in A_n|Q\in A_n)\,.
\]
For $r=0,1,\ldots,n$, let $Q_r$ be the cube in $\mathcal{Q}_r^{md}$ containing $Q$. We will estimate each
\[
\zeta_r(Q) := \sum_{Q'\in\mathcal{Q}_n^{md}(V), Q'\subset Q_r\setminus Q_{r+1}} \PP(Q'\in A_n|Q\in A_n)
\]
separately. Given $\ell\in \{0,1,\ldots,m\}$, let
\[
\mathcal{T}_{r,\ell}(Q) = \{ Q'\in\mathcal{Q}_n^{md}(V): Q'\subset Q_r\setminus Q_{r+1},|\{ j\in [m]:\pi_j(Q)=\pi_j(Q') \}|=\ell\}.
\]
We claim that
\begin{equation} \label{eq:count-dependent-cubes}
|\mathcal{T}_{r,\ell}(Q)| = O(1)\max\left(1, 2^{(n-r)(k-\ell d)}\right).
\end{equation}
Indeed, fix a subset $I\subset[m]$ with $|I|=\ell$. Let $\wt{\pi}(x)=(\pi_j(x):j\in I)$, $\R^{md}\to\R^{d\ell}$. By the transversality assumption, $V$ makes an angle $\Omega(1)$ with $\wt{\pi}^{-1}(0)$ (this can be seen from dimensional considerations). Hence, $V(\sqrt{d} 2^{-n})\cap\wt{\pi}^{-1}(\wt{\pi}(Q))\cap Q_r$ can be covered by a cube of side-length $O(2^{-n})$, if $k-d\ell \le 0$, and by a parallelepiped that has $k-\ell d$ sides of length $O(2^{-r})$ and the remaining $md-(k-\ell d)$ sides of length $O(2^{-n})$, otherwise. Hence, if there are $M$ cubes $R\in\mathcal{Q}_n^{md}$ such that $R\cap V\neq \varnothing$, $R\subset Q_r$ and $\wt{\pi}(R)=\wt{\pi}(Q)$ then, comparing volumes, $M=O(1)$ if $k-d\ell \le 0$, and otherwise
\[
M 2^{-nmd} \le O(1) 2^{-r(k-\ell d)} 2^{-n(md+\ell d-k)}\,.
\]
Since there are finitely many maps $\wt{\pi}$ to consider, we get \eqref{eq:count-dependent-cubes}.

On the other hand, if $Q'\in \mathcal{T}_{r,\ell}(Q)$, then
\begin{equation} \label{eq:conditional-prob-surviving}
\PP(Q'\in A_n|Q\in A_n) = 2^{(s-d)(n-r)(m-\ell)}.
\end{equation}
Combining \eqref{eq:count-dependent-cubes} and  \eqref{eq:conditional-prob-surviving}, we get
\begin{align*}
\zeta_r(Q) &= \sum_{\ell=0}^{m-1} \left( O(1)\max\left(1, 2^{(n-r)(k-\ell d)}\right)\right) \left(2^{(s-d)(n-r)(m-\ell)}\right)\\
&= O(1) \sum_{\ell=0}^{\lfloor k/d \rfloor} 2^{(n-r)(k-\ell d+(s-d)(m-\ell))} + O(1) \sum_{\ell=\lceil k/d\rceil}^m 2^{(n-r)(s-d)(m-\ell)}\\
&= O(1) 2^{(n-r)(k+(s-d)m)} + O(1)\\
&=O(1) 2^{(n-r)(k+(s-d)m)}\,,
\end{align*}
using that $k+(s-d)m\ge 0$ in the last line. We deduce that
\[
\zeta(Q) =\sum_{r=0}^n \zeta_r(Q) = O(1) 2^{n(k+(s-d)m)}.
\]
Plugging this into \eqref{eq:L2-estimate}, we conclude that $\EE(Y_n^2) = O(1)$, as desired.
\end{proof}

As we have remarked, our ultimate goal is to study Cartesian powers of the same fractal percolation set or measure, rather than independent copies. The following corollary will help us in achieving this.
\begin{cor}\label{cor:indep-products}
Under the same hypotheses of Theorem \ref{thm:independent-product-linear-int}, the following holds.

Let $n_0\in\N$ and $U$ be a finite union of cubes $Q\in\mathcal{Q}^{md}_{n_0}$, such that $U\cap\Delta=\varnothing$, where
\[\Delta=\{x\in(\R^d)^m\,:\,x_i=x_j\text{ for some }i\neq j\}\]
is the union of all the diagonals.

Then the conclusions of Theorem \ref{thm:independent-product-linear-int} also hold for
\[
Y_n^V=\int_{U\cap V}\nu_n^m\,d\mathcal{H}^k.
\]

Furthermore, if $s>d-k/m$, $V\in\Gamma$ and $V\cap U^\circ\neq\varnothing$, then $\PP(Y^V>0)>0$. Recall that $U^\circ$ denotes the interior of $U$
\end{cor}

\begin{proof}
For each $Q\in\mathcal{Q}^{md}_{n_0}$, $Q\subset U$, either $Q\cap A_{n_0}^m=\varnothing$, or  $\mu_n|_Q$ is (up to scaling and renormalizing) a product of independent fractal percolation measures on the projections $\pi_i(Q)\subset\R^d$. Hence we can condition on $\mathcal{B}_{n_0}$ and apply Theorem \ref{thm:independent-product-linear-int} to each of these restrictions.

Similarly, the proof of Lemma \ref{lem:Yt-survives} applies to $\mu|_U$, yielding that $\PP(Y^V>0)>0$ if $V\cap U^\circ\neq\varnothing$.
\end{proof}

\subsection{Finite patterns in fractal percolation}

We now apply Theorem \ref{thm:independent-product-linear-int} and Corollary \ref{cor:indep-products} to prove the existence of homothetic copies of finite sets in the fractal percolation limit set. Recall that $S'\subset\R^d$ is a homothetic copy of $S\subset\R^d$ if there are $\lambda>0$ and $z\in\R^d$, such that $S'=F_{\lambda,z}(S)$, where $F_{\lambda,z}(x)=\lambda(x+z)$.
\begin{cor}\label{cor:patterns}
If $m\in\N_{\ge 2}$ and $s>d-(d+1)/m$, then a.s., the fractal percolation limit set $A=A^{\text{perc}(d,p)}$ contains homothetic copies of all $m$-point sets $\{a_1,\ldots,a_m\}\subset\R^d$.
\end{cor}

\begin{proof}
Let
\[
\Gamma=\{T=(t_1,\ldots,t_{m-1})\in(\R^d)^{m-1}\,:\, 0\neq t_i^k\neq t_j^k\text{ if }i\neq j\in[m-1],\,k\in[d]\}.
\]
and denote $V_T=\{(y,\ldots,y)+\lambda(t_1,\ldots, t_{m-1},0)\,:\, y\in\R^{d}\,,\lambda\in\R\}\in\mathbb{G}_{md,d+1}$. Our goal is to apply Corollary \ref{cor:indep-products} to the family $\{V_T\}_{T\in\Gamma}$, and $\mu_n=(\nu_n^{\text{perc}(d,p)})^m$. Fix $T_0\in\Gamma$. Then, $V_{T_0}$ has transversal intersections with all coordinate planes $\{x_i^k=0\}$, $\{x_i=0\}$ for all $i\in[m], k\in[d]$. Indeed, it is clear that $V_{T_0}\cap\{x_i^k=0\}\in\AA_{md,d}$ while, denoting $t_m=0$,
\[
V_{T_0}\cap\{x_i=0\}=\{(\lambda (t_1-t_i),\ldots,\lambda (t_m-t_i))\,:\,\lambda\in\R\}
\]
is a line. Furthermore, if $i,j\in[m-1]$, $j\neq i$ and $k\in[d]$, then
\begin{equation}\label{eq:trans_d+1}
V_{T_0}\cap\{x_j=0, x_i^k=0\}=\{0\}
\end{equation}
thanks to the assumption $0\neq t_{i}^k\neq t_{j}^k$ for $i\neq j$. Note that these are the only nontrivial transversality conditions to be checked since $\dim V_{T_0}=d+1$.

For each fixed $T_0\in\Gamma$, we may find $n_0\in\N$ and disjoint cubes $Q_1,\ldots, Q_m\in\mathcal{Q}^{d}_{n_0}$ such that $V_{T_0}$ intersects the interior of $Q=Q_1\times\ldots\times Q_m\in\mathcal{Q}^{md}_{n_0}$ and $Q\cap V_{T_0}^-=\varnothing$, where
\[
V_{T_0}^- = \{(x,\ldots,x)+\lambda(T_0,0)\,:\, x\in\R^{d}\,,\lambda\le  0\}.
\]
Applying Corollary \ref{cor:indep-products}, we conclude that $Y_n^T=\int_{Q\cap V_T} \mu_n\,d\mathcal{H}^{d+1}$ converges in a neighbourhood $U_0$ of $T_0$, and that $T\mapsto Y^T=\lim_{n}Y^T_{n}$ is H\"older continuous on $U_0$. Furthermore, $\PP(Y^{T_0}>0)>0$, by the last claim of Corollary \ref{cor:indep-products}. Making $U_0$ smaller if needed, we can then ensure that $Q\cap V_T^{-}=\varnothing$ for all $T\in U_0$, and
\begin{equation}\label{eq:YTpositive}
\PP(Y^T>0\text{ for all }T\in U_0)>0\,.
\end{equation}
If $Y^T>0$, then $V_T\cap A^m\cap Q\neq\varnothing$. In other words, there are $x\in A\cap Q$ and $\lambda>0$ such that $x+\lambda t_i\in A$ for all $i=1,\ldots,m-1$.

The property of containing a homothetic copy of $\{t_1,\ldots, t_{m-1},0\}$ for all $t=(t_1,\ldots,t_{m-1})\in U_0$ is invariant under homotheties and passing to supersets. Corollary \ref{cor:0-1} then implies that a.s., $A$ contains a homothetic copy of $\{t_1,\ldots,t_{m-1},0\}$ for all $(t_1,\ldots,t_{m-1})\in U_0$. Since $\Gamma$ is $\sigma$-compact, we may cover $\Gamma$ by countably many such neighbourhoods $U_0$, and conclude that a.s., $A$ contains a homothetic copy of any $m$-element set $\{a_1,\ldots, a_m\}\subset\R^d$ with $a_i^k\neq a_j^k$ for all $i\neq j$ and $k\in[d]$.

To show that $A$ contains also homothetic copies of those $\{a_1,\ldots,a_m\}$ with $a_i^k=a_j^k$ for some $i\neq j$, $k\in[d]$, we have to modify the argument slightly. This is due to the fact that if there are coincidences $t_{i}^k=t_{j}^k$ for some $i\neq j$, $k\in[d]$, or if some $t_{i}^k=0$, then $\mathcal{R}(\{V_T\})$ contains lines parallel to some hyperplane $\{x_{i}^j=0\}\in\mathbb{A}_{(m-1)d,(m-1)d-1}$.

Note that in the proof of Theorem \ref{thm:independent-product-linear-int}, the transversality with respect to the coordinate hyperplanes is only used to obtain the a priori H\"{o}lder bound via \eqref{eq:Lipschitz-a-priori}. The key to solving this problem is the observation that in order for the proof of Theorem \ref{thm:independent-product-linear-int} to work, we may freely use any metric on $\mathcal{R}(\Gamma)$ as long as $\mathcal{R}(\Gamma)$ has finite box dimension with respect to this metric and \eqref{eq:Lipschitz-a-priori} remains valid. Furthermore, instead of the whole of $\mathcal{R}(\Gamma)$, we may consider each
\[\mathcal{R}_i(\Gamma):=\{V'_{i,t}\,:\,V\in\Gamma, t\in[0,1]^d\}\]
separately, and even use a different metric on each $\mathcal{R}_i(\Gamma)$, $i\in[m]$. Note that in Theorem \ref{thm:independent-product-linear-int}, $\mathcal{R}(\Gamma)=\cup_{i\in[m]}\mathcal{R}_i(\Gamma)$, but in the proof, the induction assumption on $\widetilde{Y}^{V'}_n$ is applied to one $\mathcal{R}_i(\Gamma)$ at a time.

Now to the details. We consider different coincidence classes separately: let
\[\mathcal{N}\subset\{(i,j,k)\,:\,i,j\in[m],k\in[d],i\neq j\}\]
and (still denoting $t_m=0$), define
\begin{align*}
\Gamma_\mathcal{N}=\{T=(t_1,\ldots,t_{m-1})\in(\R^d)^{m-1}\,:\, t_i^k=t_j^k\text{ if }(i,j,k)\in\mathcal{N}\\
\text{ and }t_i^k\neq t_j^k\text{ if }i\neq j\text{ and }(i,j,k)\notin\mathcal{N}\}\,.
\end{align*}

Fix $\mathcal{N}$ and $i\in[m]$. Given
$T\in\Gamma_\mathcal{N}$, and $V=V_T$, the projections
\[V'_{i,c}=\pi^i(\{V_T\cap\{x_i=c\})\]
are parallel to
\[H_i=\pi^i\left(\{x\in\R^{d(m-1)}\,:\,x_j^k=0\text{ whenever } (i,j,k)\in\mathcal{N}\}\right)\]
and transversal with respect to the coordinate hyperplanes orthogonal to $H_i$.

To define a new metric on $\mathcal{R}_i(\Gamma_{\mathcal{N}})$, let $d_2$ denote the dyadic metric on $H_i^\perp$. For $x,y\in H_i^\perp$, this is defined as the side length of the largest dyadic cube in $\mathcal{Q}^{\dim H_i^\perp}$ that contains both $x$ and $y$ (or $0$ if $x=y$). Note that $d_2$ is indeed a metric since the families $\mathcal{Q}_n$ consist of pairwise disjoint dyadic cubes.
The argument now carries on by defining a metric on $\AA_{d(m-1),1}$ as
\[d(\ell,\ell')=\max\{d_2(\widetilde{\pi}_1(\ell),\widetilde{\pi}_1(\ell')), d_1(\widetilde{\pi}_2(\ell),\widetilde{\pi}_2(\ell'))\}\,,\]
where $\widetilde{\pi}_2$ denotes the orthogonal projection onto $H_i$, $\widetilde{\pi}_1$ is the orthogonal projection onto $H_i^\perp$, and $d_1$ denotes our usual metric on $\AA_{\text{dim} H_i,1}$.

Note that since $\ell,\ell'\in\Gamma_{\mathcal{N}}$, the projections $\widetilde{\pi}_1(\ell),\widetilde{\pi}_1(\ell')$ are points, and so $d$ is well defined, and it is easy to check that it is indeed a metric. Furthermore, for each fixed $T_0\in\Gamma_\mathcal{N}$, and for all $\ell\in\mathcal{R}_i(\{V_{T_0}\})$, the projections $\widetilde{\pi}_2(\ell)$ are transversal with respect to the coordinate hyperplanes in $\widetilde{\pi}_2([0,1]^{(m-1)d})\subset H_i$. This allows us to conclude that for a small neighbourhood $U_0$ of $T_0$ in $\Gamma_{\mathcal{N}}$, the maps $\ell\mapsto\mathcal{H}^1(Q\cap\ell)$, $\mathcal{R}(U_0)\to[0,+\infty[$ are Lipschitz continuous with respect to the $d$-metric, with a Lipschitz constant independent of $Q\in\mathcal{Q}^{(m-1)d}$.
Indeed, using the expression $\ell=\widetilde{\pi}_1(\ell)\times\widetilde{\pi}_2(\ell)$, the Lipschitz continuity is easy to check with respect to both coordinates; for $\widetilde{\pi}_2$ it is the same argument as in \eqref{Nn-Holder-product} and for $\widetilde{\pi}_1$ it follows from the definition of the dyadic metric. Note also that $(\Gamma_\mathcal{N},d)$ has finite box dimension.

We conclude that the claims of Lemma \ref{lem:Yt-survives} and Corollary \ref{cor:indep-products} continue to hold in $U_0$ (note that transversality with respect to the planes \eqref{eq:trans_hyperplanes} is not used at all in the proof of Lemma \ref{lem:Yt-survives}).
This yields $\PP(Y^T>0\text{ for all }T\in U_0)>0$ and further ensures the a.s. existence of homothetic copies of all $\{t_1,\ldots,t_{m-1},0\}$ for all $t\in \Gamma_\mathcal{N}$ by virtue of Corollary \ref{cor:0-1}. Since there are only finitely many coincidence classes $\mathcal{N}$, this finishes the proof.
\end{proof}

\begin{rem}
The special case $m=2$ implies that the fractal percolation set contains all directions almost surely, provided $s>(d-1)/2$. A small variant of the proof of Theorem \ref{thm:patterns} yields the following generalization: if $s>(d-k)/2$, then a.s. for all $k$-planes $V$, there are two points in $A^{\text{perc}(d,p)}$ determining a direction in $V$.
\end{rem}

We turn to finding translated copies of finite sets inside the fractal percolation set $A$ (without scaling). Just because $A$ is bounded, we cannot hope to find a translation of every $m$ point set. However, if $s>d-d/m$, we have the following variant of Corollary \ref{cor:patterns}

\begin{cor}\label{cor:patterns_without_scaling}
Let $\mathcal{X}$ be a compact subset of $([0,1[^d)^m\setminus \Delta$, where $\Delta=\{ x_i=x_j \text{ for some }i\neq j\}$ is the union of the diagonals. If $s>d-d/m$, then with positive probability $A=A^{\text{perc}(d,p)}$ contains a translated copy of each $(x_1,\ldots,x_m)\in\mathcal{X}$.
\end{cor}

\begin{proof}
The proof is very similar to the first part of the proof of Corollary \ref{cor:patterns} above.
Fix $T_0=(x_1,\ldots,x_{m-1})\in(\R^{d})^{m-1}$. We first find $n_0\in\N$ and a finite union $U$ of cubes $Q\in\mathcal{Q}_{n_0}^{md}$ which do not intersect the $\Delta$, and such that $V_{T_0}\cap U^\circ\neq\varnothing$, where for $T=(t_1,\ldots, t_{m-1})\in(\R^d)^{m-1}$,
\[
V_{T}=\{(x,\ldots,x)+(t_1,\ldots, t_{m-1},0)\,:\, x\in\R^d\}\in\mathbb{A}_{md,d}\,.
\]
We apply Corollary \ref{cor:indep-products} in a suitably small neighbourhood $\mathcal{U}$ of $T_0$ to deduce that there is a positive probability that
$Y^T>0$ for all $T\in \mathcal{U}$, where $Y^T=\lim_{n}Y^T_{n}$ and
\[Y_n^T=\int_U\mu_n\,d\mathcal{H}^d|_{V_T}\,.\]
The application of Corollary  \ref{cor:indep-products}  is justified, since each $V_T$ is transversal to the coordinate planes $\{x_{i}^j=0\}, \{x_{i}=0\}$ ($i\in[m]$, $j\in[d]$). Recall that these are the only transversality conditions to be checked since $\dim V_T=d$.
If $Y^T>0$, then $V_T\cap A^m\neq\varnothing$ and so there is $x\in A$ such that $x+\{t_1,\ldots, t_{m-1},0\}\subset A$.

Since $\mathcal{X}$ is compact, we can cover it by finitely many neighbourhoods $\mathcal{U}$ as above. The claim then follows from Harris' inequality (Lemma \ref{lem:Harris}).
\end{proof}

\subsection{Optimality of the results}

As noted in the introduction, there are sets $A\subset\R^d$ of zero Hausdorff dimension which contain homothetic copies of all finite sets of $\R^d$. This suggests that other notions of dimension might be more natural for this problem. If we instead consider packing dimension, we find that the dimension thresholds in Corollaries \ref{cor:patterns}, \ref{cor:patterns_without_scaling} are optimal (up to the endpoint), even for deterministic sets:

\begin{prop}\label{prop:threshold_sharp}
Let $A\subset\R^d$ be a Borel set of packing dimension $s$.
\begin{enumerate}
\item\label{sharp_2} If $A$ contains a homothetic copy of $\{t_1,\ldots, t_{m-1},0\}$ for a non-empty open set of $T=(t_1,\ldots,t_{m-1})\in(\R^d)^{m-1}$, then $s\ge d-(d+1)/m$.
\item\label{sharp_3} If $A$ contains a translated copy of $\{t_1,\ldots, t_{m-1},0\}$ for a non-empty open set of $T\in(\R^d)^{m-1}$, then $s\ge d-d/m$.
\end{enumerate}
Moreover, in both cases `non-empty open set' may be replaced by  `set of full packing dimension' in the relevant ambient space.
\end{prop}

\begin{proof}
The only properties of packing dimension that we use are  (a) it does not increase under locally Lipschitz maps, and (b) $\dim_P(A^q) \le q \dim_P(A)$ for $q\in\N$, $A\subset\R^d$. For the latter property see e.g. \cite[Theorem 8.10]{Mattila95}.

Let $g$ be the map
\[(x_1,x_2,\ldots,x_m)\mapsto|x_2-x_1|^{-1}(x_2-x_1,\ldots,x_m-x_1)\colon (\R^d)^m\to S^{d-1}\times(\R^d)^{m-2}\,,\]
which is locally Lipschitz outside of $\Delta$. Hence
\[
 d-1+(m-2)d = \dim_P(g(A^m\setminus\Delta)) \le m \dim_P(A)\,
\]
giving \eqref{sharp_2}.· The claim \eqref{sharp_3} obtained in a similar manner, by considering the map
 \[(x_1,x_2,\ldots,x_m)\mapsto (x_2-x_1,\ldots,x_m-x_1)\colon (\R^d)^m\to (\R^d)^{m-1}\,.\]

The last claim is clear from the argument.
\end{proof}

What happens at the threshold? We recall that fractal percolation with the critical parameter $p=2^{-d}$ goes extinct a.s. Much more is true: if $V\subset [0,1]^d$ is a Borel set with finite $(d-s)$-capacity, and $A=A^{\text{perc}(d,p)}$ has dimension $s$, then $A\cap V=\varnothing$ almost surely, see \cite{LyonsPeres17}. The usual proofs of these facts do not seem to easily extend to the setting of the product of independent realizations of fractal percolation. Nevertheless, we have the following result:
\begin{prop} \label{prop:extinction-critical}
Let $A^{(1)},\ldots,A^{(m)}$ be independent realizations of $A^{\text{perc}(d,p)}$. Then for each compact set $V\subset (\R^d)^m$  of finite $(d-s)m$-dimensional Hausdorff measure, almost surely $V\cap (A^{(1)}\times\cdots\times A^{(m)})=\varnothing$.
\end{prop}
\begin{proof}
Without loss of generality, $V\subset [0,1)^{md}$. By assumption, there exists a constant $C$ (depending only on $\mathcal{H}^{(d-s)m}(V)$) and collections $\mathcal{C}_n$ of dyadic cubes such that:
\begin{enumerate}
\item $\mathcal{C}_0=\{ [0,1)^d\}$.
\item The cubes in each $\mathcal{C}_n$ are disjoint, and their union covers $V$,
\item Each element of $\mathcal{C}_{n+1}$ is strictly contained in some element of $\mathcal{C}_n$,
\item $\sum_{Q\in\mathcal{C}_n} 2^{-\ell(Q) (d-s)m} \le C$,
\end{enumerate}
where $\ell(Q)$ is the side-length of $Q$. As usual, write $A_n=A^{(1)}_n\times\cdots\times A^{(m)}_n$, and let $K_n$ be the number of cubes $Q$ in $\mathcal{C}_n$ such that $Q\subset A_{\ell(Q)}$.

Let $\mathfrak{F}_n$ be the event $(K_n\ge 1)$. Note that the events $\mathfrak{F}_n$ are decreasing by (3), and $\cap_n \mathfrak{F}_n$ is the event $V\cap (A^{(1)}\times\cdots\times A^{(m)})\neq \varnothing$. Assume that $\PP(\cap_n \mathfrak{F}_n)=q>0$. Since the $\mathfrak{F}_n$ are nested,
\[
\PP(\mathfrak{F}_n) = \PP(\mathfrak{F}_n|\mathfrak{F}_{n-1})\cdots \PP(\mathfrak{F}_1|\mathfrak{F}_0).
\]
So if we can show that $\PP(\mathfrak{F}_{n+1}^c|\mathfrak{F}_n)\ge c$ for some constant $c>0$ independent of $n$, we are done. Note that $2^{-\ell(Q) (d-s)m}$ is the probability that $Q\in A_{\ell(Q)}$. Hence, by (4), $\EE(K_n) \le C$ for all $n$. By Markov's inequality, $\PP(K_n>M_0)\le C/M_0$. Hence, if $M_0$ is large enough (depending only on $C,q$) then $\PP(1\le K_n\le  M_0) \ge q/2$. Write $\mathfrak{C}_n$ for the event $(1\le K_n\le M_0)\subset \mathfrak{F}_n$. Note that
\begin{equation} \label{eq:conditional-probability-dying}
\PP(\mathfrak{F}_{n+1}^c|\mathfrak{C}_n) \ge (1-2^{(s-d)m})^{2^d M_0} =: c_0>0.
\end{equation}
Indeed, suppose $\mathfrak{C}_n$ holds and let $Q_1,\ldots,Q_M$ be the cubes in $\mathcal{C}_n$ such that $Q_i\subset A_{\ell(Q_i)}$ (so that $1\le M\le M_0$). For each $j=0,\ldots,M-1$, the probability that $Q_{j+1}\cap A_{\ell(Q_{j+1})+1}=\varnothing$ given that $Q_{i}\cap A_{\ell(Q_{i})+1}=\varnothing$ for $i=1,\ldots,j$ is bounded below by $(1-2^{(s-d)m})^{2^d}$. Indeed, this number is the unconditional probability that $Q_{j+1}$ has no offspring in $A_{\ell(Q_{j+1})+1}$, and the information that the previous cubes had no offspring can only increase the probability (if there is overlap in some coordinate projection). Hence the probability that all $Q_j$ have no offspring is at least $c_0$, but thanks to (3) this is a sub-event of $\mathfrak{F}_{n+1}^c$, giving \eqref{eq:conditional-probability-dying}.

To conclude, note that
\[
\PP(\mathfrak{F}_{n+1}^c|\mathfrak{F}_n) \ge \frac{\PP(\mathfrak{F}_{n+1}^c\cap \mathfrak{E}_n)}{\PP(\mathfrak{F}_n)}\ge \frac{q}{2} \PP(\mathfrak{F}_{n+1}^c\,|\, \mathfrak{C}_n)\ge \frac{q c_0}{2} =: c >0.
\]
 \end{proof}

As an immediate corollary, we get:
\begin{cor}
Let $(S_j)_j$ be a countable collection of $m$-element sets in $\R^d$. If $s=s(d,p)\le d-(d+1)/m$, then a.s. $A^{\text{perc}(d,p)}$ does not contain a homothetic copy of any of the $S_j$. Likewise, if $s\le d-d/m$, then a.s. $A^{\text{perc}(d,p)}$ does not contain a translated copy of any of the $S_j$.
\end{cor}
\begin{proof}
It is enough to show the claim for a single set $S$. Moreover, by the usual decomposition of $A^m\setminus\Delta$ into dyadic cubes, we may replace $A^{\text{perc}(d,p)}$ by the product of $m$ independent realizations of $A^{\text{perc}(d,p)}$. The claim is then immediate from Proposition \ref{prop:extinction-critical}.
\end{proof}

We also have the following corollary of (the proof of) Proposition \ref{prop:extinction-critical} that will be required later.
\begin{cor} \label{cor:extinction-quantitative}
  Suppose $V\subset\R^{md}$ can be covered by $C 2^{m(d-s)n}$ cubes in $\mathcal{Q}_n$ for all $n$. Let $A^{(1)},\ldots,A^{(m)}$ be independent realizations of $A^{\text{perc}(d,p)}$. Then there exists a sequence $q_n\downarrow 0$, depending only on $C$, such that
  \[
  \PP(V\cap A_n^{(1)}\times \cdots\times A_n^{(m)}\neq \varnothing) \le q_n \quad\text{for all }n\in\N.
  \]
\end{cor}
\begin{proof}
  Let $\mathcal{C}_n$ be the cubes in $\mathcal{Q}_n$ that hit $V$, so that (1)-(4) in the proof of Proposition \ref{prop:extinction-critical}  hold by assumption. Let $\mathfrak{C}_n$ denote the event that $V\cap A_n^{(1)}\times \cdots\times A_n^{(m)}\neq \varnothing$.  Fix $q>0$ and suppose $q\le \PP(\mathfrak{C}_n)$. Then the proof  of Proposition \ref{prop:extinction-critical} shows that there is $c=c(q,C)\in (0,1)$ such that
  \[
  q \le \PP(\mathfrak{C}_n)  \le (1-c)^n.
  \]
  This is absurd if $n$ is large enough depending on $q, C$, so $\PP(\mathfrak{C}_n)\le q$ for $n$ sufficiently large depending on $q, C$, which yields the claim.
\end{proof}

\section{Nonlinear intersections and applications}\label{sec:nonlinear}

In this section, we prove a non-linear generalization of Theorem \ref{thm:independent-product-linear-int} : we will replace the family $\Gamma\subset\AA_{md,k}$ in Theorem \ref{thm:independent-product-linear-int} by a family of $k$-dimensional algebraic surfaces of bounded degree. As applications, we obtain the sharp dimension thresholds for the existence of many non-linear configurations in $A$, such as all angles, similar copies of all triangles and all small enough distances.

\subsection{Preliminaries}

Before stating the main result of the section, we recall some notation and  prove some preliminary results. For $M,q\in\N$ fixed, $q\le M$, let $\cP_{r,q,M}$ be the family of non-constant polynomials $\R^M\to\R^q$ of degree $\le r$ (as in Section \ref{sec:affine-intersections}, in our applications $M$ will be of the form $M=md$ for some $m\in\N$). We will often shorten this to $\cP_{r}$ when the parameters $q$ and $M$ are clear from context.
Also write $\cP_{r,q,M}^{\text{reg}}$ for the polynomials in $\cP_{r,q,M}$ for which $0$ is a regular value on $[0,1]^M$ (that is, the rank of $D P(x)$ is $q$ for all $x\in P^{-1}(0)\cap[0,1]^M$). We identify elements $P=(P_1,\ldots,P_q)$ of $\cP_r$ with the coefficients of $P_i$, $i\in[q]$ and in this way see $\cP_r$ as a subset of some Euclidean space $\R^N$. This induces a metric, given by the Euclidean distance of the coefficients, which will be referred to as the Euclidean distance on $\cP_r$.

Given a polynomial $P:\R^M\to\R^q$, we denote $P^{-1}(0)\cap [0,1]^M$ by either $\zero_P$ or $\zero(P)$. The dimensions $M,q$ will always be clear from context. Furthermore, we denote $\eta_P = \mathcal{H}^{M-q}|_{\zero(P)}$, with the convention that $\eta_P$ is the trivial measure if $\mathcal{H}^{M-q}(\zero(P))=0$.

Globally, the Euclidean metric does not match the geometric closeness of the varieties: two polynomials $P_1, P_2$ with $|P_1-P_2|$ small may have completely different zero sets $\zero_{P_1},\zero_{P_2}$. The next lemma asserts that, near a polynomial for which $0$ is a regular value, this does not happen. For $M=2$ and $q=1$, the statement of the lemma is contained in Lemma 8.8 of \cite{ShmerkinSuomala14}.
\begin{lemma} \label{lem:isotopic-curves}
Let $P\in\cP_{r,q,M}^{\text{reg}}$
such that $P^{-1}(0)\cap ]0,1[^M\neq\varnothing$.

Then there exist neighbourhoods $\mathcal{O}$ of $P$ (in $\cP_r$ with the Euclidean metric) and $U$ of $\zero_P$, and a real analytic map $G:\mathcal{O}\times \mathcal{O}\times U\to\R^M$  such that for all $P_1,P_2\in \mathcal{O}$:
\begin{enumerate}
\item\label{G1} $\zero_{P_1}\subset U$, and $G(P_1,P_1,\cdot)$ is the identity on $P_1^{-1}(0)\cap U$,
\item\label{G2} $G(P_1,P_2,\cdot)|_{P_1^{-1}(0)\cap U}$ is a diffeomorphism onto its image,\\ and $\zero_{P_2}\subset G(P_1,P_2,P_1^{-1}(0)\cap U)$,
\item\label{G3} $P_2(G(P_1,P_2,u))=0$ whenever $u\in\zero_{P_1}$.
\end{enumerate}
\end{lemma}

\begin{proof}
The idea of the proof is to define $G$ as follows: given $P_1, P_2\in\cP_{r,q,M}$ close to $P$ and $x$ close to $\zero_P$, we define $G(P_1,P_2,x)$ as the `first' point in $\{ P_2 = P_1(x)\}$ that is reached by moving from $x$ orthogonally to the variety $\{ P_2 = P_1(x) \}$. The implicit function theorem will ensure that this function is indeed well defined and has the claimed properties.

As above, we identify $\cP_{r,q,M}$ with the appropriate Euclidean space $\R^N$. We let $\Phi_1,\Phi_2:\R^N\times\R^N\times \R^M\times \R^M\times \R^q\to \R^M\times \R^M$ be given by
\begin{align*}
\Phi_1(P_1,P_2,x,y,\lambda) &= \left(P_2(y)-P_1(x),y-x-DP_2(y)^t \cdot \lambda\right),\\
\Phi_2(P_1,P_2,x,y,\lambda) &= \left(P_2(y)-P_1(x),y-x-DP_1(x)^t \cdot \lambda\right)\,.
\end{align*}
We write $\partial\Phi_i/\partial(y,\lambda)$ for the partial derivatives of $\Phi_i$ with respect to the $M+q$ variables $(y,\lambda)$. Then we have the following block structures for $i=1,2$:
\[
\frac{\partial\Phi_i}{\partial(y,\lambda)}(P,P,x,x,0)= \left(
                                             \begin{array}{c|c}
                                               DP(x) & 0_{q\times q} \\
                                               \hline
                                               I_{d\times d} & -DP(x)^t \\
                                             \end{array}
                                           \right).
\]
Since $P\in\cP_{r,q,M}^{\text{reg}}$, for $x\in \zero_P$ we have:
\[
\det\left(
                                             \begin{array}{c|c}
                                               DP(x) & 0_{q\times q} \\
                                               \hline
                                               I_{d\times d} & -DP(x)^t \\
                                             \end{array}
                                           \right) = \det\left(-DP(x)\cdot DP(x)^t\right) \neq 0.
\]
The implicit function theorem together with compactness then provides neighbourhoods $\widetilde{\mathcal{O}}, \widetilde{U}$ of $P$, $\zero_P$, and real-analytic functions $(y_1,\lambda_1),(y_2,\lambda_2):\widetilde{\mathcal{O}}\times \widetilde{\mathcal{O}}\times \widetilde{U}\to \R^M\times\R^q$,  such that
\[
\Phi_i(P_1,P_2,x,y_i(P_1,P_2,x),\lambda_i(P_1,P_2,x)) \equiv 0,\quad i=1,2.
\]
Let $\mathcal{O},U$ be neighbourhoods of $P,\zero_P$ which are compactly contained in $\widetilde{\mathcal{O}}, \widetilde{U}$ respectively.

We define $G:=y_1$. Note that $\Phi_1(P_1,P_1,x,x,0)=0$ whenever $P_1\in\widetilde{\mathcal{O}}$, $x\in  \widetilde{U}$, so the uniqueness of the implicit function (assuming $\widetilde{\mathcal{O}}$ is connected as we may) ensures that $G(P_1,P_1,x)=x$. Using this, compactness, and the continuity of $\Phi_i$, by making $\mathcal{O}, U$ smaller, we may ensure that $G(P_1,P_2,x)\in \widetilde{U}$ whenever $P_1,P_2\in\mathcal{O}$ and $z\in P_1^{-1}(0)\cap U$. Again by compactness, and making $\mathcal{O}$ smaller if needed, we may guarantee that $\zero_{\wt{P}}\subset U$ for all $\wt{P}\in \mathcal{O}$.  In particular, \eqref{G1} holds.

We claim that, perhaps after making $\mathcal{O}$ even smaller, if $f_1=y_1(P_1,P_2,\cdot)$ and $f_2=y_2(P_2,P_1,\cdot)$, then $f_2 f_1(x)=x$ for all $P_1,P_2\in \mathcal{O}$ and all $x\in P_1^{-1}(0)\cap U$. Indeed, note first that $f_2 f_1(x)\in P_1^{-1}(0)$ is well defined. Suppose $z=f_2 f_1(x)-x\neq 0$. Given $\wt{P}\in\mathcal{O},y\in U$, let $H_{\wt{P},y}$ be the span of the gradients of the coordinate functions of $\wt{P}$ evaluated at $y$. Since $H_{\wt{P},y}$ is perpendicular to the tangent of $\{ \wt{P}=\wt{P}(y)\}$ at $y$, and the map $(\wt{P},y)\to H_{\wt{P},y}$ is continuous at points $(\wt{P},y)$ such that $y$ is a regular point of $\wt{P}$, we find that (making $\mathcal{O}$ smaller again) the angle between $H_{P_2,f_1(x)}$ and $z$ is non-zero. But it follows from the definitions that $z$ has the form
\[
z =  DP_2(f_1(x))^t\cdot \lambda
\]
for some $\lambda\in \R^q\setminus\{0\}$, which means precisely that $z\in H_{P_2,f_1(x)}$, contradicting our hypothesis. So $z=0$ as we had claimed. Taking stock, we have shown that $f_2$ is the inverse function to $f_1$, so that $f_1:P_1^{-1}(0)\cap U\to\R^M$ is a diffeomorphism onto its image.

By definition of $G$, we have that $P_2(G(P_1,P_2,x))=0$ whenever $P_1,P_2\in \mathcal{O}$ and $x\in P_1^{-1}(0)\cap U$, in particular giving \eqref{G3}. By making $\mathcal{O}$ smaller one more time, we may assume that $\zero(P_2)\subset G(P_1,P_2,P_1^{-1}(0)\cap U)$, showing that \eqref{G2} holds.
\end{proof}

Next, we present a useful consequence of the coarea formula for submanifolds of $\R^M$. Recall that the $\ell$-Jacobian $J_\ell L$ of a linear map $L:\R^M\to\R^\ell$, $d\ge \ell$, is given by the product of its singular values, which are the square roots of the eigenvalues of $L L^t$ (so that $(J_\ell L)^2 = \det(L L^t)$). We denote the tangent of a submanifold $S\subset\R^M$ at $x\in S$ by $T_x(S)$.

\begin{prop} \label{prop:coarea}
Let $P\in\mathcal{P}_{r,q,M}^{\text{reg}}$, let $V\subset [0,1]^M$ be a Borel set, and let $f:U\to \R^\ell$ be a $C^1$ map, where $U$ is a neighbourhood of $\zero_P\cap V$ and $\ell\in [M]$. Assume that:
\begin{itemize}
\item $\inf_{x\in V\cap \zero_P} J_\ell(Df(x)) \ge c_1 >0$,
\item For all $x\in\zero_P\cap V$, the angle between the tangent plane to $P^{-1}(0)$ at $x$ and $\text{ker}(Df(x))$ is at least $c_2>0$.
\end{itemize}
If, furthermore, $M-q \ge \ell$, then
\[
\mathcal{H}^{M-q}(\zero_P\cap V) \le O(c_1^{-1} c_2^{-(M-q)}) \int_{f(\zero_P\cap V)} \mathcal{H}^{M-q-\ell}(\zero_P\cap V\cap f^{-1}(y)) \, d\mathcal{H}^{\ell}(y).
\]
\end{prop}
\begin{proof}
By a standard approximation, we may assume $V$ is open. Let $S=\zero_P\cap V$; it is a regularly embedded submanifold of $\R^M$. Write $J_\ell^S f(x)$ for the $\ell$-Jacobian relative to $S$; see \cite[Definition 5.3.3]{KrantzParks2008}. We claim that
\begin{equation} \label{eq:lower-bound-ell-Jacobian}
J_\ell^S f(x)=\Omega(c_1 c_2^\ell)\quad\text{ for all }x\in S.
\end{equation}
Assuming this, the claim follows immediately from the coarea formula for submanifolds, see \cite[Theorem 5.3.9]{KrantzParks2008}, so the task is to establish \eqref{eq:lower-bound-ell-Jacobian}.

Write $L=Df(x)$ by simplicity. By the singular value decomposition, after orthonormal changes of bases in $\R^M$ and $\R^\ell$ we may assume that $L$ has the form
\begin{equation} \label{eq:matrix-SVD}
\left(
  \begin{array}{c|c}
    D & O_{M-\ell\times\ell}
  \end{array}
\right),
\end{equation}
where $D$ is an $\ell\times\ell$ diagonal matrix with the singular values of $L$ on the diagonal, so that $J_\ell(L)=\det(D)$. Now let $e_1,\ldots, e_\ell$ be an orthonormal basis for the orthogonal complement $W$ of $T_x(S)\cap \text{ker}(L)$ inside $T_x(S)$; by assumption, any non-zero element of $W$ makes an angle $\ge c_2$ with $\text{ker}(L)$. Let $\pi$ denote orthogonal projection onto the orthogonal complement of $\text{ker}(L)$. We note that $\pi^{-1}:\pi(W)\to W$ is well defined and $O(c_2)$-Lipschitz; this is the same argument as in \eqref{eq:bilip}. In particular, restricting $\pi^{-1}$ to the cube spanned by $\pi(e_i)$, we see that the parallelogram spanned by $\pi(e_i)$ has $\ell$-area $\Omega(c_2^\ell)$. On the other hand, from \eqref{eq:matrix-SVD} we see that if $E$ is the matrix with $e_i$ as columns and $E'$ is the matrix with $\pi(e_i)$ as columns, then
\[
\det(LE) = \det(D)\det(E') \ge \Omega(c_1 c_2^\ell),
\]
so that the definition of relative Jacobian gives \eqref{eq:lower-bound-ell-Jacobian}.
\end{proof}

\subsection{Continuity for intersections with algebraic varieties}

We are ready to extend Theorem \ref{thm:independent-product-linear-int} to intersections with algebraic varieties. We note that for $d>2$, this result is new even in the SI-martingale case (that is, in the case $m=1$). Recall the notations $H^I$, $H^{I,i,k}$ from Section \ref{sec:notation}.

\begin{thm}\label{thm:nonlinear}
Let $\nu_{n}^{(j)}$, $j\in[m]$, be $m$ independent realizations of $\nu_n^{\text{perc}(d,p)}$, and let
\[
\mu_n = \nu_n^{(1)}\times\cdots\times \nu_n^{(m)}.
\]
Let $q\in[md-1]$ and $P_0\in\cP_{r,q,md}^{\text{reg}}$, such that $P_0^{-1}(0)\cap]0,1[^{md}\neq\varnothing$.
Suppose that for each index set $I\subsetneq [m]$, and each $i\in[m]\setminus I$, $k\in[d]$, the tangent planes of $P_0^{-1}(0)$ at each $a\in P_0^{-1}(0)\cap[0,1]^{md}$ form an angle $>0$ with the planes $H^I$, $H^{I,i,k}$.

Then there is a neighbourhood $\mathcal{O}$ of $P_0$ such that:
\begin{enumerate}
\item
Assume $s=s(d,p)>q/m$. Denoting
\[
Y_n^{P}=\int_{P^{-1}(0)\cap[0,1]^{md}}\mu_n\,d\mathcal{H}^{md-q}\,,
\]
the sequence $Y_n^P$ converges uniformly over all $P\in \mathcal{O}$ and $Y^P=\lim_{n\to\infty}Y_n^P$ satisfies
\[
|Y^{P_1}-Y^{P_2}|\le K|P_1-P_2|^\gamma\,,
\]
where $\gamma>0$ is a deterministic constant depending on $s,r,q,d,m$, and $K$ is a finite random variable.
\item Now suppose $s\le q/m$.  Then almost surely
\[\sup_{n\in\N, P\in \mathcal{O}}2^{-\gamma n}Y_{n}^P<\infty\,,\]
for any $\gamma>q-ms$.
\end{enumerate}
\end{thm}
We emphasize that the neighbourhood $\mathcal{O}$ is independent of $s$.

The proof of the theorem depends on several lemmas that we present first.
\begin{lemma} \label{lem:Frostman-algebraic-varieties}
Let $P_0\in\mathcal{P}_{r,q,M}^{\text{reg}}$. Then there are a neighbourhood $\mathcal{O}$ of $P_0$ and $C=C(P_0)>0$ such that
\[
\sup_{P\in\mathcal{O}}\sup_{x\in \R^M}\sup_{r>0} \eta_P(B(x,r)) \le C\, r^{M-q}.
\]
\end{lemma}
\begin{proof}
We can find finitely many open sets $U_i\subset\R^M$ whose union covers $\zero_{P_0}$, and coordinate projections $\wt{\pi}_i:U_i\to \R^{M-q}$, such that $\wt{\pi}_i|_{U_i}$ is injective and, moreover, the angle between $\text{ker}(\wt{\pi}_i)$ and the tangent plane $T_x(P_0^{-1}(0))$ is at least $c>0$ for all $i$ and $x\in P_0^{-1}(0)\cap U_i$. Proposition \ref{prop:coarea} applied to $P_0^{-1}(0)\cap U_i\cap B(x,r)$ and the maps $\wt{\pi}_i$ now yields, for any $x\in \zero_{P_0}$ and $r>0$,
\[
\eta_{P_0}(B(x,r)) \le O_c(1) \sum_i\mathcal{H}^{M-q}(\wt{\pi}_i(B(x,r))) = O_c(r^{M-q}).
\]
Letting $G=G(P_0,P,\cdot)$ be as in Lemma \ref{lem:isotopic-curves} and using that $B(x,r)\cap \zero_{P}\subset G(B(G^{-1}(x),2r))$, the argument extends to a sufficiently small neighbourhood of $P_0$.
\end{proof}

\begin{lemma} \label{lem:algebraic-a-priori-Holder}
Let $P_0\in\mathcal{P}_{r,q,M}^{\text{reg}}$, $q\in[M-1]$. Assume that the tangent planes of $\zero_{P_0}$ make an angle $>0$ with all coordinate hyperplanes. Then there are a neighbourhood $\mathcal{O}$ of $P_0$ and $C=C(P_0)>0$ such that
\[
\sup_{P\in\mathcal{O}}\sup_{H}\eta_P(H(\delta)) \le C\,\delta,
\]
where $H$ runs over all coordinate hyperplanes, and $H(\delta)$ denotes the $\delta$-neighbourhood of $H$.
\end{lemma}
\begin{proof}
Firstly, we claim that there is a neighbourhood $\mathcal{O}$ of $P_0$ such that
\begin{equation} \label{eq:supremum-int-hyperplanes}
\sup_{P\in\mathcal{O}}\sup_{H} \mathcal{H}^{M-q-1}(H\cap \zero_P) =: C'=C'(P_0) < \infty,
\end{equation}
where again $H$ runs over coordinate hyperplanes. Indeed, fix $t$ and $i\in [M]$ and let $H_{t,i}=\{(x_1,\ldots,x_M)\,:\,x_i=t\}$ be a coordinate hyperplane. Consider the polynomial $P_0^{t,i}=(P_0,x_i-t):\R^M\to \R^{q+1}$. By the assumption that the tangents to $\zero_{P_0}$ make a positive angle with $H_{t,i}$, $0$ is a regular value of $P_0^{t,i}$ on $[0,1]^M$. By Lemma \ref{lem:Frostman-algebraic-varieties}, there are neighbourhoods $\mathcal{O}$ of $P_0$ and $U$ of $t$ such that the claim holds when taking supremum over all planes of the form $\{x_i=u\}, u\in U$. By compactness, we conclude that \eqref{eq:supremum-int-hyperplanes} holds.

Fix $i\in [M]$, and let $\wt{\pi}(x)=x_i$.
By compactness, and making $\mathcal{O}$ smaller if needed, we may ensure that all tangent planes to $\zero_P, P\in\mathcal{O}$ make a uniformly positive angle with all coordinate hyperplanes. Proposition \ref{prop:coarea} now yields
\[
\eta_P(H_{t,i}(\delta))\le O(1) \int_{s=t-\delta}^{t+\delta} \mathcal{H}^{M-q-1}( \wt{\pi}^{-1}(s)\cap \zero_P)\, ds=  O(\delta),
\]
using \eqref{eq:supremum-int-hyperplanes}. This is what we wanted to prove.
\end{proof}

\begin{proof}[Proof of Theorem \ref{thm:nonlinear}]
The proof proceeds along the lines of the proof of Theorem \ref{thm:independent-product-linear-int}. We highlight the main differences. As before, conditions \eqref{RM:bounded}--\eqref{RM:quotients-bounded} are clearly valid. Our parameter space will consist of a suitable neighbourhood $\mathcal{O}$ of $P_0$ in $\cP_{r,q,md}^{\text{reg}}$, where the conclusions of Lemma \ref{lem:isotopic-curves} hold, and such that the tangent planes of each $P\in \mathcal{O}$ at each $x\in \zero_P$ satisfy the required transversality condition with respect to the planes \eqref{eq:trans_coordinate}-- \eqref{eq:trans_hyperplanes}. Furthermore, denoting the family of all such tangent planes by $\Gamma\subset\AA_{md,md-q}$, we assume that also $\mathcal{R}^p(\Gamma)$ satisfy a similar condition. Note that due to our assumptions and Lemma \ref{lem:isotopic-curves}, all these conditions hold in a sufficiently small neighbourhood $\mathcal{O}$ of $P_0$ (recall Remark \ref{rem:transversality}). In the sequel, we will apply various estimates that all hold in a small neighbourhood of $P_0$, and we assume that $\mathcal{O}$ is fixed and sufficiently small such that all of these are valid simultaneously. We stress that all the induced constants are allowed to depend on $P_0$.

We will show that the assumptions of Theorem \ref{thm:Holder-continuity} hold for $\mu_n$ and the family of measures $\eta_P$, $P\in \mathcal{O}$ (recall that $\eta_P=\mathcal{H}^{md-q}|_{\zero(P)}$). 

Trivially,  \eqref{H:size-parameter-space} holds. Furthermore, \eqref{H:codim-random-measure} holds with $\alpha=m(d-s)$.  The Frostman condition \eqref{H:dim-deterministic-measures} holds, with $\kappa=md-q$, thanks to Lemma \ref{lem:Frostman-algebraic-varieties}.

To prove \eqref{H:Holder-a-priori}, we first show that $P\mapsto\mathcal{H}^{md-q}(Q\cap P^{-1}(0))$ is uniformly Lipschitz for all $Q\in\mathcal{Q}^{md}$, $P\in \mathcal{O}$. To that end, fix $P_1, P_2\in \mathcal{O}$ and let $\varepsilon=|P_1-P_2|$. Write $G=G(P_1,P_2,\cdot)$, and note that $G$ is $O(\e)$-close to the identity in the $C^1$ topology by Lemma \ref{lem:isotopic-curves}. In particular,
\[
J_{md-q}^{\zero(P_1)}(G)(x) \in (1-O(\e),1+O(\e)) \quad\text{for all }x\in \zero(P_1),
\]
where $J_{md-q}^{\zero(P_1)}(G)$ is the relative Jacobian, see \cite[Definition 5.3.3]{KrantzParks2008}. Using this and the estimate $\mathcal{H}^{md-q}(\zero(P_1))=O(1)$, the area formula for submanifolds (\cite[Theorem 5.3.7]{KrantzParks2008}) applied to the submanifold $S=Q^\circ \cap \zero_{P_1}$ and the map $G$ gives
\[
\mathcal{H}^{md-q}(S) \le \mathcal{H}^{md-q}(G(S)) + O(\e).
\]
By Lemma \ref{lem:isotopic-curves} and using again that $G$ is $C^1$ close to the identity,  we know that
\[
\mathcal{H}^{md-q}(G(S)) \le \mathcal{H}^{md-q}\left( P_2^{-1}(0)\cap Q(O(\e))\right),
\]
recall $Q(\delta)$ is the $\delta$ neighbourhood of $Q$. The last two displayed equations together with Lemma \ref{lem:algebraic-a-priori-Holder} yield
\[
\mathcal{H}^{md-q}(\zero_{P_1}\cap Q) \le \mathcal{H}^{md-q}(\zero_{P_2}\cap Q)+ O(\e),
\]
which by symmetry implies that $P\mapsto\mathcal{H}^{md-q}(Q\cap P^{-1}(0))$ is Lipschitz with the constant depending only on $P_0$, as claimed. Note that we are using a quite coarse estimate here since we are not taking into account the size of $Q$ (apart from the fact that it is contained in $[0,1]^{md}$).

Note that each $\zero(P)$, $P\in\mathcal{O}$ intersects at most $O(2^{n(md-q)})$ cubes in $\mathcal{Q}_n$. Indeed, this is true for $P_0^{-1}\cap U$ (where $U$ is the neighbourhood from \ref{lem:isotopic-curves}), since this is a bounded piece of a $(md-q)$-dimensional embedded manifold. The claim for arbitrary $P\in\mathcal{O}$ follows from Lemma \ref{lem:isotopic-curves}, since the image of a cube $Q\in\mathcal{Q}_n^{md}$ hitting $\zero_P$ under $G(P,P_0,\cdot)$ can be covered by $O(1)$ cubes in $\mathcal{Q}_n^{md}$. Using this, we conclude that
\[
\sup_{P_1,P_2\in \mathcal{O}}|Y_n^{P_1}-Y_n^{P_2}|=O(|P_1-P_2|2^{n(m(d-s)+md-q)})\,,
\]
that is \eqref{H:Holder-a-priori} holds with the parameters $\theta=m(2d-s)-q$ and $\gamma_0=1$.

We note that Lemma \ref{lemma:AWSD_product} also holds in this setting with essentially the same proof. The only difference is in the proof of \eqref{Nn-Holder-product}, which we have just explained.

As in the proof of Theorem \ref{thm:independent-product-linear-int}, the claims follow from Theorems \ref{thm:Holder-continuity} and \ref{thm:small_dimension_projections} if we can show that
\begin{equation}\label{AWSDclaimnonlinear}
\mu_n\text{ is a WSDM with parameter }\delta
\end{equation}
for some $\delta<ms-q$, if $s>q/m$, and for any $\delta>0$, if $s\leq q/m$.

Suppose first that $md-q\le d$. In this case, $\mathcal{DI}_n=O(1)$ (deterministically), recall that $\mathcal{DI}_n$ denote the dependency degree of $\mu_n$. To see this, note that because the tangent planes to $\zero_P$ make an angle $\ge c>0$ with each plane $H_{z,j}=\pi_j^{-1}(z)$, for any $P\in\mathcal{O}$, $x\neq x'\in \zero_P\cap H_{z,j}$ we have $|x-x'|=\Omega_{P_0,c}(1)$. Using transversality again, this implies that $\zero_P\cap \pi_j^{-1}(B(z,2^{-n}))$ can be covered by $O_{P_0,c}(1)$ balls of radius $O_{P_0,c}(2^{-n})$, which gives the claim.

Therefore, we assume from now on that $md-q>d$. Again, we proceed by induction on $m$. If $m=1$, the claim holds since then $(\mu_n)$ is even an SI-martingale. Suppose the claim holds for $m-1\ge 1$. We will need an algebraic variant of $\mathcal{R}(\Gamma)$ (Recall that $\Gamma$ denotes the collection of all the tangent planes of $\zero_P$, $P\in \mathcal{O}$). Given $P\in\mathcal{O}, c\in [0,1]^d, j\in [m]$ and $y\in B(0,\delta)\subset\R^q$, let
\[
\wt{P}_{P,c,j,y} = P(x_1,\ldots,x_{j-1},c,x_{j},\ldots,x_{m-1})-y,
\]
and set
\begin{equation}\label{eq:Utildedef}
\widetilde{\mathcal{O}}=\{\widetilde{P}_{P,c,j,y}:P\in \mathcal{O}, c\in[0,1]^d, j\in[m],  y\in B(0,\delta)\subset\R^q \}\,.
\end{equation}
Here $\delta$ is chosen small enough that the transversality assumptions continue to hold for $\wt{P}\in\wt{\mathcal{O}}$ (this is possible thanks to the transversality assumptions for $\mathcal{R}(\Gamma)$).

For each $\widetilde{P}\in\widetilde{\mathcal{O}}$, denote
\begin{equation*}
\widetilde{Y}_{n}^{\widetilde{P}}=\int_{\zero(\widetilde{P})}\nu_n^{(1)}\times\cdots \times \nu_n^{(m-1)}\,d\mathcal{H}^{(m-1)d-q}\,.
\end{equation*}

As in the proof of Theorem \ref{thm:independent-product-linear-int}, letting $Z_n=\sup_{\widetilde{P}\in\widetilde{\mathcal{O}}}\widetilde{Y}_{n}^{\widetilde{P}}$, we will show that for $n\ge O(1)$ there is a constant $C<\infty$ such that
\begin{equation}\label{eq:Delta_nonlinear}
\Psi(n)=C Z_n2^{n(-q+(m-1)s)}\,.
\end{equation}
is an upper bound for the dependency degree of $\mu_n$. To verify this, it suffices to show that
\begin{equation}\label{eq:pigeon2}
|\mathcal{Q}|=O(Z_n2^{n(-q+(m-1)s)})\,,
\end{equation}
whenever $j\in[m]$, $P\in \mathcal{O}$, and $\mathcal{Q}$ is a collection of dyadic cubes $Q\in\mathcal{Q}^{md}_n$
such that each $Q\in\mathcal{Q}$ satisfies $Q\subset A_n$, $Q\cap \zero(P)\neq\varnothing$, and $\pi_j(Q')=\pi_j(Q)$ for all $Q,Q'\in\mathcal{Q}$. Here (slightly abusing notation) we denote $A_n=\supp(\mu_n)=A_{n,1}\times\cdots\times A_{n,m}$, where $A_{n,i}$ are the steps in the construction of each of the independent fractal percolations.

To that end, we fix such $j\in[m]$, $P\in\widetilde{\mathcal{O}}$ and collection $\mathcal{Q}$, and adapt the estimation \eqref{eq:volume_argument} to the present setting as follows:
\begin{equation}\label{eq:volume_argument_coarea}
\begin{split}
&|\mathcal{Q}|2^{-nmd}=\mathcal{L}^{md}\left(\cup_{Q\in\mathcal{Q}}Q\right)=O(1)\int\limits_{y\in B_1}\mathcal{H}^{md-q}(P^{-1}(y)\cap\cup_{Q\in\mathcal{Q}}Q)\,d\mathcal{L}^q(y)\\
&=O(1)\int\limits_{y\in B_1}\int\limits_{z\in B_2}\mathcal{H}^{(m-1)d-q}(P^{-1}(y)\cap  A_n\cap\{x_j=z\})\,d\mathcal{H}^d(z)\,d\mathcal{L}^q(y)\,,
\end{split}
\end{equation}
where $B_1=B(0,O(2^{-n}))\subset \R^q$ and $B_2=\pi_j(Q_i)\subset\R^d$.
Here, the first estimate follows from Proposition \ref{prop:coarea} applied to the trivial polynomial and $f:=P$ (recall that $P\in\cP^\text{reg}$ for $P\in\mathcal{O}$). The integration is over $B(0,O(2^{-n}))$ because the $P\in\mathcal{O}$ are uniformly Lipschitz and $Q\cap P^{-1}(0)\neq \varnothing$ for all $Q\in\mathcal{Q}$.
The second estimate in \eqref{eq:volume_argument_coarea} follows from Proposition \ref{prop:coarea} applied to $\pi_j$, using the hypothesis of transversality with respect to the planes $\{ x_j=0 \}$.

From now on we assume that $n$ is large enough that $B_1\subset B(0,\delta)$, where $\delta$ is the number in the definition of $\wt{\mathcal{O}}$.  Note that
\[
\mathcal{H}^{(m-1)d-q}\left(P^{-1}(y)\cap A_n \cap\{x_j=z\}\right) \le 2^{n(m-1)(s-d)}\widetilde{Y}^{\widetilde{P}}_n\,,
\]
for $\widetilde{P}\in\widetilde{\mathcal{O}}$ defined by $\wt{P}(x_1,\ldots,x_{m-1})=P(x_1,\ldots,x_{j-1},z,x_j,\ldots,x_{m-1})-y$. We infer from \eqref{eq:volume_argument_coarea} that
\[
|\mathcal{Q}| 2^{-nmd} \le O(1) 2^{-nq} 2^{-nd} 2^{n(m-1)(s-d)} Z_n,
\]
from which \eqref{eq:pigeon2} follows.

After this, the rest of the proof proceeds just like the proof of Theorem \ref{thm:independent-product-linear-int}, using that $\nu_n^{(1)}\times\cdots\times \nu_n^{(m-1)}$ is a WSDM with respect to $\mathcal{O}$.
\end{proof}

\begin{rems}\label{rem:nodiagonal}
\begin{enumerate}[(i)]
\item
In general, one cannot remove the requirement that the tangents of $P^{-1}(0)$ are transversal to the planes $\{x_i=0, i\in I\}$, even if they fail at only finitely many points. To see this, let us consider the following example: Let $A=A^{\text{perc}(1,p)}\subset[0,1]$ with $\tfrac12<s<\tfrac23$, $\mu_n=\nu_n^{\text{perc}(1,p)}\times\nu_n^{\text{perc}(1,p)}$ and consider the random variables $Y_n^P$ for the quadratic polynomials $P(x,y)=|(x,y)-(x_0,y_0)|^2-\lambda$, for $x_0,y_0\in[0,1]$, $\lambda>0$. Thus, we are intersecting $A\times A$ and $\nu\times\nu$ with circles. Then, for $\nu\times\nu$ almost every $(x,y)\in A\times A$, and all circles $P^{-1}(0)$ through $(x,y)$ with a vertical tangent at $(x,y)$, it holds that $Y^P_n\longrightarrow\infty$ as $n\to\infty$.
Let us give a short sketch of the proof: The circular arc $P^{-1}(0)\cap\pi_1^{-1}(I)$ has length $\Omega(2^{-n/2})$, where $I\in\mathcal{Q}_n^1$ is the dyadic interval containing $x$. If $I'\in\mathcal{Q}_{k}^1$ is the dyadic interval containing $y$ with $k\sim n/2$, then $P^{-1}(0)$ typically intersects roughly $2^{ns/2}$ squares $I\times I''\in\mathcal{Q}_n^2$, where $I''\subset I'$. Thus, $Y_n^P$ is bounded from below by $\Omega(2^{2n(1-s)}\times 2^{ns/2}\times 2^{-n})=\Omega(2^{n(1-3s/2)})$ and this grows at an exponential rate as $n\to\infty$, whenever $s<2/3$.

This example also shows that the first order transversality condition \eqref{eq:trans_coordinate} for the tangents of $P^{-1}(0)$ cannot be weakened to the maps $P\mapsto\mathcal{H}^{md-q}(P^{-1}(0)\cap Q)$ being H\"older continuous for $Q\in\mathcal{Q}^{md}$.
\item On the other hand, there is scope to weaken the transversality assumption with respect to coordinate hyperplanes, and it might even be possible to eliminate this assumption altogether by considering a suitable version of the dyadic metric (adapting the proof of Corollary \ref{cor:patterns}). To avoid excessive technicalities, and since this is not an issue in our main applications, we do not pursue this.
\item The method behind the proof of Theorem \ref{thm:nonlinear} is not tied to algebraic surfaces and it also works for other parametrized families of smooth surfaces satisfying analogous transversality and dimensional conditions. Since all our applications regarding the existence of various configurations in $A$ will be derived by intersecting the Cartesian powers of $A$ with the elements of some $\cP^{\text{reg}}_{r,q,M}$, we shall not discuss these generalizations here.
\end{enumerate}
\end{rems}

In this setting, the analogue of Lemma \ref{lem:Yt-survives} also holds:
\begin{lemma} \label{lem:Yt-survives-nonlinear}
In the setting of Theorem \ref{thm:nonlinear}, if $s>q/m$, then $\PP(Y^{P_0}>0)>0$.
\end{lemma}
\begin{proof}
The proof is nearly identical to the proof of Lemma \ref{lem:Yt-survives}. The geometry of the planes $V$ only enters the proof via the estimate $\mathcal{H}^k(V\cap Q)=O(2^{-nk})$ for $Q\in\mathcal{Q}_{n}^{md}$, and the cardinality bound \eqref{eq:count-dependent-cubes}. In the present context, we clearly have $\mathcal{H}^{md-q}(\zero_{P_0}\cap Q) = O(2^{-n(md-q)})$ for  $Q\in\mathcal{Q}_{n}^{md}$. With respect to \eqref{eq:count-dependent-cubes}, we note that in Lemma \ref{lem:Yt-survives} this is obtained by showing that $V(\sqrt{d}2^{-n})\cap \wt{\pi}^{-1}(\wt{Q})\cap Q_j$ can be covered by a suitable parallelepiped. If we replace $V$ by $\zero_{P_0}$ then the same holds except that instead of the linear parallelepiped, we need to consider an $O(2^{-n})$-neighbourhood of the variety $\zero(P_0)\cap\widetilde{\pi}^{-1}(x_0)$ for some fixed $x_0\in\widetilde{\pi}(Q)$ and use the coarea formula (Proposition \ref{prop:coarea}) to verify the volume argument. Details are left to the interested reader.
\end{proof}

\subsection{Non-linear configurations}\label{ssec:nlc}

In order to find non-linear configurations inside the fractal percolation set, we will apply the following localized version of Theorem \ref{thm:nonlinear}.

\begin{cor} \label{cor:nonlinear}
Let $P_0\in\mathcal{P}_{r,q,md}$ and let $a=(a_1,\ldots,a_m)\in\R^{md}$ be a regular point for $P_0$ such that $P_0(a)=0$ and $a_i\neq a_j$ for all $i\neq j$. Let $V$ be the tangent plane to $P_0^{-1}(0)$ at $a$.
Suppose that for each index set $I\subsetneq [m]$, and each $i\in[m]\setminus I$, $k\in[d]$,
\begin{align*}
\dim V\cap H^I&=\max\{0,(m-|I|)d-q\}\,,\\
\dim V\cap H^{I,i,k}&=\max\{0,(m-|I|)d-q-1\}\,.
\end{align*}
Further, let $s>q/m$.

Then there is a neighbourhood $\mathcal{O}$ of $P_0$ such that
\[
\PP(\zero_P\cap A^m\neq\varnothing \text{ for all }P\in\mathcal{O})>0,
\]
where $A=A^{\text{perc}(d,p)}$ is the fractal percolation set of dimension $s$.
\end{cor}

\begin{proof}
Since $V$ is transversal to coordinate hyperplanes, by replacing $a$ by a nearby point we may assume that $a$ is not on the boundary of any dyadic cube. Pick $n_0$ large enough that
\begin{itemize}
\item The dyadic cubes $Q_i\in\mathcal{Q}_{n_0}^d$ containing $a_i$ are all disjoint (this is possible since $a\notin\Delta$),
\item The transversality conditions continue to hold if $V$ is replaced by the tangent plane at $y$ to $\{ P_0 = P_0(y)\}$ for all $y\in Q_1\times\cdots\times Q_m$ (this is possible by continuity).
\end{itemize}
Let $f_i, i\in[m]$ be the homotheties mapping $[0,1]^d$ to $\overline{Q_i}$, and let $\nu_n^{(i)}=f_i^{-1}\nu_n|_{Q_i}$ be the restriction of the original process to the cubes $Q_i$, rescaled back to the unit cube. Let us condition on the cubes $Q_1,\ldots,Q_m$ surviving, so that $\nu_n^{(i)}$ become multiples of independent copies of $\nu^{\text{perc}(d,p)}$. Now consider the polynomial
\[
\overline{P}_0(x_1,\ldots,x_m) = P_0(f_1(x_1),\ldots,f_m(x_m)).
\]
We have set things up so that $\overline{P}_0$ satisfies the conditions of Theorem \ref{thm:nonlinear}, which we can then apply to $\nu_n^{(i)}$ as above. An application of Lemma \ref{lem:Yt-survives-nonlinear} concludes the proof.
\end{proof}

The next straightforward lemma provides a convenient way to verify the transversality assumptions in the last corollary in the case when $q\le d$.
\begin{lemma} \label{lem:verify-transversality}
Let $P\in\mathcal{P}_{r,q,md}$ and let $a=(a_1,\ldots,a_m)\in\R^{md}$. If either of the following conditions hold, then $P$ and $a$ satisfy the assumptions of Corollary \ref{cor:nonlinear}.
\begin{enumerate}
\item $d=q$, and for each $j\in[m]$, the vectors $(\partial P_i/\partial x_j(a))_{i=1}^q$ are linearly independent.
\item $d>q$, and for each $j\in[m]$ and $k\in [d]$, the vectors $(\partial P_i/\partial x_j(a))_{i=1}^q$ together with the canonical vector $e_k=(0,\ldots,1,\ldots,0)\in\R^d$ are linearly independent.
\end{enumerate}
\end{lemma}
\begin{proof}
Fix $I\subsetneq [m]$, let $L_1=DP(a)$ and $L_2(x)=(x_i)_{i\in I}$. The tangent plane to $\{P=P(a)\}$ at $a$ is $\text{ker}(L_1)$, and $\text{ker}(L_2)=\{x_i=0, i\in I\}$. Note also that for the map $L(x)=(L_1 x, L_2 x)$, $\text{ker}(L)=\text{ker}(L_1)\cap\text{ker}(L_2)$. Hence, transversality with respect to $\text{ker}(L_2)$ follows if $L$ has full rank $q+d|I|$. Suppose then that there is a non-trivial linear combination among the rows of $L_1$ and $L_2$. Since the rows of $L_2$ are clearly linearly independent, some row of $L_1$ has a non-zero coefficient. Pick $j\in [m]\setminus I$. Since projections respect linear dependency, there is a non-trivial linear combination among the rows $L_1, L_2$ corresponding to $x_j$, but for $L_2$ these rows consist of $0$'s, so $(\partial P_i/\partial x_j(a))_{i=1}^q$ is linearly dependent, contradicting the assumption. Observe that taking $I=\varnothing$ we get that $a$ is a regular point of $P$.

Note that transversality with respect to the planes $\{x_{\ell}^k\}$ for $\ell\in I$ follows from the above considerations. It remains to check transversality also with respect to the planes of the form $\{ x_\ell^k=0\}$ where $\ell\notin I$. If $|I|<m-1$, then the same argument as above applies, taking $j\notin I\cup\{\ell\}$. If $I=m-1$ and $d=q$, then we already have $\text{ker}(L)=0$ and we are done. Finally, if $|I|=m-1$, and $q>d$, we apply the same argument as above with $j\notin I$, using that linear independence still holds when a canonical vector is added.
\end{proof}

We now provide several applications of Corollary \ref{cor:nonlinear} to the existence of various geometric configurations inside $A^\text{perc}$. We start with a small extension of a Theorem of Rams and Simon \cite{RamsSimon14} asserting that the set of distances between points of $A^\text{perc}$ has non-empty interior whenever $s>\tfrac12$; we show that in this case $A^\text{perc}$ contains all sufficiently small distances. This stronger form can be easily deduced from the result of Rams and Simon together with Lemma \ref{lem:0-1-strong}; we give an alternative proof using Theorem \ref{thm:nonlinear}
because the proof is a model case for the later applications, and is also a simple example of a situation where the linear result, Theorem \ref{thm:independent-product-linear-int}, cannot be directly used.
\begin{cor}\label{cor:distance}
Suppose $s>\tfrac12$, and write $D(A)=\{|x-y|\,:\,x,y\in A\}$, where $A=A^{\text{perc}(d,p)}$. Then:
\begin{enumerate}
\item Almost surely, there is $\e>0$ such that $(0,\e)\subset D(A)$.
\item For any $\e>0$, $\PP\left((0,\sqrt{d}-\e)\subset D(A)\right)>0$.
\end{enumerate}
\end{cor}

\begin{proof}
For each $\lambda>0$, let $P_\lambda(x,y)=|x-y|^2-\lambda^2 \in\cP^\text{reg}_{2,1,2d}$.

If $\lambda_0\in (0,\sqrt{d})$, we can find $x_0,y_0\in [0,1]^d$ such that $|x_0-y_0|=\lambda$ and the $d$ coordinates of $x_0-y_0$ are all non-zero. A calculation shows that the hypotheses of Corollary \ref{cor:nonlinear} for $P_{\lambda_0}$ and $(x_0,y_0)$ are met, so there exists an interval $I_{\lambda_0}$ around $\lambda_0$ such that
\[
\PP(I_{\lambda_0}\subset D(A)) > 0.
\]

By Harris' inequality (Lemma \ref{lem:Harris}), this implies that e.g.
\[
\PP([1/4,1/2]\subset D(A))=:c(d,p) >0 .
\]
The zero-one law from Lemma \ref{lem:0-1-strong} then yields the first claim, while the second claim follows using the first and Harris' inequality once  again.
\end{proof}

Next, we show that $A^{\text{perc}(d,p)}$ contains simplices of all small positive volumes, provided $s>1/(d+1)$. Related to this, in \cite[Theorem 3.7]{GGIP15} (see also \cite{GIM15}) it is shown that for some explicit $\e_d>0$, if $A\subset\R^d$ has Hausdorff dimension $d-\e_d$, then the simplices determined by $d+1$ points in $A$ determine a positive measure of volumes. Again, by considering the special case of fractal percolation, we get a sharp bound for this phenomenon.
\begin{cor} \label{cor:volumes}
Suppose $s>\tfrac1{d+1}$. Then a.s. there exists $\varepsilon>0$ such that $A=A^{\text{perc}(d,p)}$ contains $d+1$ points determining a simplex of volume $v$ for each $v\in (0,\varepsilon)$.
\end{cor}
\begin{proof}
Given $v>0$, let $P_v\in\mathcal{P}_{d,1,d(d+1)}$ be given by
\[
P_v(x_1,\ldots,x_{d+1})=\det(M(x_1,\ldots,x_{d+1}))-d! v\,,
\]
where $M(x_1,\ldots,x_{d+1})$ is the matrix with columns $(x_1,1),\ldots,(x_{d+1},1)$. If $P_v(x)=0$, then the simplex with vertices $x_i$ has volume $v$ (the reciprocal is not true since the determinant gives the oriented volume, so it could be $-v$ as well). Note that $\partial P/\partial x_i^k(z_1,\ldots,z_{d+1})$ is $\pm$ the $(d-1)$-volume of the simplex with vertices $\wt{\pi}_k(z_j), j\in[m]\setminus\{i\}$, where $\wt{\pi}_k(x_1,\ldots,x_{d+1})=(x_1,\ldots,x_{k-1},x_{k+1},\ldots,x_{d+1})$. Hence if $z_1,\ldots,z_{d+1}$ are generic points in $]0,1[^d$, then Lemma \ref{lem:verify-transversality}, and hence Corollary \ref{cor:nonlinear}, apply to $P_v$ and $t_v z=(t_v z_1,\ldots,t_v z_{d+1})$ where $t_v\in (0,1)$ is chosen so that $P_v(t_v z)=0$ (after possibly permuting the $z_i$ to ensure $\det(M(z))>0$). The zero-one law and Harris' inequality can then be used to conclude the proof, as in the proof of Corollary \ref{cor:distance}.
\end{proof}

Now, we look at congruent copies of triangles in $\R^2$.
\begin{cor}\label{cor:isometric-triangles}
Suppose $s>1$. Then for every triple $z=(z_1,z_2,z_3)$ of non-collinear points in $]0,1[^2$, there exists a neighbourhood $\mathcal{O}$ of $z$ such that with positive probability, $A^{\text{perc}(2,p)}$ contains an isometric copy of $\{ y_1,y_2,y_3\}$ whenever  $(y_1,y_2,y_3)\in\mathcal{O}$.
\end{cor}
\begin{proof}
Given $y=(y_1,y_2,y_3)\in(\R^2)^3$, let
\[
P_y(x_1,x_2,x_3)=(|x_2-x_1|^2 - |y_2-y_1|^2, |x_3-x_1|^2-|y_3-y_1|^2, |x_3-x_2|^2-|y_2-y_1|^2).
\]
Then $P_y\in\mathcal{P}_{2,3,6}$, and if $P_y(x_1,x_2,x_3)=0$, then $\{ x_1,x_2,x_3\}$ is isometric to $\{ y_1,y_2,y_3\}$. Note that Lemma \ref{lem:verify-transversality} is not applicable since $q>d$ in this case, but one can directly verify the assumptions from Corollary \ref{cor:nonlinear} as follows. Firstly, $DP_z(z)$ has full rank. By rotating the given $z_1,z_2,z_3$, we may assume $z_i-z_j$ is not in a coordinate line for $i\neq j$. Then for any $j\in [3]$ and any $k\in [2]$, one can verify that $\dim(\text{ker}(DP_z(z))\cap \{ x_i^k=0\})=2$, and $\text{ker}(DP_z(z))\cap \{ x_i=0\}$ is a one-dimensional subspace not contained in a coordinate hyperplane. We can then apply Corollary \ref{cor:nonlinear} to obtain the desired statement.
\end{proof}

\begin{rem} \label{rem:problem-isometric-copies}
The proof of Corollary \ref{cor:isometric-triangles} does not extend to simplices (or other polyhedra) in higher dimensions, since one can see that the transversality conditions fail. Geometrically, the reason is the following: to find an isometric copy of a simplex of vertices $z_1,\ldots,z_{d+1}$ (with the same vertex order), the first coordinate $y_1$ is free, but $y_2$ is constrained to lie in the sphere of center $y_1$ and radius $|z_2-z_1|$, so once $y_1$ is fixed, the angle of the corresponding tangent plane with the coordinate plane $y_2=0$ is zero. More precisely, the transversality of $\text{ker}DP$ (for the corresponding polynomial $P$) fails with respect to the planes $H^{I}$, when $|I|=2$.

This issue does not arise in $\R^2$ because once the second vertex is fixed, the third vertex (and therefore the full triangle) is completely determined. Formally, this means that $\dim\text{ker}(DP_z(z))\cap H^{I}=0$ whenever $|I|\ge 2$.
\end{rem}

\subsection{Scale-invariant patterns}

So far, in \S \ref{ssec:nlc}, we have considered configurations that are not scale invariant, in the sense that they are not preserved under scalings of the set where one is seeking the configuration. For this reason, in Corollaries \ref{cor:distance}-\ref{cor:isometric-triangles}, one can not hope to have \emph{all} the corresponding configurations (distances, volumes or triangles) in the fractal percolation limit set. We now turn to a class of configurations which are scale, and indeed homothety-invariant. Thanks to Corollary \ref{cor:0-1}, for this kind of patterns we will be able to obtain more pleasant results: a.s. $A^{\text{perc}(d,p)}$ will contain \emph{all} the configurations in each class.

We start by considering the angles determined by $A^{\text{perc}(d,p)}$. The problem of what lower bound on the Hausdorff dimension ensures that a subset of $\R^d$ contains a given angle was investigated in \cite{Mathe12, HKKMMMS13} and is far from settled. For fractal percolation, we get the following result:
\begin{cor}\label{cor:angles}
Suppose $s>\tfrac13$. Then, almost surely, all angles in $]0,\pi[$ can be formed by three points in $A^{\text{perc}(d,p)}$.
\end{cor}

\begin{proof}

For each $\lambda\in]-1,1[$, consider the polynomial
\[
P_\lambda(x_1,x_2,x_3)=((x_1-x_2)\cdot(x_3-x_2))^2-\lambda^2|x_1-x_2|^2|x_3-x_2|^2\in\cP_{4,1,3d}.
\]
Note that $x=(x_1,x_2,x_3)\in\R^{3d}$ satisfies $\cos(\angle x_1x_2x_3)=\lambda$ if and only if $P_\lambda(x)=0$.

Fix $\lambda_0\in]-1,1[$  and pick a point $x'=(x'_1,x'_2,x'_3)\in]0,1[^{3d}$ such that $P_{\lambda_0}(x')=0$ and, moreover, $x'_i-x'_j$ does not lie in a coordinate plane for $i\neq j$. Then Lemma \ref{lem:verify-transversality} can be used to show that $P_{\lambda_0}$ and $x'$ satisfy the hypotheses of Corollary \ref{cor:nonlinear}. We deduce that there is an open interval $I_{\lambda_0}$ containing $\cos^{-1}(\lambda_0)$ such that, with positive probability,  $A^{\text{perc}(d,p)}$ contains all angles in $I_{\lambda_0}$. By the zero-one law for fractal percolation (Corollary \ref{cor:0-1}), `positive probability' can be replaced by `full probability'. Finally, covering $]0,\pi[$ by countably many such intervals $I_{\lambda_0}$ we reach the final conclusion.
\end{proof}

\begin{rem}
The extreme angles $\alpha\in\{0,\pi\}$ correspond to three collinear points in $A$. For the existence of three points on a line in $A^\text{perc}$, the threshold is $s>(d-1)/3$, which can be seen by applying Lemma \ref{lem:Yt-survives} to the plane
$V=\{(x,y,z)\in(\R^d)^3\,:\,z=x+\lambda(y-x)\text{ for some }\lambda\in\R\}\in\AA_{3d,2d+1}$.
\end{rem}

Another problem that has received a lot of attention is: what geometric conditions on a subset $A\subset\R^d$ ensure that $A$ contains the vertices of an equilateral triangle (or, more generally, a similar copy of a given triangle)? In particular, does any lower bound on the Hausdorff dimension of a set $A\subset\R^d$ suffice? For $d=2$, the answer is no: there exists sets $A\subset\R^2$ of Hausdorff dimension two which do not contain the vertices of an equilateral triangle, see \cite{Falconer92}, \cite{Maga10}. However, the answer is yes, for similar copies of any triangle, if one additionally assumes a Fourier decay condition on a Frostman measure on $A$ \cite{Chanetal16}, although even in this case the lower bound on the dimension is not explicit. Recently, A. Iosevich and B. Liu \cite{IosevichLiu16} showed that for $d\ge 4$ there is $\e_d>0$ (still not explicit) such that any set $A\subset\R^d$ of Hausdorff dimension $>d-\e_d$ does contain the vertices of an equilateral triangle. The problem is open in dimension $d=3$. For fractal percolation sets, we have the following result:
\begin{cor} \label{cor:triangles}
Suppose $s>2/3$. Then almost surely, $A=A^{\text{perc}(d,p)}$ contains the vertices of a triangle similar to an arbitrary non-degenerate triangle in $\R^d$.
\end{cor}
\begin{proof}
Given $a,b>0$ such that $a+b>1$, let $P_{a,b}:\R^{3d}\to \R^2$ be given by
\[
P_{a,b}(x_1,x_2,x_3) = (|x_3-x_1|^2 - a^2 |x_2-x_1|^2,|x_3-x_2|^2 - b^2 |x_2-x_1|^2).
\]
Then $P_{a,b}(x_1,x_2,x_3)=0$ if the triangle $(x_1,x_2,x_3)$ is similar to a triangle with side-lengths $1,a,b$. Fix, then, $a,b$ as above, and let $y_1,y_2,y_3\in\R^d$ be such that $P_{a,b}(y_1,y_2,y_3)=0$ and $y_i-y_j$ does not lie in a coordinate hyperplane for $i\neq j$. We can then apply Lemma \ref{lem:verify-transversality}, Corollary \ref{cor:nonlinear} and conclude the proof as in Corollary \ref{cor:angles}.
\end{proof}

For similar reasons to those explained in Remark \ref{rem:problem-isometric-copies}, our methods do not directly apply to the problem of finding similar copies of higher-dimensional polyhedra inside $A^{\text{perc}(d,p)}$. However, in the plane, we can extend Corollary \ref{cor:triangles} to general polygons:
\begin{cor} \label{cor:polygons}
If $s>2-4/m$, then a.s. $A^{\text{perc}(2,p)}$ contains a similar copy of all $m$-point sets $\{ a_1,\ldots,a_m\}$ such that no three of the $a_i$ are collinear.
\end{cor}
\begin{proof}
The starting point of the proof is the following observation: Suppose that $\{x'_1,\ldots, x'_m\}$ is a similar copy of $\{a'_1,\ldots,a'_m\}$ such that $x'_j=h(a'_j)$ for all $j\in[m]$ and some $h\in\simi_2$. Then there are small neighbourhoods $U$ of $x'=(x'_1,\ldots,x'_m)$ and $\widetilde{U}$ of $a'=(a'_1,\ldots,a'_m)$ such that for all $(a_1,\ldots,a_m)\in\widetilde{U}$, $x\in U$ yields a similar copy of $a'$ if $|x_j-x_i|/|x_2-x_1|=|a_j-a_i|/|a_2-a_1|$ whenever $3\le i\le m$ and $j\in\{1,2\}$.

Now to the details. Let
\[\mathcal{N}=\{(i,j)\,:\,3\le i\le m\,,j\in\{1,2\}\}\,.\]
For any $a=(a_1,a_2,\ldots, a_m)\subset(\R^2)^m$, where no three of the $\{a_1,\ldots, a_m\}$ are collinear,
denote $a(i,j)=|a_j-a_i|/|a_2-a_1|$, for all $(i,j)\in\mathcal{N}$, and define
\[
P_{i,j,a(i,j)}(x_1,\ldots,x_m)=|x_j-x_i|^2-a(i,j)^2|x_2-x_1|^2
\]
for $(x_1,\ldots,x_m)\in(\R^2)^m$. Finally, consider the polynomial
\[
P_a=(P_{i,j,a(i,j)})_{(i,j)}\in\cP_{2,2m-4,2m}.
\]

Fix $a'\in(]0,1[^d)^m\setminus\Delta$. As we have observed, if $x'$ is close to $a'$ and $P_{a'}(x')=0$, then $x'$ is similar to $a'$. Our goal is to verify that the assumptions of Corollary \ref{cor:nonlinear} apply to $P_{a'}$ and $a'$. We cannot apply Lemma \ref{lem:verify-transversality} directly, but we argue analogously as follows: let $1\le i_0<j_0\le m$. It is enough to show that the rows of $DP_{a'}(a')$ together with the vectors $e_{i_0}^k, e_{j_0}^k$, $k\in\{1,2\}$, are linearly independent (and hence a basis). Here $e_i^k$ is the canonical vector for the coordinate $x_i^k$. Suppose, on the contrary, that there is a non-trivial linear combination. Since $e_{i_0}^k, e_{j_0}^k$ are linearly independent, some $\nabla P_{i,j,a'(i,j)}(a'_1,\ldots,a'_m)$ must have a non-zero coefficient. Pick $\ell\in \{1,2,i,j\}\setminus \{i_0,j_0\}$. Then there is a non-trivial linear combination among $\partial P_{i,j,a'(i,j)}/\partial x_\ell(a'_1,\ldots,a'_m)$, but a calculation (using the non-collinearity of the $a'_i$) shows this is not the case.

Hence  Corollary \ref{cor:nonlinear}  can be applied, and together with Corollary \ref{cor:0-1} and Lemma \ref{lem:Harris} this concludes the proof.
\end{proof}

\begin{rem}
A straightforward adaptation of the proof of Proposition \ref{prop:threshold_sharp}, one can again see that the dimension thresholds provided by Corollaries \ref{cor:distance}--\ref{cor:polygons} are sharp for packing dimension even for deterministic sets (up to the endpoint).
\end{rem}

\section{Patterns in sets of positive measure}
\label{sec:szemeredi}

\subsection{The dimension of intersections and patterns}

We refine the results in Sections \ref{sec:affine-intersections}--\ref{sec:nonlinear} by providing upper and lower bounds for the dimension of intersections of $A^m$ with algebraic varieties. As a direct corollary, this yields dimension estimates for the `number of times' a given configuration is found in the set $A$.

\begin{thm} \label{thm:dim_patterns}
Let $A=A^{\text{perc}(d,p)}$. In the setting of Theorem \ref{thm:nonlinear}, if $s>q/m$, there is a neighbourhood $\mathcal{O}$ of $P_0$ such that:
\begin{enumerate}
\item For every $\delta>0$ there is a finite random variable $K$ such that $A^m\cap \zero_P\setminus\Delta(\delta)$ can be covered by $K 2^{n(ms-q)}$ cubes in $\mathcal{Q}_n^{md}$ for all $P\in\mathcal{O}$. In particular, a.s.
\[
    \dim_P(A^m\cap \zero_P\setminus\Delta)\le ms-q \quad\text{for all }P\in\mathcal{O}.
\]
\item For any $\delta>0$, with positive probability, $\dim_H(A^m\cap \zero_P\setminus\Delta)>ms-q-\delta$ for all $P\in\mathcal{O}$.
\end{enumerate}
\end{thm}
\begin{proof}
To begin, we can partition $[0,1]^{md}\setminus\Delta$ into a countable union of products $Q:=Q_1\times\ldots\times Q_m$, where $Q_i\in\mathcal{Q}_n^d$ for some (variable) $d$ and $(\pi_j(Q_i))_{i=1}^m$ are disjoint for all $j$. By considering the restriction to each of these cubes, and applying the usual independence and rescaling arguments, we may replace $A^m$ by the product of $m$ independent fractal percolations (for the upper bound we are using the countable stability of packing dimension). Let $A'$ denote the product of the $m$ copies of $A^{\text{perc}(d,s)}$.

Write  $\mu_n = \nu_n^{(1)}\times\cdots\times \nu_n^{(m)}$ with $\nu_n^{(i)}$ independent realizations of $\nu_n^{\text{perc}(d,p)}$. To establish the first claim, suppose that $\zero_P$ intersects $M$ disjoint cubes $Q\in\mathcal{Q}^{md}_n$ for some $P\in\mathcal{O}$. Write $P_y=P-y$, and note this is in $\mathcal{O}$ for small enough $y$. By the coarea formula (Proposition \ref{prop:coarea}), using that the $d$-Jacobian of $P_y$ is bounded away from zero for $y$ small, we deduce that if $n$ is large enough, then
\begin{equation}\label{eq:co-area}
M 2^{-nms}\le O(1)\int_{y\in B(0,O(2^{-n}))\subset\R^q}Y_n^{P_y}\,d\mathcal{L}^q\le O(K 2^{-nq})\,.
\end{equation}
This shows that $M\le O_K(1) 2^{n(ms-q)}$, so that the packing (and indeed the box-counting) dimension of $\zero_P$ is at most $ms-q$ for all $P\in\mathcal{O}$, establishing the first claim.

For the second claim, fix $t\in (0,md)$ and let $\widetilde{A}=\widetilde{A}^{\text{perc}(md,q)}$ be fractal percolation on $\R^{md}$ of Hausdorff dimension $md-t$, independent of $\mu_n$, with corresponding approximating measures $\widetilde{\mu}_n$. Now set
\[
\overline{\mu}_n = \mu_n \widetilde{\mu}_n.
\]
The sequence $\overline{\mu}_n$ is neither fractal percolation nor a product of independent percolations. However, it is easily checked to be a martingale measure, and the dependency structure matches exactly that of the product of independent copies $\mu_n$: there are dependencies only among coordinate directions. In particular, $\overline{\mu}_n$ converges almost surely to a measure $\overline{\mu}$ supported on $A'\cap \widetilde{A}$, which has dimension $ms-t$. The proof of Theorem \ref{thm:nonlinear} apply verbatim to $\overline{\mu}$, while the proof of Lemma \ref{lem:Yt-survives-nonlinear} (or rather Lemma \ref{lem:Yt-survives}) extends with very minor changes (the estimates for the $L^2$ norm actually get better since there is more independence than in the setting of Lemma \ref{lem:Yt-survives-nonlinear}). We deduce that if $ms-t>q$, then
\[
\PP\left((A'\cap\zero_P)\cap \wt{A} \neq \varnothing \text{ for all }P\in\mathcal{O}\right) > 0,
\]
where $\mathcal{O}$ is a neighbourhood of $P_0$ which is independent of $s$ and $t$ (so long as $ms-t>q$). Since $A'$ and $\wt{A}$ are independent, this implies that there are positively many realizations of $A'$ such that the above holds  with positive probability with respect to the construction of $\wt{A}$. Fixing such a realization of $A'$, Theorem \ref{thm:hausdorff-dim-percolation} allows us to conclude that $\dim_H(A'\cap \zero_P)\ge t$ for all $P\in\mathcal{O}$. Since $t<ms-q$ is arbitrary, we get the second claim.
\end{proof}

\begin{rems} \label{rems:dim-patterns}
\begin{enumerate}[(i)]
\item The same result holds in the setting of Theorem \ref{thm:independent-product-linear-int} (with $k$ in place of $md-q$), either by noting that the proof works works verbatim in that case, or by seeing Theorem \ref{thm:independent-product-linear-int} as (essentially) the particular case of Theorem \ref{thm:nonlinear} in which the polynomial $P$ is affine.
\item It seems very likely that using the approach of \cite[Theorem 12.8]{ShmerkinSuomala14}, one can improve the second part of the theorem as follows: almost surely, for every $P\in\mathcal{O}$ such that $Y^P>0$,
\[
\dim_H(A'\cap \zero_P\setminus\Delta)\ge ms-q.
\]
Recall that $Y^P>0$ for an open set of $P$'s. For the sake of simplicity, we will work with the slightly weaker version above.
\end{enumerate}
\end{rems}

For a fixed $P$, we also have the following upper bound, without any transversality assumptions:
\begin{lemma} \label{lem:dim-upper-bound-no-assumptions}
Let $P\in\mathcal{P}_{r,md,q}^{\text{reg}}$, and let $A=A^{\text{perc}(d,p)}$. Assume $s>q/m$. Then almost surely, for any $\delta>0$,
\[
\overline{\dim}_B(\zero_P \cap A^m\setminus \Delta(\delta)) \le ms-q.
\]
\end{lemma}
\begin{proof}
We use the first moment method.  Let $K_n$ be the number of cubes $Q\in\mathcal{Q}_n^{md}$ such that $Q\cap A_n^m\cap \zero_P\setminus\Delta(\delta)\neq \varnothing$. For large $n$ (in terms of $\delta$) each cube $Q\subset [0,1]^{md}\setminus\Delta(\delta)$ survives in $A_n^m$ with probability $2^{nm(s-d)}$, so for each $\e>0$ Markov's inequality yields
\[
\PP(K_n> 2^{n(ms-q+\e)}) \le  2^{-n\e} 2^{n(q-ms)} \EE(K_n) = O(2^{-\e n})\,,
\]
note that $\zero_P$ intersects $O(2^{n(md-q)})$ cubes in $\mathcal{Q}_n^{md}$.
By the Borel-Cantelli Lemma, a.s. $K_n\le 2^{n(ms-q+\e)}$ for all large $n$, which gives the claim.
\end{proof}

As a corollary of Theorem \ref{thm:dim_patterns}, we can now prove Theorem \ref{thm:dim-patterns-homothetic}:
\begin{thm}\label{cor:dim_patterns}
If $s>d-(d+1)/m$, then a.s. for each $m$ point set  $S\subset\R^d$,
\[\dim\left\{(a,b)\in \R\times \R^d\,:\,aS+b\subset A\right\}=m(s-d)+d+1\,,\]
where $\dim$ is either Hausdorff or packing dimension.
\end{thm}

\begin{proof}
By covering the parameter space by countably many neighbourhoods $\mathcal{O}$, we can work with a fixed $\mathcal{O}$ for which the conclusions of Theorem \ref{thm:dim_patterns} hold. The upper bound is a direct consequence of the first part of Theorem \ref{thm:dim_patterns}. For the lower bound, we apply the second part of Theorem \ref{thm:dim_patterns} to each $\delta>0$, use Corollary \ref{cor:0-1} to upgrade positive probability to full probability, and then let $\delta\downarrow 0$. As in the proof of Corollary \ref{cor:patterns}, we need to consider the dyadic metric to bypass the failure of transversality with respect to the coordinate hyperplanes for certain patterns.
\end{proof}

\begin{rem}
Similar results hold for other classes of configurations. For patterns which are not scale-invariant, we cannot in principle use the zero-one law to get rid of the $\delta$ in the lower bound, but see Remarks \ref{rems:dim-patterns}(ii) for a possible approach to overcome this. If transversality with respect to coordinate hyperplanes fails at some points, then we do not in get a uniform upper bound for the dimension (but we do for each given configuration, thanks to Lemma \ref{lem:dim-upper-bound-no-assumptions}). However, as pointed out in Remarks \ref{rem:nodiagonal}(ii), it may be possible to remove the assumption of hyperplane transversality altogether.
\end{rem}

\subsection{Lack of patterns in sets of full measure}

As an application of Lemma \ref{lem:dim-upper-bound-no-assumptions}, we show that whenever $s<\tfrac{q}{m-1}$ it is possible to find a full $\nu$-measure subset $A'\subset A$ such that $(A')^n\cap \zero_P=\varnothing$. We will show in the next section that the opposite happens when $s>\tfrac{q}{m-1}$.
\begin{prop}\label{prop:dim_threshold}
Let $\nu=\nu^{\text{perc}(d,p)}$ and $A=A^{\text{perc}(d,p)}$. Let $\nu=\nu^{\text{perc}(d,p)}$ and $A=A^{\text{perc}(d,p)}$. If $P\in\mathcal{P}_{r,q,md}^{\text{reg}}$ and $s=s(d,p)<\tfrac{q}{m-1}$,  then a.s there is a Borel set $A'\subset A$ of full $\nu$-measure such that $A'$ does not contain distinct points $x_1,\ldots,x_m$ with $P(x_1,\ldots,x_m)=0$.
\end{prop}
\begin{proof}
We know from Lemma \ref{lem:dim-upper-bound-no-assumptions} that a.s.
\[
\dim_P(\zero_P\cap A^m\setminus\Delta) \le ms-q< s\,
\]
using the assumption $s<\tfrac{q}{m-1}$ for the second inequality. Let
\[
A' =  A\setminus  \pi_1(\zero_P\cap A^m\setminus\Delta)\,.
\]
Since $A\setminus A'$ has dimension $<s$, we have $\nu(A')=\nu(A)$. On the other hand, it is clear from the definition that $A'$ cannot contain distinct points $x_1,\ldots,x_m$ such that $P(x_1,\ldots,x_m)=0$.
\end{proof}

Assuming the same transversality conditions as in Corollary \ref{cor:nonlinear}, the previous Proposition extends to the more delicate case of the threshold $s=q/(m-1)$.
\begin{thm}\label{thm:sharp_Szemeredi_threshold}
Let $\nu=\nu^{\text{perc}(d,p)}$ and $A=A^{\text{perc}(d,p)}$. If $P\in\mathcal{P}_{r,q,md}^{\text{reg}}$ satisfies the Assumptions of Corollary \ref{cor:nonlinear} and $s=s(d,p)=\tfrac{q}{m-1}$,  then a.s. there is a Borel set $A'\subset A$ of full $\nu$-measure such that $A'$ does not contain distinct points $x_1,\ldots,x_m$ with $P(x_1,\ldots,x_m)=0$.
\end{thm}

\begin{proof}
As in the proof of Proposition \ref{prop:dim_threshold}, it is enough to show that a.s.
\begin{equation}\label{eq:p1null}
\nu(\pi_1(\zero_P\cap A^m\setminus\Delta))=0\,.
\end{equation}
Covering $\R^{md}\setminus\Delta$ by cubes $Q\in\mathcal{Q}_{n}^{md}$, $n\in\N$, with $Q\cap\Delta=\varnothing$ and conditioning on $Q\subset A_n^m$, we see that \eqref{eq:p1null} is implied by
\begin{equation}\label{eq:p1null_indep}
\nu(\pi_1(A^{(1)}\times\cdots\times A^{(m)}\cap\zero_P))=0\,,
\end{equation}
where $A^{(1)},\ldots,A^{(m)}$ are independent fractal percolations, and $\nu=\nu^{(1)}$ is the fractal percolation measure on $A^{(1)}$.

Given $a\in\R^d$,
define $\widetilde{P}_a\in\mathcal{P}_{r,q,(m-1)d}$ by
\[\widetilde{P}_a(x_2,\ldots,x_m)=P(a,x_2,\ldots,x_m)\,.\]
We claim that for each large enough $n$ and each $Q\in\mathcal{Q}^{d}_n$, there is a set $\{a_1,\ldots, a_\ell\}$ with $\ell=O_{P}(1)$, such that $A_n^{(2)}\times\cdots\times A_n^{(m)}\cap\pi^1(\pi_1^{-1}(Q\cap \zero_{P}))\neq\varnothing$ only if
\[
A_n^{(2)}\times\cdots\times A_n^{(m)}\cap\zero(\widetilde{P}_{a_i})\neq\varnothing
\]
for some $a_i\in\{a_1,\ldots,a_\ell\}$. To that end, let $Q'\subset A^{(2)}_n\times\cdots\times A^{(m)}_n$, $Q'\in\mathcal{Q}^{(m-1)d}_n$ and denote by $x'$ the center point of $Q'$. Suppose $Q'\cap\pi^1(\pi_1^{-1}(Q\cap \zero_{P}))\neq\varnothing$ so that $P(x_1,x_2,\ldots,x_m)=0$ for some $x_1\in Q$, $(x_2,\ldots,x_m)\in Q'$. Note that $P(x_1,x')=O_P(2^{-n})$. Letting $y=G(P_1,P_2,x_1)$, where $P_i\colon\R^d\to\R^q$, $P_1(x)=P(x,x')-P(x_1,x')$, $P_2(x)=P(x,x')$, and $G$ is as in Lemma \ref{lem:isotopic-curves}, we have that $P(y,x')=0$ and $|y-x_1|=O_P(2^{-n})$. Here the use of Lemma \ref{lem:isotopic-curves} is justified because the tangent planes of $\zero_P$ are uniformly transversal with respect to the coordinate plane $\pi_1(\R^{md})$.

Let $C$ be a sufficiently large constant to be chosen later, and let $\{a_1,\ldots, a_\ell\}$ be a maximal $C^{-1} 2^{-n}$ separated subset of $CQ$ (the cube with the same centre as $Q$ and side length $C2^{-n}$). Pick $a\in\{a_1,\ldots,a_\ell\}$ such that $|y-a|\le C^{-1}2^{-n}$ and let $z=\widetilde{G}(\widetilde{P}_1,\widetilde{P}_2,x')$, where $\widetilde{P}_1,\widetilde{P}_2\colon\R^{(m-1)d}\to\R^q$, $\widetilde{P}_1(x)=P(y,x)$, $\widetilde{P}_2(x)=P(a,x)=\widetilde{P}_{a}(x)$ and $\widetilde{G}$ is again provided by Lemma \ref{lem:isotopic-curves} (now using the transversality with respect to the plane $\pi^1(\R^{md})$). We conclude that $z\in\zero(\widetilde{P}_{a})$ and $|z-x'|=O_P(C^{-1} 2^{-n})<2^{-n-1}$ so that $z\in Q'$ provided $C$ is large enough depending on $P$. Thus, if $C=O_P(1)$ is large enough, then $\{a_1,\ldots,a_\ell\}$ is the desired family.

Now fix $n\gg 1$ and $Q\in\mathcal{Q}^d_n$, and let $\mathfrak{F}_Q$ denote the event
\[
A^{(2)}_n\times\cdots\times A^{(m)}_n\cap\pi^1(\pi_1^{-1}(Q\cap \zero_{P}))\neq\varnothing.
\]
Then $\mathfrak{F}_Q$ is independent of the realization of $(A_n^{(1)})_n$, and
\[
\PP(\mathfrak{F}_Q) \le \sum_{i=1}^\ell \PP(A_n^{(2)}\times\cdots\times A_n^{(m)}\cap\zero(\widetilde{P}_{a_i})\neq\varnothing).
\]
Note that $\zero(\widetilde{P}_a)$ can be covered by $O(1) 2^{n((m-1)d-q)}$ cubes in $\mathcal{Q}_n$ for $a\in [0,1]^d$ (with the $O$ constant independent of $a$); the proof is the same as \eqref{eq:co-area}, together with compactness. Since $s=q/(m-1)$, Corollary \ref{cor:extinction-quantitative} implies that for each $a\in [0,1]^d$,
\[
\PP(A^{(2)}_n\times\cdots\times A^{(m)}_n\cap \zero(\widetilde{P}_a)\neq \varnothing)\le q_n
\]
for some sequence $q_n\downarrow 0$ independent of $a$.  We deduce that $\PP(\mathfrak{F}_Q)=O(q_n)$.
Let us write $\wt{A} = A^{(1)}\times\cdots\times A^{(m)}$. Observe that
\[
\EE(\nu(Q\cap\pi_1(\wt{A}\cap\zero_P))) \le \PP(\mathfrak{F}_Q) \EE(\nu(Q)) = O(q_n) \EE(\nu(Q)),
\]
using the independence of $\nu$ and $\mathfrak{F}_Q$. Since $\EE(\|\nu\|)=1$, we conclude that
\[
\EE(\nu(\pi_1(\wt{A}\cap\zero_P))) = \sum_{Q\in\mathcal{Q}_n^d} \EE(\nu(Q\cap\pi_1(\wt{A}\cap \zero_P))) = O(q_n).
\]
Since $n\gg 1$ is arbitrary, this establishes \eqref{eq:p1null_indep}, as desired.
\end{proof}

\subsection{Patterns in sets of positive measure}

Having obtained the critical dimension for the existence of different patterns in fractal percolation, we next turn our attention to the following problem: what is the critical value $s_c$, such that  if $s>s_c$ then all `large' subsets $A'\subset A$ (i.e. the ones with $\nu(A')>0$) contain the required pattern? As described in the introduction, this question is motivated by various analogous results in the discrete setting, in particular the Green-Tao and Conlon-Gowers-Schacht theorems on the existence of arithmetic progressions inside positive density subsets of the primes and random discrete sets.

It turns out that the answer to this question is closely related to the dimension of the pattern structures investigated earlier in this section. Let us consider the problem of finding homothetic copies of a set $S$ for concreteness. As explained in the proof of Proposition \ref{prop:dim_threshold}, if the dimension of $\{ (\lambda,x):x+\lambda S\subset A\} $ is $<s$, then the points in $A$ that belong to at least one such $x+\lambda S$  also lie in a set of dimension $<s$, so we can remove all such points to end up with a full $\nu$-measure subset of $A$ which contains no homothetic copy of $S$. On the other hand, if $\dim_H\{ (\lambda,x):x+\lambda S\subset A\}>s$, then we will show that, almost surely, all subsets $A'\subset A$ with positive $\nu$ measure do contain a  homothetic copy of $S$. Indeed, we have the following abstract result for polynomial configurations:
\begin{thm} \label{thm:relative-S-abstract}
Let $\Lambda$ be an open subset of $\R^M$ for some $M\in\N$. Suppose $\lambda\mapsto P_\lambda$ is a continuous map from $\Lambda$ to $\mathcal{P}_{r,q,md}$ such that for each $\lambda\in\Lambda$,  the polynomial $P_\lambda$ satisfies the hypotheses of Corollary \ref{cor:nonlinear} (for some point $a=a(\lambda)$).

Suppose $s(d,p)>\tfrac{q}{m-1}$. Then a.s. there is a nonempty open set $U\subset\Lambda$ such that for all Borel sets $A'\subset [0,1]^{d}$ with $\nu(A')>0$ there exist $r>0$ and $t\in\R^d$ such that for all $\lambda\in U$ there are points $x_1,\ldots,x_m\in A'$ with $P_\lambda(r x_1+t,\ldots,r x_m+t)=0$.

If, furthermore,  $P_\lambda(x_1,\ldots,x_m)=0$ if and only if $P_\lambda(rx_1+t,\ldots,r x_m+t)=0$ for all $\lambda\in\Lambda$, $r>0$ and $t\in \R^{d}$, then a.s. for all $\lambda\in\Lambda$ and all Borel sets $A'$ with $\nu(A)'>0$ there are distinct points $x_1,\ldots,x_m\in A'$ such that $P_\lambda(x_1,\ldots,x_m)=0$.
\end{thm}

Combining this with (the proofs of) Corollaries \ref{cor:patterns}, \ref{cor:distance},  \ref{cor:volumes}, \ref{cor:isometric-triangles}, \ref{cor:angles}, \ref{cor:triangles}, \ref{cor:polygons}, we immediately obtain Theorem \ref{thm:relative_S}.

We note that the threshold for $s$ is sharp in a rather strong way, thanks to Theorem \ref{thm:sharp_Szemeredi_threshold}.
For non-scale invariant patterns, Theorem \ref{thm:sharp_Szemeredi_threshold}  has to be applied to a countable family of polynomials that witnesses a dense set of patterns in the appropriate family. For example, if $s\in (1/2,1]$, then, although $A$ contains all small distances a.s., there exists a set $A'\subset A$ of full $\nu$-measure which does not contain any rational distances.

\begin{rem}
Note that in Corollary \ref{cor:patterns_without_scaling}, we have $k=d$ so that $s<\tfrac{md-k}{m-1}$ for all $s\le d$. So there is no relative Szemer\'{e}di theorem for translated copies.
\end{rem}

We give the idea of the proof of Theorem \ref{thm:relative-S-abstract}  in the special case of homothetic copies. Suppose that $s>d-\tfrac{1}{m-1}$, and for $T=\{t_1,\ldots,t_{m-1},0\}$, let $V_T$ be as in the proof of Corollary \ref{cor:patterns}, and let $\Gamma$ be the family of all such planes. Then, an application of Theorem \ref{thm:independent-product-linear-int} to $(\nu_n)^{m-1}$ and $\mathcal{R}(\Gamma)$ implies  that for any fixed $Q\in\mathcal{Q}_n^{d}$, there cannot be more than
\[
O(2^{n(1+(m-1)(s-d))})
\]
cubes $Q'\in\mathcal{Q}^{md}_n$, $Q'\subset (A_n)^m$, $Q'\cap V_T\neq\varnothing$ with $\pi_j(Q')=Q$.  Note that the $O$-constant is random but uniform in $n$ and $T$. Recalling that $(A_n)^m\cap V_T$ intersects roughly $2^{n(d+1+m(s-d))}$ cubes in $\mathcal{Q}^{md}_n$, we observe that in order to violate $(A_n)^m\cap V_T\neq\varnothing$, we need to remove at least $\Omega(2^{ns})$ cubes from $A_n$. Letting $n\to\infty$, this essentially shows that $(A')^m\cap V_T\neq\varnothing$ whenever $A'\subset A$ has full $\mu$-measure. Finally, to pass from full $\nu$-measure to $\nu(A')>0$, we will use a density point argument.

We pass to the details. For $F\subset [0,1]^d$, let $N_n(F)$ denote the the number of cubes $Q$ in $\mathcal{Q}^d_n$ such that $Q\cap F\neq\varnothing$. Recall also the notation $N_n:=N_n(A)$ from Section \ref{sec:perco}. For the rest of this section, we assume that $\Lambda$, $\lambda\to P_\lambda$ satisfy the assumptions of Theorem \ref{thm:relative-S-abstract}, and $s>\tfrac{q}{m-1}$.  Moreover, we always denote $\nu=\nu^{\text{perc}(d,p)}$ and $A=A^{\text{perc}(d,p)}$ with $s=s(d,p)$.

\begin{prop} \label{prop:counting-szemeredi}
Fix $\lambda_0\in\Lambda$. There are $\e>0$, $n_0\in\N$, a cube $Q_0\in\mathcal{Q}_{n_0}^{md}$, and an open neighbourhood $U\subset\Lambda$ of $\lambda_0$  (all deterministic) such that the following event has positive probability: for any compact $A'\subset A$ such that $(A')^m \cap \zero(P_\lambda)\cap Q_0=\varnothing$ for some $\lambda\in U$, we have $\nu(A\setminus A')\ge \e$.
\end{prop}

\begin{proof}
In the course of the proof, $C_i, C'_i$ denote finite positive constants independent of $n$ or any dyadic cubes. By Lemma \ref{lem:Yt-survives-nonlinear} and (the proof of) Corollary  \ref{cor:nonlinear}, there are an open neighbourhood $U$ of $\lambda_0$ and a small dyadic cube $Q_0$ such that, with positive probability,
\begin{equation}\label{eq:Y_tunif_bound}
\int_{Q_0\cap \zero(P_\lambda)}(\nu_n)^m\,d\mathcal{H}^{md-q}\ge C_1\text{ for all }\lambda\in U, n\in\N\,.
\end{equation}
Indeed, one only needs to take $Q_0$ to be a small enough dyadic cube containing $a(\lambda_0)$ and disjoint from $\Delta$, such that the transversality conditions hold on $Q$ for all $\lambda\in U$, as in the proof of Corollary  \ref{cor:nonlinear}.

In particular, if $N_{n,\lambda}=N_{n,\lambda}(A_n)$ is the number of cubes $Q\in\mathcal{Q}^{md}_n$ such that $Q\subset (A_n)^m\cap Q_0$ and $Q\cap \zero(P_\lambda)\neq\varnothing$, then \eqref{eq:Y_tunif_bound} implies
\begin{equation}\label{eq:N_nTbound}
N_{n,\lambda}\ge C'_1 2^{n(ms-q)} \text{ for all }\lambda\in U, n\in\N\,.
\end{equation}

Let $\Gamma$ be the family of all tangent planes to $\zero(P_\lambda)$ at $x$ for $\lambda\in U$, $x\in Q_0$. Since $s>\tfrac{q}{m-1}$ and $\mathcal{R}(\Gamma)$ satisfies the required transversality assumptions (by our choice of $Q_0$ and $ U$), we can apply Theorem \ref{thm:nonlinear} to each polynomial in the family
\[
\widetilde{P}_{\lambda,t,j}(x_1,\ldots,x_m) = P_\lambda(x_1,\ldots,x_{j-1},t,x_{j+1},\ldots,x_m)  \in\mathcal{P}_{r,q,d(m-1)},
\]
with $\lambda\in U$, $t\in\pi_j(Q_0)$, $j\in [m]$. This yields a constant $C_2$ such that, letting $\mathcal{R}$ consist of the sets
\[
\zero(P_\lambda)\cap Q_0\cap\{x_j=t\}, \quad\lambda\in U, t\in[0,1]^d, j\in[m]\,,
\]
there is a positive probability that \eqref{eq:Y_tunif_bound} holds together with the bound
\begin{equation}\label{eq:Y'_unif_bound}
\int_{V'} \nu_n^{m-1}(x)\,d\mathcal{H}^{(m-1)d-q} \le C_2\quad\text{ for  all }V'\in\mathcal{R},n\in\N\,.
\end{equation}
Indeed, by Theorem \ref{thm:nonlinear} (and compactness, making $U$ smaller if needed), the probability that \eqref{eq:Y'_unif_bound} holds tends to $1$ as $C_2\uparrow\infty$.

Applying \eqref{eq:co-area} in the proof of Theorem \ref{thm:dim_patterns}, we deduce the following: for any $Q'\in\mathcal{Q}_n^d$ with centre $t'$, the number of cubes $Q\subset A_n^m$ such that $Q\cap\zero(P_\lambda)\cap Q_0\cap \{x_j=t'\}$ is bounded by $C'_2 2^{n((m-1)s-q)}$. On the other hand, applying Lemma \ref{lem:isotopic-curves} as in the proof of Theorem \ref{thm:sharp_Szemeredi_threshold}, we see
that if $x\in \zero(P_\lam) \cap Q_0 \cap \{x_j=t\}$ with $t\in Q'$, then there is $x'\in B(x,O(2^{-n}))\cap \zero(P_\lam)\cap Q_0 \cap \{ x'_j=t'\}$ where $t'$ is again the centre of $Q'$ (and the $O$ constant is independent of $\lambda,Q'$). Combining these facts, we deduce that if $M_{n,\lambda,j,Q'}$ denotes the number of cubes $Q\subset A_n^{m}$ such that $\pi_j(Q)=Q'$ and $Q\cap \zero(P_\lambda)\cap Q_0\neq\varnothing$, then
\begin{equation}\label{eq:M_nTbound}
M_{n,\lambda,j,Q'}\le C'_3 2^{n((m-1)s-q)}\text{ for all }\lambda\in  U, n\in\N, j\in[m]\text{ and }Q'\in\mathcal{Q}^d_n\,.
\end{equation}

To finish the proof, we will show that the claimed conclusion holds on the positive probability event that \eqref{eq:N_nTbound} and \eqref{eq:M_nTbound}  
hold. Suppose then that $A'\subset A$ is a compact set such that $(A')^m\cap \zero(P_\lambda)\cap Q_0=\varnothing$ for some $\lambda\in  U$. Since $A'$ is compact, the same still holds if we replace $A'$ by the union of the cubes in $\mathcal{Q}_{n_1}^d$ hitting $A'$ provided $n_1$ is sufficiently large. By further enlarging $A'$ slightly, we may assume without loss of generality that $A'$ is the interior of a union of cubes $\mathcal{Q}_{n_1}^d$ for some $n_1\in\N$.

Fix $n\ge n_1$. By \eqref{eq:M_nTbound}, for each $Q'\subset A_n$, $Q'\in\mathcal{Q}^d_n$, there are at most $C'_4 2^{n((m-1)s-q)}$ cubes $Q\in\mathcal{Q}^{md}_n$ with $Q\subset A_{n}^m\cap Q_0$, $Q\cap \zero(P_\lambda)\neq\varnothing$, and $\pi_j(Q)=Q'$ for some $j\in[m]$. Combining this with \eqref{eq:N_nTbound}, it follows that $A_n\setminus A'$ contains at least
\[
C'_1 2^{n(ms-q)} (C'_4)^{-1}2^{-n((m-1)s-q)}\ge \varepsilon 2^{ns}
\]
cubes in $\mathcal{Q}_n^{d}$, where $\varepsilon=C'_1/C'_4$. 

Finally, the definition of $\nu$ as the weak limit of $p^{-n} \leb|_{A_n}$ implies (since we took $A'$ to be open) that $\nu(A\setminus A')\ge\e$, as desired.
\end{proof}

Given an open set $U\subset\Lambda$, we will say that the \emph{$\delta$-Szemer\'{e}di condition} for $ U$ holds if for any Borel $A'\subset\R^d$ with $\nu(A')>\delta\nu(A)$ there exist $r>0$ and $t\in\R^d$ such that for every $\lambda\in U$,
\[
\zero(P_\lambda) \cap (r A'+t)^m \setminus\Delta \neq \varnothing.
\]
This is an event depending on the realization of the fractal percolation process.  Since $\nu$ is a Radon measure, `Borel set' may be replaced by `compact set' without changing the definition. Hence, Proposition \ref{prop:counting-szemeredi} implies that for each $\lambda_0\in\Lambda$ there is a positive probability that the $\delta$-Szemer\'{e}di condition holds in a neighbourhood of $\lambda_0$, if $\delta$ is close enough to $1$ (indeed, we can even take $r=1,t=0$). In the next step, we upgrade ``positive probability'' to ``full probability'' at the price of changing the value of $\delta$.
\begin{lemma} \label{lem:szemeredi-0-1-law}
If there is $\delta<1$ such that
\[
\PP(\delta\text{-Szemer\'{e}di condition holds  for } U)>0\,,
\]
then there is $\delta'<1$ such that the $\delta'$-Szemer\'{e}di condition for $ U$ holds almost surely.
\end{lemma}

\begin{proof}
Denote
\[
\eta=\PP(\delta\text{-Szemer\'{e}di condition holds  for } U)\,.
\]
Let $X_n$ be the number of those $Q\subset A_n$, $Q\in\mathcal{Q}^d_n$ such that the fractal percolation measure $\nu_Q$ induced on $Q$ satisfies the $\delta$-Szemer\'{e}di condition for $ U$. Fix $\varepsilon>0$ such that $p_\varepsilon<\eta/8$, where $p_\varepsilon=\PP(\|\nu\|<\varepsilon)$. Denote by $Z_n$ the number of those $Q\in\mathcal{Q}^{d}_n$ with $\nu(Q)<\varepsilon 2^{-ns}$. We claim that almost surely the estimates
\begin{align*}
X_n\ge\frac{\eta}{2}N_n\,,\quad Z_n\le \frac{\eta}{4}N_n\,\,\text{ and }\,\,
\|\nu\|\le 2 N_n 2^{-ns}
\end{align*}
hold for all large $n$.  
Indeed, recall from \eqref{eq:many-surviving-squares} that there is $c_1>0$ such that
\begin{equation} \label{eq:many-cubes}
\PP(N_n\le c_1 n)<(1-c_1)^n.
\end{equation}
Conditioned on $\mathcal{B}_n$,  Hoeffding's inequality yields
\[
\PP\left(\left(X_n\ge\frac{\eta}{2}N_n\right) \cap (N_n\ge c_1 n)\right) \le \exp(-\Omega_\eta(n)),
\]
so Borel-Cantelli and Lemma \eqref{eq:many-cubes} ensure that a.s. $X_n \ge \tfrac{\eta}{2}N_n$ for all large enough $n$. The rest of the claim follows in a similar way.

Now suppose $A'\subset A$ is a compact set such that $\nu(A'\cap Q)\le\delta\nu(Q)$ for $\tfrac{\eta}{2} N_n$ cubes $Q\in\mathcal{Q}^d_n$ with $Q\subset A_n$. Then at least $\tfrac{\eta}{4}N_n$ of these satisfy $\nu(Q)\ge \varepsilon 2^{-ns}$, so that
\[
\nu(A\setminus A')\ge (1-\delta)\tfrac{\eta}{4}\varepsilon 2^{-ns}N_n\ge \tfrac{(1-\delta)\eta\varepsilon}{8}\nu(A).
\]
In other words, if $\nu(A')>\delta'\nu(A)$, where $\delta'=1-\tfrac{(1-\delta)\eta\varepsilon}{8}$ then $\nu(A'\cap Q)>\delta\nu(Q)$ for at least one $Q\in\mathcal{Q}^d_n$ such that $\nu_Q$ satisfies the $\delta$-Szemer\'{e}di condition. By definition, the $\delta$-Szemer\'{e}di condition for $ U$ is invariant under homothetic changes of coordinates. We thus conclude that the $\delta'$-Szemer\'{e}di condition for $ U$ holds almost surely.
\end{proof}

\begin{proof}[Proof of Theorem \ref{thm:relative-S-abstract}]
By Proposition \ref{prop:counting-szemeredi}, for each $\lambda\in \Lambda$ there exist a neighbourhood $ U$ of $\lambda$ and $\delta<1$ such that
\[
\PP(\delta\text{-Szemer\'{e}di condition holds  for } U) > 0.
\]
(By the regularity of $\nu$, we may assume $A'$ is compact.) By Lemma \ref{lem:szemeredi-0-1-law}, there exists $\delta'<1$ such that
\[
\PP(\delta'\text{-Szemer\'{e}di condition holds  for } U) = 1.
\]
In particular, a.s. for all surviving cubes $Q$, the restricted process $\nu_Q$ satisfies the $\delta'$-Szemer\'{e}di condition for $ U$. Fix a realization such that this holds,  and let $A'\subset A$ be a measurable set with $\nu(A')>0$. By a weak version of the Lebesgue density point theorem, there is a cube $Q$ (depending on $A'$) such $\nu(A'\cap Q)>\delta' \nu(Q)$. Hence, if $Q=r[0,1]^d+t$, then for each $\lambda\in U$, the set $A'$ contains points $x_1,\ldots,x_m$ such that $P_\lambda(r x_1+t,\ldots,r x_m+t)=0$. This proves the first claim in  Theorem \ref{thm:relative-S-abstract}.

The latter claim in Theorem \ref{thm:relative-S-abstract} follows by covering the parameter space $\Lambda$ by countably many such sets $U$.
\end{proof}

\bibliographystyle{plain}
\bibliography{patterns}

\begin{thebibliography}{10}

\bibitem{Carnovale15}
Marc Carnovale.
\newblock A relative {R}oth theorem in dense subsets of sparse pseudorandom
  fractals.
\newblock Preprint, available at
  \texttt{https://people.math.osu.edu/carnovale.2/Research/RFR.pdf}, 2015.

\bibitem{Chanetal16}
Vincent Chan, Izabella {\L}aba, and Malabika Pramanik.
\newblock Finite configurations in sparse sets.
\newblock {\em J. Anal. Math.}, 128:289--335, 2016.

\bibitem{CKLS14}
Changhao Chen, Henna Koivusalo, Bing Li, and Ville Suomala.
\newblock Projections of random covering sets.
\newblock {\em J. Fractal Geom.}, 1(4):449--467, 2014.

\bibitem{ConlonGowers16}
D.~Conlon and W.~T. Gowers.
\newblock Combinatorial theorems in sparse random sets.
\newblock {\em Ann. of Math. (2)}, 184(2):367--454, 2016.

\bibitem{ConlonEtAl15}
David Conlon, Jacob Fox, and Yufei Zhao.
\newblock A relative {S}zemer\'edi theorem.
\newblock {\em Geom. Funct. Anal.}, 25(3):733--762, 2015.

\bibitem{Daviesetal59}
Roy~O. Davies, J.~M. Marstrand, and S.~J. Taylor.
\newblock On the intersections of transforms of linear sets.
\newblock {\em Colloq. Math.}, 7:237--243, 1959/1960.

\bibitem{Erdogan05}
M.~Burak Erdo{\~g}an.
\newblock A bilinear {F}ourier extension theorem and applications to the
  distance set problem.
\newblock {\em Int. Math. Res. Not.}, (23):1411--1425, 2005.

\bibitem{Falconer85}
K.~J. Falconer.
\newblock On the {H}ausdorff dimensions of distance sets.
\newblock {\em Mathematika}, 32(2):206--212 (1986), 1985.

\bibitem{Falconer92}
K.~J. Falconer.
\newblock On a problem of {E}rd{\H o}s on fractal combinatorial geometry.
\newblock {\em J. Combin. Theory Ser. A}, 59(1):142--148, 1992.

\bibitem{Falconer03}
Kenneth Falconer.
\newblock {\em Fractal geometry}.
\newblock John Wiley \& Sons Inc., Hoboken, NJ, second edition, 2003.
\newblock Mathematical foundations and applications.

\bibitem{FalconerJin16}
Kenneth Falconer and Xiong Jin.
\newblock H\"{o}lder continuity of the liouville quantum gravity measure.
\newblock Preprint, available at \texttt{http://arxiv.org/1601.00556}, 2016.

\bibitem{GGIP15}
Loukas Grafakos, Allan Greenleaf, Alex Iosevich, and Eyvindur Palsson.
\newblock Multilinear generalized {R}adon transforms and point configurations.
\newblock {\em Forum Math.}, 27(4):2323--2360, 2015.

\bibitem{Greentao08}
Ben Green and Terence Tao.
\newblock The primes contain arbitrarily long arithmetic progressions.
\newblock {\em Ann. of Math. (2)}, 167(2):481--547, 2008.

\bibitem{GILP15}
Allan Greenleaf, Alex Iosevich, Bochen Liu, and Eyvindur Palsson.
\newblock A group-theoretic viewpoint on {E}rd{\H o}s-{F}alconer problems and
  the {M}attila integral.
\newblock {\em Rev. Mat. Iberoam.}, 31(3):799--810, 2015.

\bibitem{GIM15}
Allan Greenleaf, Alex Iosevich, and Mihalis Mourgoglou.
\newblock On volumes determined by subsets of {E}uclidean space.
\newblock {\em Forum Math.}, 27(1):635--646, 2015.

\bibitem{HKKMMMS13}
Viktor Harangi, Tam\'as Keleti, Gergely Kiss, P\'eter Maga, Andr\'as M\'ath\'e,
  Pertti Mattila, and Bal\'azs Strenner.
\newblock How large dimension guarantees a given angle?
\newblock {\em Monatsh. Math.}, 171(2):169--187, 2013.

\bibitem{Henriotetal16}
Kevin Henriot, Izabella {\L}aba, and Malabika Pramanik.
\newblock On polynomial configurations in fractal sets.
\newblock {\em Anal. PDE}, 9(5):1153--1184, 2016.

\bibitem{IosevichLiu16}
Alex Iosevich and Bochen Liu.
\newblock Equilateral triangles in subsets of $\mathbb{R}^d$ of large
  {H}ausdorff dimension.
\newblock {\em Israel J. Math.}
\newblock to appear. Preprint available at \texttt{http:/arxiv.org/1603.01907}.

\bibitem{Janson04}
Svante Janson.
\newblock Large deviations for sums of partly dependent random variables.
\newblock {\em Random Structures Algorithms}, 24(3):234--248, 2004.

\bibitem{Kahane87}
Jean-Pierre Kahane.
\newblock Positive martingales and random measures.
\newblock {\em Chinese Ann. Math. Ser. B}, 8(1):1--12, 1987.
\newblock A Chinese summary appears in Chinese Ann. Math. Ser. A {{\bf{8}}}
  (1987), no. 1, 136.

\bibitem{Keleti88}
Tam{\'a}s Keleti.
\newblock A 1-dimensional subset of the reals that intersects each of its
  translates in at most a single point.
\newblock {\em Real Anal. Exchange}, 24(2):843--844, 1998/99.

\bibitem{KrantzParks2008}
Steven~G. Krantz and Harold~R. Parks.
\newblock {\em Geometric integration theory}.
\newblock Cornerstones. Birkh\"auser Boston, Inc., Boston, MA, 2008.

\bibitem{LabaPramanik09}
Izabella {\L}aba and Malabika Pramanik.
\newblock Arithmetic progressions in sets of fractional dimension.
\newblock {\em Geom. Funct. Anal.}, 19(2):429--456, 2009.

\bibitem{Liu01}
Quansheng Liu.
\newblock Local dimensions of the branching measure on a {G}alton-{W}atson
  tree.
\newblock {\em Ann. Inst. H. Poincar\'e Probab. Statist.}, 37(2):195--222,
  2001.

\bibitem{LyonsPeres17}
Russell Lyons and Yuval Peres.
\newblock {\em Probability on Trees and Networks}.
\newblock Cambridge Series in Statistical and Probabilistic Mathematics.
  Cambridge University Press, 2017.

\bibitem{Maga10}
P\'eter Maga.
\newblock Full dimensional sets without given patterns.
\newblock {\em Real Anal. Exchange}, 36(1):79--90, 2010/11.

\bibitem{Mathe12}
Andr\'as M\'ath\'e.
\newblock Sets of large dimension not containing polynomial configurations.
\newblock {\em Adv. Math.}
\newblock to appear. Preprint available at
  \texttt{http://arxiv.org/1201.0548},.

\bibitem{Mattila95}
Pertti Mattila.
\newblock {\em Geometry of sets and measures in {E}uclidean spaces}, volume~44
  of {\em Cambridge Studies in Advanced Mathematics}.
\newblock Cambridge University Press, Cambridge, 1995.
\newblock Fractals and rectifiability.

\bibitem{MolterYavicoli16}
Ursula Molter and Alexia Yavicoli.
\newblock Small sets containing any pattern.
\newblock Preprint, available at \texttt{http://arxiv.org/1610.03804}, 2016.

\bibitem{PeresRams16}
Yuval Peres and Micha{\l} Rams.
\newblock Projections of the natural measure for percolation fractals.
\newblock {\em Israel J. Math.}, 214(2):539--552, 2016.

\bibitem{RamsSimon14}
Micha{\l} Rams and K{\'a}roly Simon.
\newblock The {D}imension of {P}rojections of {F}ractal {P}ercolations.
\newblock {\em J. Stat. Phys.}, 154(3):633--655, 2014.

\bibitem{Roth53}
K.~F. Roth.
\newblock On certain sets of integers.
\newblock {\em J. London Math. Soc.}, 28:104--109, 1953.

\bibitem{Schacht16}
Mathias Schacht.
\newblock Extremal results for random discrete structures.
\newblock {\em Ann. of Math. (2)}, 184(2):333--365, 2016.

\bibitem{Shmerkin16}
Pablo Shmerkin.
\newblock {S}alem sets with no arithmetic progressions.
\newblock {\em International Mathematics Research Notices}.
\newblock to appear. Published online
  \texttt{https://doi.org/10.1093/imrn/rnr085}.

\bibitem{ShmerkinSuomala16}
Pablo Shmerkin and Ville Suomala.
\newblock A class of random measures, with applications.
\newblock In {\em Proceedings of the FARF III conference}.
\newblock to appear. Preprint available at
  \texttt{https://arxiv.org/abs/1603.08156}.

\bibitem{ShmerkinSuomala14}
Pablo Shmerkin and Ville Suomala.
\newblock Spatially independent martingales, intersections, and applications.
\newblock {\em Mem. Amer. Math. Soc.}
\newblock to appear. Preprint available at
  \texttt{http://arxiv.org/abs/1409.6707}.

\bibitem{ShmerkinSuomala15}
Pablo Shmerkin and Ville Suomala.
\newblock Sets which are not tube null and intersection properties of random
  measures.
\newblock {\em J. Lond. Math. Soc. (2)}, 91(2):405--422, 2015.

\bibitem{Szemeredi75}
E.~Szemer{\'e}di.
\newblock On sets of integers containing no {$k$} elements in arithmetic
  progression.
\newblock {\em Acta Arith.}, 27:199--245, 1975.
\newblock Collection of articles in memory of Juri{\u\i} Vladimirovi{\v{c}}
  Linnik.

\bibitem{Wolff99}
Thomas Wolff.
\newblock Decay of circular means of {F}ourier transforms of measures.
\newblock {\em Internat. Math. Res. Notices}, (10):547--567, 1999.

\end{thebibliography}

\end{document}